\documentclass[10pt]{amsart}
\usepackage{amsmath}
\usepackage{amsthm}
\usepackage{amssymb}
\usepackage{amscd}
\usepackage{verbatim}
\usepackage[plainpages=false,colorlinks,hyperindex,pdfpagemode=None,bookmarksopen,linkcolor=red,citecolor=blue,urlcolor=blue]{hyperref}
\usepackage[curve]{xypic}
\usepackage{pdflscape}

\swapnumbers
\theoremstyle{plain}
\newtheorem{theorem}[subsubsection]{Theorem}
\newtheorem{axiom}[subsubsection]{Axiom}
\newtheorem{statement}[subsubsection]{Statement}

\newtheorem*{theorem*}{Theorem}
\newtheorem{proposition}[subsubsection]{Proposition}
\newtheorem{lemma}[subsubsection]{Lemma}
\newtheorem*{lemma*}{Lemma}
\newtheorem{corollary}[subsubsection]{Corollary}

\theoremstyle{definition}
\newtheorem{definition}[subsubsection]{Definition}
\theoremstyle{remark}
\newtheorem*{remark}{Remark}
\newtheorem*{remarks}{Remarks}
\newtheorem{example}[subsubsection]{Example}

\renewcommand{\descriptionlabel}[1]%
         {\hspace{\labelsep}\normalfont{#1}}

\renewcommand{\AA}{\mathbf{A}}
\newcommand{\BB}{\mathbf{B}}
\newcommand{\CC}{\mathbb{C}}
\newcommand{\DD}{\mathbf{D}}
\newcommand{\FF}{\mathbb{F}}
\newcommand{\Ff}{\mathcal{F}}
\newcommand{\GG}{\mathbf{G}}
\newcommand{\HH}{\mathbf{H}}
\newcommand{\LL}{\mathbf{L}}
\newcommand{\MM}{\mathbf{M}}

\newcommand{\PP}{\mathbf{P}}
\newcommand{\RR}{\mathbb{R}}
\renewcommand{\SS}{\mathbf{S}}
\newcommand{\TT}{\mathbf{T}}
\newcommand{\UU}{\mathbf{U}}
\newcommand{\VV}{\mathbf{V}}
\newcommand{\XX}{\mathbf{X}}
\newcommand{\YY}{\mathbf{Y}}
\newcommand{\ZZ}{\mathbf{Z}}
\newcommand{\QQ}{\mathbb{Q}}
\newcommand{\Ad}{\mathbb{A}_F}

\newcommand{\cind}{\operatorname{c-Ind}}
\newcommand{\Ind}{\operatorname{Ind}}
\newcommand{\Hom}{\operatorname{Hom}}
\newcommand{\End}{\operatorname{End}}
\newcommand{\Aut}{\operatorname{Aut}}
\newcommand{\varchi}{\mathcal{X}}
\newcommand{\Gm}{\mathbb{G}_m}
\newcommand{\Ga}{\mathbb{G}_a}
\newcommand{\supp}{\operatorname{supp}}
\newcommand{\GGL}{\operatorname{\mathbf{GL}}}
\newcommand{\GL}{\operatorname{GL}}
\newcommand{\PPGL}{\operatorname{\mathbf{\mathbf{PGL}}}}
\newcommand{\PGL}{\operatorname{PGL}}
\newcommand{\SSL}{\operatorname{\mathbf{SL}}}
\newcommand{\SL}{\operatorname{SL}}
\newcommand{\Sp}{\operatorname{Sp}}
\newcommand{\SSp}{\operatorname{\mathbf{Sp}}}
\newcommand{\SO}{{\operatorname{SO}}}
\newcommand{\Spin}{\operatorname{Spin}}
\newcommand{\SSpin}{\operatorname{\mathbf{Spin}}}
\newcommand{\SSO}{\operatorname{\mathbf{SO}}}
\newcommand{\geom}{{\operatorname{geom}}}

\newcommand{\tr}{\operatorname{tr}}

\newcommand{\spec}{\operatorname{spec}}
\newcommand{\Vol}{\operatorname{Vol}}
\newcommand{\diag}{{\operatorname{diag}}}
\newcommand{\ev}{\operatorname{ev}}

\newcommand{\temp}{{\operatorname{temp}}}
\newcommand{\Tam}{{\operatorname{Tam}}}
\newcommand{\rk}{{\operatorname{rk}}}
\newcommand{\inv}{\operatorname{inv}}
\newcommand{\Adj}{\operatorname{Ad}}

\numberwithin{equation}{section}

\begin{document}
\setcounter{tocdepth}{1}
\title{Spherical functions on spherical varieties}
\author{Yiannis Sakellaridis}
\date{}
\keywords{Spherical varieties, distinguished representations, spherical functions, Plancherel formula, period integrals}
\email{sakellar@rutgers.edu}

\address{Department of Mathematics and Computer Science, Rutgers University at Newark, 101 Warren Street, Smith Hall 216, Newark, NJ 07102, USA.}

\subjclass[2000]{22E50 (Primary); 11S40, 43A90, 14M17 (Secondary)}

\begin{abstract} Let $\XX=\HH\backslash\GG$ be a homogeneous spherical variety for a split reductive group $\GG$ over the integers $\mathfrak o$ of a $p$-adic field $k$, and $K=\GG(\mathfrak o)$ a hyperspecial maximal compact subgroup of $G=\GG(k)$. We compute eigenfunctions (``spherical functions'') on $X=\XX(k)$ under the action of the unramified (or spherical) Hecke algebra of $G$, generalizing many classical results of ``Casselman-Shalika'' type. Under some additional assumptions on $\XX$ we also prove a variant of the formula which involves a certain quotient of $L$-values, and we present several applications such as: (1) a statement on ``good test vectors'' in the multiplicity-free case (namely, that an $H$-invariant functional on an irreducible unramified representation $\pi$ is non-zero on $\pi^K$), (2) the unramified Plancherel formula for $X$, including a formula for the ``Tamagawa measure'' of $\XX(\mathfrak o)$, and (3) a computation of the most continuous part of $\HH$-period integrals of 
principal Eisenstein series.
\end{abstract}

\maketitle
\tableofcontents
\newpage

\section{Introduction}\label{secintro} 

\subsection{The problem} \label{ssfirst}
Let $\GG$ be a split reductive group over the ring of integers $\mathfrak o$ of a local non-archimedean field $k$ in characteristic zero (with prime ideal $\mathfrak p$ and residual degree $q$, a power of the prime number $p$) and let $\XX$ be a homogeneous spherical scheme for $\GG$ over $\mathfrak o$ (i.e.\ a homogeneous $\GG$-scheme on which the Borel subgroup $\BB$ -- which we fix -- has an open orbit; this includes, but is not limited to, symmetric spaces). We make throughout certain assumptions on $\XX$ (cf.\ \S \ref{ssnotation}), including that $\XX$ is quasi-affine, which cause no serious harm to generality. Then $\XX=\HH\backslash\GG$, where $\HH$ is the stabilizer of an $\mathfrak o$-point $x_0$ in the open Borel orbit ($\HH\BB$ is open in $\GG$). We denote by regular font the corresponding sets of $k$-points of each of the groups. We also denote $K=\GG(\mathfrak o)$ -- it is a (hyperspecial) maximal compact subgroup. To avoid complicated formulas, we will assume that $k$ is unramified over the 
field $\QQ_p$; the modifications 
needed to remove this assumption are trivial. 

We consider the right regular representation of $G$ on $V=C^\infty(X)$ (if $G$ has a unique orbit on $X$, this is just the induced representation $\Ind_H^G(1)$). The unramified  Hecke algebra $\mathcal H(G,K)$ (in the literature, the term ``spherical'' is often used for ``unramified'', but to avoid confusion I will reserve this term for the notion of spherical varieties) acts on $V$, and the goal of this article is to compute an explicit formula for its eigenvectors (with nonzero eigencharacter -- this will be implicit throughout) expressed in terms of the geometry of $\XX$. More generally, if $\HH$ contains the unipotent radical $\UU_P$ of a proper parabolic, $\Lambda: \UU_P\to \Ga$ is a homomorphism which is fixed under $\HH$-conjugation and $\psi:k\to\CC^\times$ is a character, then we can also consider the space $V=C^\infty(X, \mathcal L_\Psi)$, where $\Psi=\psi\circ\Lambda$ considered as a character of $H$ and $\mathcal L_\Psi$ is the 
complex line bundle defined by it. (This space is just $\Ind_H^G(\Psi)$ if $G$ has unique orbit on $X$.) 
For an explanation of how this case can be understood geometrically, cf.\ \cite[\S 5.5]{Sa2}. Though most of the results carry over verbatim to this ``twisted'' case, for notational simplicity we do not discuss it in the introduction.

The problem at hand is of both harmonic-analytic and arithmetic interest, and there is a long history of particular examples which have been computed \cite[to mention just a few]{MD71, C, CS, HS1, HS2, Hi1, Hi2, KMS, Of, Sa1}. In the first part of this paper, we compute a very general formula, covering all previously established cases (when the group $\GG$ is split) and many more. Besides its case-specific arithmetic applications, which have been the motivation for most of the literature, the computation of such a formula is relevant to the $\HH$-period integrals of principal Eisenstein series (which we undertake in \S \ref{secEisenstein}), and conjecturally also to the $\HH$-period integrals of other automorphic forms, via the local Plancherel formula (which we develop in \S \ref{secPlancherel}). These applications, presented in the third part of the paper, require an improved version of the general formula, which we develop in the second part and which, in particular, involves a certain quotient of local 
$L$-values of the unramified representation in question. A table with many examples of spherical varieties and the related quotients of $L$-values appears in the Appendix.

Before we proceed to a more detailed description of the results, let us put this work in a more general context, under the perspective of automorphic forms. The study of period integrals of automorphic forms (a major stream of which is related to the relative trace formula of Jacquet -- see \cite{Lapid} for a presentation) has revealed relationships between the non-vanishing of certain period integrals of automorphic forms (typically, over spherical subgroups), on one hand, and functorial lifts, on the other, and also between the values of these period integrals and $L$-functions or special values thereof. Though no general theory or conjectures exist to describe these phenomena, it has recently started to become clear that a general and systematic approach should be possible. The ``dual group'' attached by Gaitsgory and Nadler \cite{GN}, in the context of the Geometric Langlands program, to any spherical variety $\XX$ was shown in \cite{Sa2} to be related to unramified representations in the spectrum of $X$ 
and is, more generally, conjectured in \cite{SV} to be describing $X$-distinguished representations, in a certain sense. Moreover, under certain assumptions period integrals are (roughly speaking) conjectured to be ``Eulerian'' with local Euler factors equal to ``the $H$-invariant functionals which appear in the local Plancherel formula for $X$''. The last step is to relate these $H$-invariant functionals to special values of $L$-functions, and this is part of what we accomplish here, when the local group is split and the representation is unramified.

\subsection{The general formula}

To formulate the general formula, we will for simplicity assume in the introduction that $B=\BB(k)$ has a \emph{unique} open orbit on $X$. This includes, in particular, the \emph{multiplicity-free} case, i.e.\ the case when the eigenspaces for the unramified Hecke algebra are one-dimensional. Then, under mild assumptions, a generalized Cartan decomposition holds \cite{Sa3}: orbits of $K$ on $X$ are in bijection $x_{\check\lambda}\leftrightarrow \check\lambda$ with anti-dominant, in a suitable sense, cocharacters into $\AA_X$, where $\AA_X$ is a quotient of a suitable maximal torus $\AA$ of $\GG$ identified with the $\AA$-orbit of a chosen point $x_0\in\XX(\mathfrak o)$.

Let $A^*$ denote the ``Langlands dual'' complex torus of $\AA$ (the torus of unramified complex characters of $A$) and let $A_X^*$ denote the ``Langlands dual'' of $\AA_X$. Naturally (under the assumption of a unique open $B$-orbit) $A_X^*\subset A^*$. There is a surjective map $A^*\to\spec_M\mathcal H(G,K)$ (where by $\spec_M$ we denote the maximal spectrum of a ring) with generic fiber finite of order $|W|$, where $W$ is the Weyl group of $G$. The support of $V^K$ as a module for $\mathcal H(G,K)$ coincides with the image of a translate $\delta_{(X)}^\frac{1}{2} A_X^*$ of $A_X^*$, see \cite{Sa2}. The torus $A_X^*$ (and its translate $\delta_{(X)}^\frac{1}{2} A_X^*$) carries a natural action of a finite reflection group $W_X\subset W$, called the ``little Weyl group'' of $X$, and the map $\delta_{(X)}^\frac{1}{2} A_X^*\to \spec_M\mathcal H(G,K)$ factors through the topological quotient $\delta_{(X)}^\frac{1}{2} A_X^*/W_X$. We can now formulate our general formula:

\begin{theorem}[cf.\ Theorems \ref{maintheorem}, \ref{fetheorem}, see also (\ref{formulawithBw})] \label{intromaintheorem}
For an open dense subset of points $\eta\in\spec_M \mathcal H(G,K)$ in the support of $V^K$, the corresponding eigenspace in $V^K$ admits a basis consisting of the functions:
\begin{equation}
\Omega_{\delta_{(X)}^\frac{1}{2} \chi}(x_{\check\lambda}) = {e^{-{\check\lambda}}(\delta_{P(X)}^\frac{1}{2})}  \sum_{w\in W_X}  B_w(\chi) e^{{\check\lambda}}(^w\chi)
\end{equation}
where: $\delta_{(X)}^\frac{1}{2}\chi\in \delta_{(X)}^\frac{1}{2}A_X^*$ ranges over a set of representatives for the elements of $\delta_{(X)}^\frac{1}{2}A_X^*/W_X$ which map to $\eta\in\spec_M\mathcal H(G,K)$ and the coefficients $B_w(\chi)$ are certain cocycles: $W_X\to \CC(A_X^*)$ for the computation of which we give an explicit algorithm in terms of the geometry of $\XX$. For the definition of $\delta_{(X)}, \delta_{P(X)}$ cf.\ \S \ref{ssnotation}, \ref{ssinvariants}.
\end{theorem}

The cocycles $B_w$ have the following conceptual meaning: It was proven in \cite{Sa2} that certain natural morphisms: $S_\chi:C_c^\infty(X)\to I(\chi)$, defined using Mackey theory, satisfy proportionality relations: $T_w \circ S_\chi \sim S_{^w\chi}$ for all $w\in W_X$, where $I(\chi)=\Ind_B^G(\chi\delta^\frac{1}{2})$ is an unramified principal series representation and $T_w$ are intertwining operators between principal series. Then the $B_w$ are a standard term times the coefficients of proportionality, when $T_w$ are normalized in a ``good'' way. This ``good'' normalization comes from the theory of equivariant Fourier transform on the basic affine space $U\backslash G$ \cite{BK}, which we review in \S \ref{ssSchwartz}. Use of the equivariant Fourier transform allows us to avoid complicated formulas and mysterious cancellations which are abundant in the relevant literature.

The method employed in the first part of the paper is based on the basic idea of Casselman and Shalika \cite{C, CS}, namely that instead of computing an unramified eigenfunction one should first compute an Iwahori-invariant function coming from a vector of small support in the principal series, and then use the fact that the space of Iwahori-invariant vectors has dimension equal to the number of intertwining operators between principal series. The representation-theoretic results of \cite{Sa2} allow us to reduce the computation to the case of $\SSL_2$, and there is only a small number of possibilities to consider there. 

\subsection{The formula with $L$-values}
Because of its inductive, algorithmic definition, the general formula may be difficult to use in applications. Therefore, in the second part of the paper we transform it (under additional assumptions, which are in particular satisfied by many \emph{affine} homogeneous spherical varieties) to a more useful one, expressed in terms of the ``dual group'' $\check G_X$ of $\XX$ -- more precisely the root datum of this group (cf.\ \S \ref{ssrootsystem}). This should be, conjecturally, isogenous to the subgroup of $\check G$ attached to $\XX$ by Gaitsgory and Nadler \cite{GN}.

\begin{theorem}[cf.\ Theorem \ref{improvedformula}]\label{introimprovedformula}
(Under additional assumptions.) We have:
\begin{equation}
 \frac{\Omega_{\delta_{(X)}^\frac{1}{2}\chi} (x_{\check\lambda})}{\beta({\tilde\chi})} =  \delta_{P(X)}^{-\frac{1}{2}}(x_{\check\lambda}) \prod_{\Theta^+} (1- \sigma_{\check\theta} q^{-r_{\check\theta}}T_{\check\theta}) s_{{\check\lambda}}(\chi)
\end{equation}
where $s_{{\check\lambda}} = \frac{\sum_{W_X} (-1)^w e^{\check\rho - w\check\rho + w{\check\lambda}}}{\prod_{\check\gamma>0} (1-e^{\check\gamma})}$ is the Schur polynomial \emph{indexed by lowest weight} (that is, if ${\check\lambda}$ is anti-dominant then $s_{\check\lambda}$ is the character of the irreducible representation of $\check G_X$ with lowest weight ${\check\lambda}$) and $T_{\check\theta}$ denotes the formal operator: $T_{\check\theta} s_{{\check\lambda}} = s_{{\check\theta}+{\check\lambda}}$.

Here 
\begin{equation}
 \beta({\tilde\chi}):= \frac{\prod_{\check\gamma\in\check\Phi_X^+} (1-e^{\check\gamma})}{\prod_{({\check\theta}, \sigma_{\check\theta}, r_{\check\theta})\in\Theta^+}(1- \sigma_{\check\theta} q^{-r_{\check\theta}} e^{\check\theta})}(\chi),
\end{equation}
and the triples $(\check\theta, \sigma_{\check\theta}, r_{\check\theta})\in\Theta^+$ consisting of a weight of $A_X^*$, a sign and a positive half-integer are given explicitly in terms of the geometry of $\XX$ (\S \ref{sscolors}). 

If $\XX$ is affine then, in particular:
\begin{equation}
\Omega_{\delta_{(X)}^\frac{1}{2}\chi} (x_0) = c \beta(\chi)
\end{equation}
with the constant:
\begin{equation}
 c = \beta(\delta_{P(X)}^\frac{1}{2})^{-1}.
\end{equation}
\end{theorem}

Notice that the value of $\Omega_{\delta_{(X)}^\frac{1}{2}\chi}$ at $x_0$ is equal to ``half'' a quotient of local $L$-values for $\check G_X$ -- more precisely $L_X:=c^2\beta(\chi)\beta(\chi^{-1})$ is a quotient of $L$-values. As we will explain below, this $L$-value is related to the Plancherel measure for the unramified spectrum of $X$, and should (conjecturally) be related to $\HH$-period integrals of automorphic forms. It would certainly be desirable to have a more natural, geometric understanding of this quotient of $L$-values, rather than the combinatorial one which we provide here. 

The weights ${\check\theta}$ that appear in this formula are related to the \emph{colors} of the spherical variety, in other words the irreducible $\BB$-invariant divisors on $\XX$. Each of them induces a $\BB$-invariant valuation on $k(\XX)$ which, by restriction to $\BB$-eigenfunctions, defines a weight of $A_X^*$. Roughly speaking, in the affine case the weights $\check\theta$ contained in the above formula are those $W_X$-translates of the weights defined by colors which are contained in the positive span of the weights defined by colors.

A table with many examples of affine homogeneous spherical varieties and their related $L$-values appears in the Appendix.

\subsection{Representation-theoretic results}
A benefit of the last formula is that it allows us to draw certain representation-theoretic conclusions, in combination with the results of \cite{Sa2}. More precisely:

\begin{theorem}[The Hecke module of unramified functions, cf.\ Theorem \ref{modulestructure}.]
(Under the assumptions of the Theorem \ref{introimprovedformula}, with $\XX$ affine.)  There is an isomorphism: $C_c^\infty(X)^K \simeq \CC[\delta_{(X)}^\frac{1}{2}A_X^*]^{W_X}$, compatible with the $\mathcal H(G,K)$-structure and the Satake isomorphism. Under this isomorphism the characteristic function of ${X(\mathfrak o)}$ is mapped to the constant $1$.
\end{theorem}

In the multiplicity-free case, i.e.\ when $\CC[\delta_{(X)}^\frac{1}{2}A_X^*]$ has fibers of dimension at most one over $A^*/W$, we deduce that a non-zero $H$-invariant functional on an irreducible unramified representation $\pi$ is non-zero on $\pi^K$, which has the following global corollary: if $\pi=\otimes_v'\pi_v$ is an irreducible representation of the points of a reductive group $\GG$ over the adele ring $\Ad$ of a global field $F$, and $\HH$ is a spherical subgroup over $F$ such that $\CC[\delta_{(X)}^\frac{1}{2}A_X^*]$ is multiplicity-free over $A^*/W$ and $(\XX)_{F_v}$ satisfies the assumptions of the above theorem for almost every place $v$ (where $\XX=\HH\backslash\GG$), then  $\Hom_{H_v}(\pi_v,1)\ne 0$ for every place $v$ implies $\Hom_{\HH(\Ad)}(\pi,1)\ne 0$.

\subsection{Unramified Plancherel formula}
Another application which we present is the Plancherel formula for the unramified part of the spectrum of $X$, i.e.\ for $L^2(X)^K$. More precisely, keeping the assumptions of Theorem \ref{introimprovedformula} with $\XX$ affine, we normalize the invariant measure on $X$ such that $\Vol(x_0J)=1$, where $J$ denotes the Iwahori subgroup of $G$. We denote by $A_X^{*,1}$ the maximal compact subgroup of $A_X^*$ (the subgroup of unitary characters.) Then the Plancherel formula restricted to $C_c^\infty(X)^K$ reads: 

\begin{theorem}[cf.\ Theorem \ref{Planchereltheorem}]
For every $\Phi\in C_c^\infty(X)^K$ we have:
\begin{equation}
 \Vert \Phi\Vert^2 = \frac{1}{Q\cdot|W_X|} \int_{A_X^{*,1}} \left|\left<\Phi, \Omega_{\delta_{(X)}^\frac{1}{2}\chi}\right>\right|^2  d\chi.
\end{equation}
where $d\chi$ is probability Haar measure on $A_X^{*,1}$ and $$Q=\frac{\Vol(K)}{\Vol(Jw_lJ)}=\prod_{\check\alpha\in\check\Phi^+} \frac{1-q^{-1-\left<\check\alpha,\rho\right>}}{1-q^{-\left<\check\alpha,\rho\right>}}.$$\end{theorem}

Equivalently, if the eigenfunctions $\Omega_{\delta_{(X)}^\frac{1}{2}\chi}$ were re-normalized to have value $1$ at $x_0$, then the corresponding Plancherel measure on $\delta_{(X)}^\frac{1}{2}A_X^{*,1}/W_X$ would be given by $Q^{-1}L_X(\chi) d\chi$.

This theorem also leads to a computation of $\Vol(\XX(\mathfrak o))$ (Theorem \ref{measuretheorem}). This volume is essentially (up to a factor $(1-q^{-1})^{\rk (A_X^*)}$) the local ``Tamagawa'' volume, by which we mean the volume with respect to an \emph{integral, residually non-vanishing} invariant volume form.

\subsection{Periods of automorphic forms}

Finally, as a direct application of our formula we compute in \S \ref{secEisenstein} the ``most continuous part'' of $H$-period integrals of principal Eisenstein series, when $\HH\backslash\GG$ is a spherical subgroup of a split group defined over a global field $F$ (and locally satisfying the assumptions of Theorem \ref{introimprovedformula}) and show that it is given, essentially, by the ``half $L$-value'' $L_X^\frac{1}{2}:=c\beta(\chi)$ of Theorem \ref{introimprovedformula}. More precisely, we consider a pseudo-Eisenstein series on $\GG(F)\backslash \GG(\Ad)$, which can be analyzed as an integral of Eisenstein series $E(f_\chi,g)$, and show (for simplicity: if the Eisenstein series are induced from absolute values of algebraic characters of $\BB(\Ad)$): 

\begin{theorem}[cf.\ Theorem \ref{period}]
 The period integral of:
\begin{equation}
 \sum_{\gamma \in {\BB}(F)\backslash \GG(F)} \Phi(\gamma g) = \int_{\exp(\kappa+i\mathfrak a_\RR^*)} E(f_\chi,g) d\chi
\end{equation}
over $\HH(F)\backslash\HH(\Ad)$ is equal to:
\begin{equation}\int_{\exp(\kappa+i\mathfrak a_{X,\RR}^*)} \left(L_X^{\frac{1}{2},S}(\chi)\right)^*   \sum_{\left[W/W_{(X)}\right]} \left(\tilde j_{w}^S(\chi) \prod_{v\in S} \Delta_{{^{w}\chi},v}^{Y,{\Tam}} (f_{{^{w}\chi},v})\right) d\chi
\end{equation}
plus terms which depend on the restriction of $f_\chi$, as a function of $\chi$, to a subvariety of smaller dimension. 
\end{theorem}

For the notation, see \S \ref{secEisenstein}.
The same $L$-values should show up in the $H$-period integral of cusp forms, according to an idea of A.\ Venkatesh which we formulate as an (almost) precise conjecture in \cite{SV}.

\subsection{Assumptions and notation} \label{ssnotation} All our schemes, subgroups etc.\ are over $\mathfrak o$ unless otherwise specified. However, when we talk about ``varieties'' we mean ``over $k$'', i.e.\ the fibers of the schemes over $\spec k$, unless otherwise specified. The group $\GG$ is assumed reductive, split and with connected fibers over $\spec \mathfrak o$. We also assume that the derived group of $\GG$ is simply connected (which causes no harm to generality, since we can always replace the action of a reductive group on a given variety by the action of a finite cover whose derived group is simply connected). We fix throughout a Borel subgroup $\BB$ with unipotent radical $\UU$ and a maximal split torus $\AA\subset\BB$. The choice of $\AA$ will be explained in \S\ref{torus}. The group $\AA(\mathfrak o)$ will also be denoted by $A_0$, and similarly $B_0=\BB(\mathfrak o)$, $U_0=\UU(\mathfrak o)$. The complex torus of unramified characters of $A$ is denoted by $A^*$. 

The Weyl group is denoted by $W$, by $W_P$ we denote the Weyl group of the Levi quotient of a standard parabolic $\PP$, roots are generally denoted by small Greek letters $\alpha,\beta$ etc.\ and the corresponding co-roots by $\check\alpha, \check\beta$ etc. Similarly, $\rho$ denotes half the sum of positive roots and $\check\rho$ denotes half the sum of positive co-roots. We will be using exponential notation for characters of tori, e.g.\ $e^{\alpha}$. The real part $\Re(\chi)$ of the character $\chi=e^\theta$ of a torus is, by definition, the real part of $\theta$. We denote by $\Phi$, $\Phi^+$ and $\Delta$ the sets of roots, positive roots, and simple positive roots, respectively (for our choice of Borel). The standard parabolic corresponding to a set of simple positive roots $\{\alpha,\beta,\dots\}$ will be denoted by $\PP_{\alpha\beta\dots}$, and its standard Levi by $\LL_{\alpha,\beta,\dots}$. The one-parameter additive subgroup corresponding to a root $\alpha$ will be denoted by $\UU_\alpha$. 

We denote by $\Gm$, $\Ga$ the multiplicative, resp.\ additive group and by $\mathbb A^n$ the $n$-dimensional affine space, sometimes with an index denoting chosen coordinates for this space. For a subgroup $\MM$ we will denote by $\MM'$ or $[\MM,\MM]$ its derived group, by $\mathcal Z(\MM)$ its center, by $\mathcal N(\MM)$ (or 
$\mathcal N_{\GG}(\MM)$, when the ambient group $\GG$ is not clear from the context) its normalizer, by $\mathcal R(\MM)$ its radical, by $\UU_M$ its unipotent radical and by $\mathfrak d_M$ the modular character of its normalizer (the quotient of a right- by a left-invariant volume form); that is:
$$ \mathfrak d_M: \mathcal N(\MM) \to \Gm.$$ We denote by $\delta_M$ the absolute value of $\mathfrak d_M$, and $\delta_B$ will simply be denoted by $\delta$. In \ref{ssinvariants} we will introduce a Levi subgroup $\LL(\XX)$; for notational simplicity, the modular character of $\LL(\XX)\cap \BB$ will be denoted by $\mathfrak d_{(X)}$ (resp.\ $\delta_{(X)}$ for its absolute value).
We denote $K=\GG(\mathfrak o)$, and $J=$ the standard Iwahori subgroup (the inverse image of $\BB(\FF_q)$ under the reduction map $K\to\GG(\FF_q)$). A representation $\pi$ of $G$ will be called ``unramified'' if $\pi^K\ne 0$.

We give ourselves a quasi-affine $\GG$-scheme $\XX$ over $\mathfrak o$ with an open $\BB_k$-orbit $\mathring \XX_k$ on $\XX_k$ such that: the complement $\mathring\XX$ of the closure in $\XX$ of the complement of $\mathring\XX_k$ is smooth and surjective over $\spec\mathfrak o$, and its special fiber is homogeneous under the special fiber of $\BB$. We assume that $X$ admits a $G$-eigenmeasure and $\mathring X$ admits a $B$-invariant measure. This can always be achieved by a trivial modification, see \cite[\S 3.8]{Sa2}. We also assume that $\XX$ satisfies Axioms \ref{integralityaxiom}, \ref{orbitstheorem} and \ref{Iwahori}, which is always the case at almost every place if $\XX$ and $\GG$ are defined over a global field (with $\GG$ split). The stabilizer of a point $\xi$ on a $\GG$-space will be denoted by $\GG_\xi$.

The quotient of a scheme $\VV$ (resp.\ a topological space $V$) by the action of a group scheme $\mathbf\Gamma$ (resp.\ a topological group $\Gamma$), whenever it exists (in the sense of geometric quotient for schemes, or topological quotient in the topological setting), will be denoted by $\VV/\mathbf\Gamma$ (resp.\ $V/\Gamma$). Thus, the reader should not be confused by expressions of the form $G/\Gamma$, where $G$ is a group and $\Gamma$ is another group acting on it, though not a subgroup of $G$.
We fix throughout a complex character $\psi$ of the additive group $k$, whose conductor is equal to the ring of integers $\mathfrak o$; this character will be used for our Fourier transforms, but also when we consider line bundles $\mathcal L_\Psi$ as described in \S \ref{ssfirst}. 

Finally, we will always work in the category of smooth representations of $p$-adic groups, except in the brief general discussion of $L^2(X)$ in section \ref{secPlancherel}. We denote the space of smooth (i.e.\ locally constant), compactly supported measures on the $k$-points of a smooth variety $\YY$ invariably by $C_c^\infty(Y)$ or $\mathcal S(Y)$; the space of smooth measures will be denoted by $M_c^\infty(Y)$. The characteristic function of an open subset $S\subset Y$ will be denoted by $1_S$.

\subsection{Acknowledgements}

I would like to thank Joseph Bernstein and Akshay Venkatesh for very useful and motivating discussions. I also thank the referee for a very diligent work that led to numerous little improvements and corrections. The first part of this work was conducted when I was supported by the Marie Curie Research Training Network in Arithmetic Algebraic Geometry, EU Contract MRTN-CT-2003-504917.


\part{The sum formula}

\section{Invariants and orbits}\label{secorbits}

\subsection{Invariants associated to spherical varieties} \label{ssinvariants} We start with the description of certain invariants associated to a spherical variety. For a variety $\YY$ with a $\BB$-action, we denote by $k(\YY)^{(\BB)}$ the multiplicative group of non-zero rational $\BB$-eigenfunctions on $\YY$ and by $\varchi(\YY)$ the group of $\BB$-eigencharacters on $k(\YY)^{(\BB)}$. If $\YY$ has a dense $\BB$-orbit, then we have a short exact sequence: $0\to k^\times \to k(\YY)^{(\BB)} \to \varchi(\YY) \to 0$. The \emph{rank} of a $\BB$-variety $\YY$ is, by definition, the rank of the abelian group $\varchi(\YY)$.

Our assumption that $\mathring\XX$ is smooth and surjective over $\mathfrak o$ implies that there is a point $x_0\in \mathring \XX(\mathfrak o)$; indeed, by the argument of \cite[Proposition 3.2.1]{Sa2} there is a point in $\XX(\FF_q)$, and by smoothness it can be lifted to an $\mathfrak o$-point.
We make once and for all a choice of a point $x_0\in \mathring\XX(\mathfrak o)$ and let $\HH$ denote its stabilizer; hence  $\XX=\HH\backslash\GG$ and $\HH\BB$ is open in $\GG$. 

Given a spherical variety $\XX$ with open Borel orbit $\mathring \XX$, the \emph{associated parabolic} is the standard parabolic $\PP(\XX):= \{ p\in \GG | \mathring\XX \cdot p = \mathring\XX\}$. We will denote by $\Delta(\XX)$ the set of simple roots to which it corresponds. It is known that the stabilizer $\HH$ of $x_0$ contains the derived group of a Levi subgroup $\LL(\XX)$ of $\PP(\XX)$ \cite[Th\'eor\`eme 3.4]{BLV}. To choose such a Levi subgroup, one can pick an element $f\in \mathfrak o[\XX]$ whose set-theoretic zero locus is the complement of $\mathring \XX$. Its differential $df$ is an element of the coadjoint representation, and its stabilizer is such a subgroup $\LL(\XX)$ (which for a symmetric variety is what is called a $\theta$-stable Levi of a minimal $\theta$-split parabolic subgroup, where $\theta$ is the involution defined by the choice of point $x_0$). 
We define
\label{torus} $\AA_X:= \LL(\XX)/(\LL(\XX)\cap\HH)$; it is a torus, but it can also be considered as a subset of $\XX$ by identifying it with the orbit of $\LL(\XX)$ through $x_0$. We also fix a maximal torus $\AA\subset\BB\cap\LL(\XX)$ (so that $\AA_X$ is a quotient of $\AA$ as well). Let $\Lambda_X$ be the coweight lattice of $\AA_X$; it can be naturally identified with $\Hom(\varchi(\XX), \mathbb Z)$. Let $\Lambda_X^+$ denote the monoid of $\GG$-invariant ($\mathbb Z$-valued, trivial on $k$) valuations on $k(\XX)$; it can be considered as a submonoid of $\Lambda_X$ by restriction to $k(\XX)^{(\BB)}$. (Indeed, no non-trivial $\GG$-invariant valuation vanishes on $k(\XX)^{(\BB)}$, cf.\ \cite[Corollary 1.8]{KnLV}.) Let $\mathcal Q = \Hom(\varchi(\XX),\mathbb Q)$, and let $\mathcal V\subset \mathcal Q$ be the cone spanned by $\Lambda_X^+$ in $\mathcal Q$. There is a natural quotient map: $\Hom (\varchi(\AA),\mathbb Q)\to \Hom(\varchi(\XX),\mathbb Q)$, and it is known that $\mathcal V$ contains the image of the 
\emph{
negative} Weyl chamber. In many cases (notably, symmetric varieties) it coincides with it \cite[\S 5]{KnLV}. We have a natural bijection: $\Lambda_X\simeq \AA_X(k)/\AA_X(\mathfrak o)$, induced from ${\check\lambda}\mapsto e^{\check\lambda}(\varpi)$. Elements of $A_X$ which are mapped to elements of $\Lambda_X^+$ under this bijection will be called ``$X$-anti-dominant'', and the set of those will be denoted by $A_X^+$.

We denote by $A_X^*$ the (complex) dual torus of $A_X$ -- equivalently, the torus of its unramified characters. The quotient map $\AA\to \AA_X$ induces a morphism $A_X^*\to A^*$ whose image we will denote by $\overline{A_X^*}$.

\subsection{Knop's action}
F.\ Knop has defined an action of the Weyl group of $\GG$ on the set of Borel orbits over $\bar k$ (cf.\ \cite{KnOrbits}). We review it briefly: If $\YY_{\bar k}$ is a $\BB_{\bar k}$-orbit on $\XX_{\bar k}$ and $\alpha$ is a simple root, then $(\YY\PP_\alpha/\mathcal R(\PP_\alpha))_{\bar k}$ is a homogeneous spherical variety for $(\PP_\alpha/\mathcal R(\PP_\alpha))_{\bar k}\simeq \PPGL_2$, hence has one of the following types over $\bar k$: Type $G$ ($\PPGL_2\backslash\PPGL_2$), type $U$ ($\SS\UU\backslash\PPGL_2$ where $\UU$ a maximal unipotent subgroup and $\SS\subset\mathcal N(\UU)$), type $T$ ($\TT\backslash\PPGL_2$, where $\TT$ is a non-trivial torus) or type $N$ ($\mathcal N(\TT)\backslash\PPGL_2$). We say, correspondingly, that the pair $(\YY_{\bar k},\alpha)$ is of type $G$, $U$, $T$ or $N$. Since we are working over a non-algebraically closed field, if $\YY_{\bar k}$ is defined over $k$ and $(\YY_{\bar k},\alpha)$ is of type $T$ then we will distinguish two sub-cases, called ``split'' and ``non-
split'', according to whether $\TT_k$ has the corresponding property. 

Knop defines an action of the Weyl group on the set of Borel orbits (over $\bar k$), characterized by the fact that for every orbit $\YY_{\bar k}$ the simple reflection $w_\alpha$ fixes the orbit of maximal rank in $(\YY\PP_\alpha)_{\bar k}$, unless the pair $(\YY_{\bar k},\alpha)$ is of type $U$, in which case there are two orbits of maximal rank and $w_\alpha$ interchanges them. The action defined this way is transitive on the set of Borel orbits of maximal rank (which includes the open $\BB$-orbit). The stabilizer of the open orbit is $W_{(X)}:=W_X\ltimes W_{P(X)}$, where $W_X$ is a canonical subgroup of $W$ called the ``little Weyl group'' of $\XX$. The group $W_{(X)}$ normalizes $\varchi(\XX)$ and acts on it through the quotient $W_X$. It is known that the dual action of $W_X$ on $\mathcal Q$ is faithful, generated by reflections, and admits as a fundamental domain the cone $\mathcal V$ defined previously. (More on this action will be recalled in \S \ref{ssrootsystem}.)

Since we are also discussing non-trivial line bundles $\mathcal L_\Psi$ as were described in \S \ref{ssfirst}, there is also a fifth case, called $(U,\psi)$. For a discussion of how this fits into the same setting (more precisely, into Knop's extension of his action to non-spherical varieties) we refer the reader to \cite[\S 5.5]{Sa2}. We caution the reader that, while the little Weyl group is still well-defined in this case, it does not coincide with the little Weyl group of the spherical variety. In this case the action of Knop is not transitive on all orbits of maximal rank, but only on some, called ``admissible''.

Notice that all $\BB_{\bar k}$-orbits of maximal rank are defined (and have a point) over $k$ \cite[Proposition 3.6.1]{Sa1}. We require that the same isomorphisms hold for orbits of maximal rank over $\mathfrak o$, which is clearly the case at almost every place if $\XX$ is defined globally:
\begin{axiom}\label{integralityaxiom}
 For every $\BB_k$-orbit $\YY_k\subset \XX_k$ of maximal rank and every simple root $\alpha$, the $\mathfrak o$-scheme $\YY\PP_\alpha/\mathcal R(\PP_\alpha)$ is isomorphic to one of: $\PPGL_2\backslash\PPGL_2$, $\SS\UU\backslash\PPGL_2$ (where $\SS\subset\mathcal N(\UU)$), $\TT\backslash\PPGL_2$, where $\TT$ is a non-trivial torus over $\mathfrak o$) or $\mathcal N(\TT)\backslash\PPGL_2$. Moreover, the complement of $\YY\PP_\alpha$ is the closure of the complement of $(\YY\PP_\alpha)_k$.
\end{axiom}

The scheme-theoretic structure on $\YY$ (and hence on $\YY\PP_\alpha$) implicit in the formulation of the axiom is defined as follows: $\YY$ is the complement of the closure of the complement of $\YY_k$ in its closure. It follows by the assumed properties of $\mathring\XX$ and an inductive application of this axiom that for every orbit $\YY_k$ of maximal rank (over $k$), the scheme $\YY$ is smooth and surjective over $\mathfrak o$ and its fibers are homogeneous under the fibers of $\BB$. For example, let $\alpha$ be a simple root such that $\mathring \XX\PP_\alpha/\mathcal R(\PP_\alpha)$ is of type $U$, and let $\YY_k$ be the complement of $\mathring\XX_k$ in $(\mathring\XX\PP_\alpha)_k$. Then by the axiom it follows that $\YY/\mathcal R(\PP_\alpha)$, is smooth, surjective and with $\BB$-homogeneous fibers over $\spec \mathfrak o$, hence the same holds for $\YY$. Since every orbit $\YY$ of maximal rank can be obtained from the open orbit by successive Weyl reflections of type $U$, the same conclusion holds 
for every orbit of maximal rank. In particular, Borel orbits of maximal rank over $k$ are in bijection with orbits of maximal rank over the residue field $\FF_q$ and always contain $\mathfrak o$-points.

\subsection{Brion's description of spherical reflections.} \label{ssBrion}

Consider the oriented graph $\mathfrak G$ (to be called ``weak order graph'') whose vertices are the $\BB_{\bar k}$-orbits on $\XX_{\bar k}$ (we will omit the subscript $\bar k$ for the rest of this section) and whose edges are labelled by the set $\Delta$ of simple roots, where an orbit $\YY$ is joined to an orbit $\ZZ$ by an edge labelled $\alpha$ if $\alpha$ raises $\YY$ to $\ZZ$ (that is, $\ZZ$ is the open $\BB$-orbit in $\YY\PP_\alpha$). There is also a variant of this graph, where two vertices corresponding to two orbits $\ZZ$ and $\YY$ (possibly the same) are joined by an edge labelled $\alpha$ (again, $\alpha$ is a simple root of $\GG$) if ${^{w_\alpha}\YY}=\ZZ$ under Knop's action; the connected component of this graph consisting of all orbits of maximal rank will be called ``Knop's graph''; this will only appear in some diagrams in Section \ref{secsimplespherical}, where we are not interested in depicting orbits of lower rank. 

Let $\YY$ be a $\BB$-orbit. Following Brion \cite{BrO}, for every oriented path $\gamma$ from $\YY$ to $\mathring \XX$ in the weak order graph we denote by $w(\gamma)$ the Weyl group element corresponding to the path, i.e.\ the element $w_{\alpha_n}\cdots w_{\alpha_1}$, where $\alpha_1,\dots,\alpha_n$ are the consecutive labels for the edges in the oriented path; in particular, if $\YY$ is of maximal rank then $^{w(\gamma)}\YY=\mathring\XX$ under Knop's action. We also denote by $\mathfrak G(\YY)$ the set of such paths and by $W(\YY)$ the set of all $w(\gamma)$, $\gamma\in \mathfrak G(\YY)$. Then we have:

\begin{proposition}[\cite{BrO}, Propositions 2 and 4]\label{associates}
Given $w,w'\in W(\YY)$ there is a sequence of elements $w=w_0,w_1,\dots, w_n=w'$ in $W(\YY)$ such that any consecutive $w_i,w_{i+1}$ can be written as $w_i=u w_\alpha v$, $w_{i+1} = u w_\beta v$ with $l(w_i)=l(w_{i+1})= l(u)+l(v)+1$ and $\alpha, \beta$ two mutually orthogonal simple roots.
\end{proposition}

Based on this, Brion proves:

\begin{theorem}[\cite{BrO}, Theorem 4]\label{Brionsdescription}
 A set of generators of $W_{(X)}$ consists of the elements $w_\alpha$, $\alpha\in\Delta(\XX)$, and elements $w$ with the following property: There is a decomposition $w=w_1^{-1} w_2 w_1$ such that:
\begin{itemize}
 \item $^{w_1}\mathring \XX =: \YY$ with $\operatorname{codim}(\YY)= l(w_1)$ (i.e.\ $w_1^{-1}\in W(\YY)$.)
 \item $w_2$ is either of the following two:
\begin{enumerate}
 \item equal to $w_\alpha$, where $\alpha$ is a simple root such that $(\YY,\alpha)$ is of type $T$ or $N$, or
 \item equal to $w_\alpha w_\beta$, where $\alpha,\beta$ are two orthogonal simple roots which raise the same orbit of maximal rank $\YY'$ to $\YY$.
\end{enumerate}
\end{itemize} 
\end{theorem}

(This description does not apply to the case of a non-trivial $\mathcal L_\Psi$ and the extension of Knop's action there.)

It follows from a theorem that we will recall later (Theorem \ref{thmSa2}) that the generators $w$ of this theorem which are not of the form $w_\alpha,\alpha\in\Delta(\XX)$ all belong to the little Weyl group $W_X$. We will also see in \S \ref{ssrootsystem} that the ``canonical'' generators for $W_X$, the simple reflections corresponding to ``spherical roots'', admit a description of this form.

\subsection{Hyperspecial and Iwahori orbits} 
As mentioned, we assume that $\GG$ possesses a smooth model over $\mathfrak o$, so that $K=\GG(\mathfrak o)$ is a hyperspecial maximal compact subgroup. We also set $J$ for the Iwahori subgroup of $K$, that is, the inverse image of the standard Borel subgroup under the map $K\to \GG(\FF_q)$. We will assume throughout the following axioms pertaining to $K$- and $J$-orbits on $X$.

\begin{axiom} \label{orbitstheorem}
The set $A_X^+$ contains a complete set of representatives for $K$-orbits on $X$; elements of $A_X^+$ which map to distinct elements of $\Lambda_X^+$ belong to different $K$-orbits.
\end{axiom}

\begin{axiom} \label{Iwahori}
For $x \in A_X^+$ we have $xJ \subset x\BB(\mathfrak o)$.
\end{axiom}

The first axiom generalizes the Iwasawa (for $\XX=\UU\backslash\GG$) and Cartan (for $\XX=\GG'$, $\GG=\GG'\times\GG'$) decompositions. It was proven by Luna and Vust \cite{LV} in the case $\mathfrak o=\CC[[t]]$. An alternate proof was given by Gaitsgory and Nadler in \cite{GN}, which was adapted to the $p$-adic case, under assumptions which hold at almost every place if the variety is defined over a global field, in \cite{Sa3}. In the case of a symmetric variety, alternate proofs of similar statements have recently been presented by Benoist and Oh \cite{BO}, Delorme and S\'echerre \cite{DS}.

The second axiom was also proven in \cite{Sa3} under similar assumptions.

\begin{remark}
From the first axiom it follows that we have surjective maps: $A_X^+/A_0\to X/K \to \Lambda_X^+$, but in general these maps are not bijective. An example is: $\GG=\Gm$, $\XX=\{\pm 1\}\backslash\GG$ where the first map is bijective but not the second (since the residue field has non-trivial square classes), and another is $\XX=\mathbf O_2\backslash\GGL_2$, the space of non-degenerate quadratic forms (assume that the residue characteristic is not two), where the first map is also not bijective since the set $\AA_X(\mathfrak o)/\AA(\mathfrak o)$ has four elements, but there are only two classes of non-degenerate quadratic forms over the residue field, and hence these four elements are contained in only two orbits of $K$ on $X$. (We certainly expect that the second map is bijective when $\HH$ is connected.)
\end{remark}


\section{Intertwining operators}\label{secintertwining}

\subsection{Overview} The goal of this paper is to compute eigenfunctions of the Hecke algebra on $C^\infty(X)$ or, more generally, on $C^\infty(X,\mathcal L_\Psi)$, when $L_\Psi$ is as in \S \ref{ssfirst}. The case of $\mathcal L_\Psi$ will generally be suppressed from our notation, except when it needs to be treated separately. Unless otherwise stated, all our statements hold -- with obvious modifications -- in that case, as well.

We let $I(\chi)= \Ind_B^G(\chi\delta^\frac{1}{2})$, a principal series representation. (For normalized induction from another parabolic subgroup $P$ we will be using the notation $I_P$.) Then a Hecke eigenfunction on $C^\infty(X)$ is the image of a $K$-invariant vector $\phi_{K,\chi} \in I(\chi)$ via an intertwining operator $I(\chi)\to C^\infty(X)$ (for some $\chi$).

The details of the present section are quite technical, therefore we give here an overview of its contents. We will introduce a $G$-eigenmeasure $|\omega_X|$ on $X$ to set up a duality:
\begin{eqnarray}\label{duality}
  C_c^\infty(X)\otimes \nu \otimes C^\infty(X)\to \CC \\
  \phi \otimes |\omega_X| \otimes f \mapsto \int_X \phi f |\omega_X| \nonumber
\end{eqnarray}
(where $\nu$ is the character of $|\omega_X|$). As notation suggests, $|\omega_X|$ is the absolute value of a volume eigen-form on $\XX$, the eigencharacter of which will be denoted by $\mathfrak n$ (hence $\nu=|\mathfrak n|$).

For every $B$-orbit $Y$ on $X$ (assuming for now that $Y=\YY(k)$ is a single $B$-orbit) we will introduce morphisms:
$$S_\chi^Y: C_c^\infty(X) \to I(\chi)$$
given (in some domain of convergence for $\chi$) by the formula:
$$ S_\chi^Y(\phi) (1)= \int_{Y} \phi \mu_\chi^Y$$
where $\mu_\chi^Y$ is a suitable $B$-eigenmeasure on $Y$. For $\YY=\mathring \XX$ we will sometimes omit the exponent $^Y$ from the notation.

We view $I(\chi)$ as a subspace of the space of smooth functions on $U\backslash G$, and through the duality (\ref{duality}) we have an adjoint for $S_{\chi^{-1}\nu^{-1}}^Y$ (again, by fixing an invariant measure on $U\backslash G$):
$$ \Delta_\chi^Y: \mathcal S(U\backslash G)\to C^\infty(X).$$
The space $\mathcal S(U\backslash G)$ is a larger space than $C_c^\infty(U\backslash G)$ suitable for the theory of Fourier transforms, called the \emph{Schwartz space of $\overline{U\backslash G}^{\rm{aff}}$}. (The exponent $\rm{aff}$ denotes the \emph{affine closure} of $U\backslash G$; however we will allow ourselves to abuse language and notation, write $\mathcal S(U\backslash G)$ and say ``Schwartz space of $U\backslash G$''.)

We will show that $\Delta_\chi^Y$ can be expressed (in a suitable domain of convergence for $\chi$) by a formula:
$$ \Delta_\chi^Y(\Phi) (\xi)= \int_{U\backslash G} \Phi \mu'^Y_\chi$$
where the point of evaluation $\xi$ belongs to $Y$ and $\mu'^Y_\chi$ is a suitable $A\times G_\xi$-eigenmeasure on $U\backslash G$. (Notice that $U\backslash G$ carries a natural action of $A\times G$.) 

While our goal is to compute $\Delta_\chi^Y (\Phi_K)$, with $\Phi_K \in \mathcal S(X)^K$, this cannot be done directly. Rather, we first compute $\Delta_\chi^Y (\Phi_J)$, where $\Phi_J$ is a suitable function of small support in $\mathcal S(U\backslash G)^J$, and explain the steps needed to deduce from this the computation of $\Delta_\chi^Y (\Phi_K)$. These steps will be undertaken in the following two sections.

The biggest part of this section is devoted to choosing eigenmeasures in compatible ways and rigorously proving the formula for $\Delta_\chi^Y$. A major complication that arises is that $Y=\YY(k)$ will, in general, consist of several $B$-orbits -- in this case the morphisms $\Delta_\chi^Y$ are not uniquely defined, and we must instead introduce variants, denoted $\Delta_{\tilde\chi}^Y$. These are defined by suitable eigenmeasures on $Y$, and the correct way to choose eigenmeasures is that these be absolute values of differential eigenforms. Then we need to compare differential forms on different $\BB$-orbits with each other, or with differential forms on $\UU\backslash \GG$. The most technical parts of this section can be skipped at first reading.

\subsection{Morphisms and multiplicity}\label{ssmorphisms}
Let $\YY$ be a $\BB$-orbit on $\XX$ with a distinguished $k$-point $y_0$ (we will later discuss how to choose $y_0$), and let $\AA_Y$ denote the quotient of $\AA$ by the stabilizer modulo $\UU$ of $y_0$. Fix a top-degree $\BB$-eigenform $\omega_Y$ on $\YY$, whose character we will denote by $\mathfrak c$, and let $\tilde \chi'$ vary over all characters of $A_Y$ which are unramified when restricted to the image of $A\to A_Y$. Through our choice of $y_0$ we get an identification: $\YY/\UU\simeq \AA_Y$, and we consider $\tilde\chi'$ as a function on $Y$. In \cite{Sa2} we defined a morphism $S^Y_{\tilde\chi}: C_c^\infty(X)\to I(\chi)$ which, composed with evaluation at ``1'' is given by the rational continuation of the integral:

\begin{equation}\label{defofS}
\ev_1\circ S_{\tilde\chi}^Y(\phi)= \int_Y \phi(y) \tilde\chi'^{-1}(y) |\omega_Y|(y).
\end{equation}
This integral converges and represents a rational function for a certain domain of the parameter $\chi$.

\begin{definition}
 An eigenmeasure on $Y$ of the form $\tilde\chi'^{-1} |\omega_Y|$ will be called \emph{pseudo-rational}.
\end{definition}

We explain now what the index $\tilde\chi$ which appears in the notation is: Consider the map: $\AA(\bar k)\to \AA_Y(\bar k)$ and let $R$ denote the preimage of $\AA_Y(k)$. The expression (\ref{defofS}) provides a complex character $\tilde\chi$ of $R$ defined as: $\tilde\chi=\delta^{-\frac{1}{2}}\cdot \tilde\chi'\cdot |\mathfrak c|^{-1}$. (Notice that unramified characters such as $\delta^{-\frac{1}{2}}$ and $|\mathfrak c|^{-1}$ make sense on the whole of $\AA(\bar k)$ -- the reason is that there are canonical algebraic characters, of which these unramified characters are a power of the absolute value.) The characters obtained this way are a torsor for the group of characters of $A_Y$ which are unramified on the image of $A$. If $\chi$ denotes the restriction of $\tilde\chi$ to $A$, the above expression (\ref{defofS}) defines a $(B,\chi\delta^\frac{1}{2})$-equivariant functional on $C_c^\infty(X)$, and by Frobenius reciprocity a morphism: $C_c^\infty(X)\to I(\chi)$. When no confusion arises, we will denote 
by $S_{\tilde\chi}^Y$ both the functional and the corresponding morphism into $I(\chi)$. 

For the principal series $I(\chi)$ to admit such a morphism, the character $\chi$ must satisfy the condition:
\begin{equation}\label{cond1}
 \chi^{-1}\delta^{\frac{1}{2}}|_{B_y} = \delta_{B_y}
\end{equation}
for a (any) point $y\in Y$. By abuse of language we will often say that ``$\tilde\chi$ is a character of $A_Y$ which extends $\chi$'', although neither $\tilde\chi$ nor $\chi$ are, in general, characters of $A_Y$.

A main result of \cite{Sa2} was:

\begin{theorem}\label{thmmult}
 Let $\pi$ be an irreducible unramified representation of $G$. Then a necessary condition for $\pi$ to be $X$-distinguished (i.e.\ $\Hom_G(\pi,C^\infty(X))\ne 0$) is that $\pi$ is the unramified subquotient of $I_{P(X)}(\chi)$ for some $\chi\in \overline{A_X^*}$. For a generic $\chi\in \overline{A_X^*}$, if $\pi$ is the irreducible unramified subquotient of $I_{P(X)}(\chi)$ then:
$$ \dim\Hom_G(\pi, C^\infty(X)) = g\cdot a,$$
where $g$ is the degree of the generic fiber of the map: $\overline{A_X^*}/W_X\to A^*/W$ and $a$ is the number of open $B$-orbits.
\end{theorem}

Moreover, it was shown that the space of such morphisms is spanned by the adjoints of the above morphisms defined using the open orbits, and their composites with standard intertwining operators (see also Theorem \ref{thmSa2}). The factor $g$ should be thought of as a ``geometric'' multiplicity factor, while the factor $a$ should be thought of as an ``arithmetic'' (or Galois-cohomological) factor.

Notice that until now our definition of the family $S_{\tilde\chi}^Y$ depends on the choice of $y_0$ and of $\omega_Y$. This was enough for the purposes of \cite{Sa2}, but here we need to be more careful and to normalize the morphisms $S_{\tilde\chi}^Y$ in a precise way.

\subsection{Compatible choices of eigenforms and $B_0$-orbits}\label{eigenforms} Our goal now is to normalize the morphisms $S_{\tilde\chi}^Y$ of (\ref{defofS}) in compatible ways, for all $\BB$-orbits $\YY$ of maximal rank. At the same time, we will also define compatible isomorphisms: $\YY/\UU\simeq \AA_Y$, up to multiplication by $A_0$.

\begin{lemma}
 Given a $\BB$-orbit $\YY$ (over $k$), the following data are equivalent:
\begin{enumerate}
 \item A trivialization of the $\AA_Y$-torsor $\YY/\UU$, i.e.\ an isomorphism: $\YY/\UU\simeq \AA_Y$.
 \item A splitting of the short exact sequence:
$$ 1\to k^\times \to k(\YY)^{(\BB)} \to \varchi(\YY)\to 1.$$
 \item A family $\{\omega_{\mathfrak c}\}_{\mathfrak c}$ of non-zero $k$-rational $\BB$-eigen-volume forms on $\YY$, determined up to a common multiple, where $\mathfrak c$ ranges over all possible characters of such eigenforms and the $\omega_{\mathfrak c}$ have the property that their quotients (which are $\BB$-eigenfunctions) all are equal to $1$ at the same $k$-point of $Y$.
\end{enumerate}
\end{lemma}

\begin{proof}
 The implication $1\Rightarrow 2$ is obvious. For the converse, we notice that $k(\YY)^{(\BB)}=k[\YY]^{(\BB)}$ and $k[\YY/\UU]=\oplus_{\chi\in\varchi(\YY)} k[\YY]_{\chi}$, where $k[\YY]_\varchi$ denotes the corresponding $1$-dimensional subspace. For any splitting $\chi\mapsto e_\chi\in k[\YY]_{\chi}$ the maps $e_\chi\mapsto 1\in k$ extend uniquely to a homomorphism of algebras: $k[\YY/\UU]\to k$, i.e.\ a $k$-point on $\YY/\UU$.

For the third condition we notice that there exists an eigen-volume form on $\YY$ with character $\mathfrak c$ if and only if $\mathfrak c|_{\BB_y}=\mathfrak d_{\BB_y} \mathfrak d_{\BB}^{-1}|_{\BB_y}$, and that every algebraic character of $\BB_y$ (where $y\in\YY(k)$) extends to a character of $\BB$. In particular, there exists an eigen-volume form. Multiplying such a form by elements of $\varchi(\YY)$, considered as eigenfunctions on $\YY$ via an identification $\YY/\UU\simeq \AA_Y$ we get a family $\{\omega_{\mathfrak c}\}_{\mathfrak c}$ with the stated property, which depends only up to a common scalar on the form chosen originally. Vice versa, the quotients of elements in such a family are eigenfunctions distinguishing a unique point in $\YY/\UU(k)$.
\end{proof}

Recall that $A_0$ and $B_0$ denote, respectively, the maximal compact subgroups of $A$ and $B$. A variant of this lemma is:

\begin{lemma}\label{equivalence}
 Given a $\BB$-orbit $\YY$, the following data are equivalent:
\begin{enumerate}
 \item \label{one} An isomorphism: $Y/A_0U\simeq A_Y/A_0$.
 \item \label{two} A continuous splitting of the short exact sequence:
$$ 1\to \CC^\times\to $$ $$ \{A_Y\mbox{-eigenfunctions on }Y/U\mbox{ with eigencharacters which are trivial in the image of }A_0\}$$ $$\to \{\mbox{characters of }A_Y\mbox{ which are trivial in the image of }A_0\}\to 1$$
 \item \label{three} A family of non-zero pseudo-rational $B$-eigenmeasures on $Y$, determined up to a common scalar multiple, which have the property that their quotients by any fixed element in the family form a family of eigenfunctions as in the image of the splitting of (\ref{two}).
\end{enumerate}
\end{lemma}

Notice also that choosing an isomorphism: $A_Y/A_0 \to Y/A_0U$ with the additional property that $1\mapsto ($the coset of an element of $\YY(\mathfrak o))$ is equivalent to choosing a $B_0$-orbit on $\YY(\mathfrak o)$.

\subsubsection{Generalities on differential forms.} \label{generalities} Let $\GG_1\supset \GG_2\supset \GG_3$ be algebraic groups. The following is the algebro-geometric version of the factorization of an integral: $$ \int_{G_3\backslash G_1} f(g) dg = \int_{G_2\backslash G_1} \int_{G_3\backslash G_2} f(hx) dh dx$$ (where $dg, dh, dx$ are suitable eigenmeasures), which will be useful when the corresponding sets of $k$-points do not surject on the $k$-points of the quotient.

Let $\mathfrak q:\VV\to \YY$ be a smooth morphism of schemes. Then, by definition, the sheaf of relative differentials $\Omega_{V/Y}$ is locally free on $\VV$. Now assume in addition that $\YY$ (and hence also $\VV$) is smooth over $\spec(k)$, and let $\bigwedge^{\rm{top}}$ denote the top-degree exterior powers (i.e.\ determinants) of the vector bundles (sheaves of differentials) $\Omega_V, \Omega_Y, \Omega_{V/Y}$. From the short exact sequence: $0\to\Omega_Y\overset{\mathfrak q^*}{\to}\Omega_V\to\Omega_{V/Y}\to 0$ we get a canonical isomorphism of line bundles on $\VV$: $\bigwedge^{\rm{top}}(\Omega_V)=\bigwedge^{\rm{top}}(\Omega_{V/Y})\otimes \mathfrak q^*\bigwedge^{\rm{top}}(\Omega_Y)$. If $\omega_1$ denotes a section of $\bigwedge^{\rm{top}}(\Omega_{V/Y})$ and $\omega_2$ denotes a section of $\bigwedge^{\rm{top}}(\Omega_Y)$ the corresponding section $\omega_3$ of $\bigwedge^{\rm{top}}(\Omega_V)$ will be denoted (by a slight abuse of notation) by $\omega_1\wedge\mathfrak q^*(\omega_2)$ and will be called a 
\emph{factorization} of $\
omega_3$ with respect to the morphism $\mathfrak q:\VV\to \YY$. 

Now we return to the situation of $\GG_1\supset \GG_2\supset \GG_3$. Let $\omega$ be a $\GG_1$-eigen-volume form on $\GG_3\backslash \GG_1$. For a non-zero such form to exist, the character $\mathfrak c$ of $\omega$ must satisfy: $\mathfrak c|_{\GG_3}=\mathfrak d_{\GG_3} \mathfrak d_{\GG_1}^{-1}|_{\GG_3}$. Let $\mathfrak c'$ be the character $\mathfrak c \mathfrak d_{\GG_2}^{-1} \mathfrak d_{\GG_1}|_{\GG_2}$ of $\GG_2$. Any character $\mathfrak c'$ of $\mathfrak \GG_2$ defines a line bundle over $\GG_2\backslash\GG_1$, equipped with a trivialization of its pull-back to $\GG_1$; by definition, the sections over an open $\YY\subset\GG_2\backslash\GG_1$ with preimage $\VV\subset\GG_1$ are sections $f$ of the structure sheaf of $\VV$ such that $f(g_2 g)=\mathfrak c'(g_2) f(g)$ for any $g_2\in\GG_2$. Let $\mathcal L$ be the line bundle defined by $\mathfrak c'$. Then $\omega$ admits a factorization: $\omega_1\wedge \mathfrak q^*(\omega_2)$ with respect to $\mathfrak q: \GG_3\backslash\GG_1\to\GG_2\backslash\GG_1$ 
and 
$\mathcal L$, where $\omega_1\in \Gamma(\GG_3\backslash\GG_1, \bigwedge^{\rm{top}}(\Omega_{(G_3\backslash G_1)/(G_2\backslash G_1)})\otimes \mathcal \mathfrak q^* \mathcal L^{-1})$ and $\omega_2\in \Gamma(\GG_2\backslash\GG_1,\bigwedge^{\rm{top}}(\Omega_{G_2\backslash G_1})\otimes \mathcal L)$ are eigenforms for $\GG_1$ and $\GG_2$, with eigencharacters $\mathfrak c$ and $\mathfrak c'$, respectively. For every point $x\in \GG_3\backslash \GG_1$, the form $\omega_1$ gives rise to a volume form on the $\GG_2$-orbit of $x$: First, the choice of $x$ gives rise to a trivialization of $\mathfrak q^*\mathcal L^{-1}$ over its $\GG_2$-orbit, and then the pull-back of $\omega_1$ via the inclusion $x\GG_2\to \GG_3\backslash\GG_1$ is a volume form on $x\GG_2$. Therefore, the form $\omega_1$ will be called a ``volume form along $\GG_2$-orbits''.

\subsubsection{The case of two adjacent orbits.} \label{twoorbits} Now let $\YY$ and $\ZZ$ be two $\BB$-orbits on $\XX$ with $\ZZ\subset\overline\YY$, $\alpha$ a simple root joining them of type $U$ under Knop's action. Assume that we are given a family of non-zero pseudo-rational $B$-eigenmeasures on $Y$ as in Lemma \ref{equivalence}. Assume moreover that the distinguished $A_0U$ orbit on $Y$ is that of a point $y_0\in\YY(\mathfrak o)$. We will show how to associate to this family a similar family of eigenmeasures on $Z$.

Pick a representative $\tilde w_\alpha$ for $w_\alpha$ in $\mathcal N(\AA)(\mathfrak o)$; then as $\mathfrak o$-schemes: $\ZZ \times \UU_\alpha \simeq \YY$ under the map $(z,u)\mapsto z\tilde w_\alpha u$. Thus, to a $\BB$-eigen-volume form $\omega_Z$ on $\ZZ$ (with eigencharacter $\mathfrak c$) we can associate the volume form $\omega_Y=\omega_Z\wedge dx$ on $\YY$, where $dx$ denotes the obvious eigenform on $\UU_\alpha$ induced through an $\mathfrak o$-isomorphism: $\UU_\alpha\simeq \mathbb A_x^1$ (hence uniquely determined up to $\mathfrak o^\times$). Then $\omega_Y$ is also an eigenform for $\BB$ with eigencharacter equal to $e^{-\alpha}\cdot \,{^{w_\alpha}\mathfrak c}$. Notice that the corresponding measure $|\omega_Y|$ depends only on $\omega_Z$ and not on any other of the choices made. Hence, for any family of pseudo-rational eigenmeasures on $Y$ as in Lemma \ref{equivalence}.(\ref{three}), we get a family of pseudo-rational eigenmeasures on $Z$. 

This process, in particular, gives rise by Lemma \ref{equivalence} to compatible choices of $B_0$-orbits on $\YY(\mathfrak o)$ and $\ZZ(\mathfrak o)$ which we explicate here: Let $y_0B_0$ be a given $B_0$-orbit on $\YY(\mathfrak o)$, consider the quotient $\YY\PP_\alpha\to \YY\PP_\alpha/\UU_{P_\alpha}$ and denote by $\overline{\bullet}$ images under this map, e.g.\ $\overline{y_0}$, $\overline{y_0 B_0}$ etc. The $\LL_\alpha(\mathfrak o)$-orbit of $\overline{y_0}$ is isomorphic to $F \UU_\alpha(\mathfrak o)\backslash \LL_\alpha(\mathfrak o)$, where $F$ is a subgroup normalizing $\UU_\alpha$ with $F\cap [L_\alpha,L_\alpha]$ finite. Let $J_\alpha$ be the Iwahori subgroup of $\LL_\alpha$ with respect to the Borel $\BB_\alpha:=\LL_\alpha\cap \BB$. By the Iwahori-Bruhat decomposition, $\LL_\alpha(\mathfrak o)= J_\alpha \sqcup J_\alpha w_\alpha J_\alpha = J_\alpha \sqcup \UU_\alpha(\mathfrak o) w_\alpha \BB_\alpha(\mathfrak o)$, the group $J_\alpha$ has two orbits in $\overline{y_0}\LL_\alpha(\mathfrak o)$: one of 
them is the $\BB_\alpha(\mathfrak o)$-orbit of $y_0$ and the other intersects the 
smaller $\BB_\alpha$-orbit in a unique $\BB_\alpha(\mathfrak o)$-orbit $\overline{z_0 B_0}$. Each fiber of the quotient: $\YY\PP_\alpha\to \YY\PP_\alpha/\UU_{P_\alpha}$ being homogeneous under the unipotent group $\UU_{P_\alpha}$, the preimage of $\overline{z_0 B_0}$ intersects $\ZZ(\mathfrak o)$ in a unique $B_0$-orbit $z_0 B_0$.

\begin{lemma}\label{composeStypeU}
Let $\YY$, $\ZZ$ be two orbits joined by a root $\alpha$ under Knop's action, and assume that they are equipped with compatible families of pseudo-rational $B$-eigenmeasures as above. Use these eigenmeasures to define the morphisms $S_{\tilde\chi}^Y$ of (\ref{defofS}). Let $T_{w_\alpha}:I(\chi)\to I({^{w_\alpha}\chi})$ denote the standard intertwining operator between principal series:
$$T_{w_\alpha}(f)(g)= \int_{U_\alpha} f(\tilde w_\alpha ug) du$$
defined with some preimage $\tilde w_\alpha \in \mathcal N(\AA)(\mathfrak o)$ of $w_\alpha$, and with measure $dx$ on $U_\alpha\simeq \mathbb A_x^1$ (with the latter isomorphism over $\mathfrak o$).
Then:
$$ T_{w_\alpha} \circ S_{\tilde\chi}^Z =  S_{^{w_\alpha}\tilde\chi}^Y.$$
\end{lemma}

This is just a strengthened version of a special case of \cite[Theorem 5.2.1]{Sa2}.

\begin{proof}
We know already from \cite[Theorem 5.2.1]{Sa2} that one side has to be a rational multiple of the other. To compute the proportionality constant, it is enough to do it when the morphisms of both sides -- viewed as functionals by Frobenius reciprocity, are applied to functions with support in $Z\cup Y$. For those, one writes down explicitly the integral expressions for both sides, when $\tilde\chi$ is in the region of convergence of $T_{w_\alpha}$.
\end{proof}

Now consider the weak order graph $\mathfrak G$, which was defined in \S \ref{ssBrion}.
We fix a class of pseudo-rational eigenmeasures, to be called ``standard'', on $\mathring X$, with the property that their quotients are equal to $1$ on our distinguished point $x_0$. This condition determines them up to a common multiple, and the precise normalization will be discussed in a later section. Choosing a path $\gamma$ in $\mathfrak G$ from $\YY$ to $\mathring \XX$ (where $\YY$ is of maximal rank), the above process induces a similar class of eigenmeasures on $Y$ (and hence also a $B_0$-orbit $y_0B_0 \subset \YY(\mathfrak o)$) which, a priori, depends on the choice of path $\gamma$. Our goal is to show:

\begin{proposition}\label{independence}
 The induced class of eigenmeasures on $Y$ (and therefore also the orbit $y_0B_0$) does not depend on the choice of path $\gamma$.
\end{proposition}

\begin{proof}
First, we claim that the induced class of eigenmeasures on $Y$ does not depend on $\gamma$ itself but only, possibly, on $w(\gamma)$. For this, it is enough to show that for two paths $\gamma_1,\gamma_2\in\mathfrak G(\YY)$ with $w(\gamma_1)=w(\gamma_2)=w$ the families of morphisms $S_{\tilde\chi}^Y$ defined with the eigenmeasures obtained through each of the two paths coincide. But this follows from Lemma \ref{composeStypeU}: Indeed, the lemma implies that $T_w\circ S_{\tilde\chi}^Y = S_{^w\tilde\chi}^{\mathring X}$ for both families $S_{\tilde\chi}^Y$, and therefore $S_{\tilde\chi}^Y$ is the same rational multiple of $T_{w^{-1}} \circ S_{^w\tilde\chi}^{\mathring X}$ for both families. By Lemma \ref{equivalence}, this implies that the corresponding families of eigenmeasures are the same.

Now, using Proposition \ref{associates}, it is sufficient to prove the following: If $\alpha$, $\beta$ are two orthogonal simple roots which raise an orbit $\ZZ$ to an orbit $\YY$ then the identifications defined by $\alpha$ and $\beta$ coincide. We can substitute the variety $\YY\PP_{\alpha\beta}$ by the variety $\YY\PP_{\alpha\beta}/\UU_{P_{\alpha\beta}}$. The $\LL_{\alpha\beta}'$-orbit of a $\mathfrak o$-point on the latter is isomorphic to one of the varieties $\SSL_2$, $\PPGL_2$, acted upon by left and right multiplication (under an isomorphism $\LL_{\alpha\beta}'\simeq \SSL_2\times\SSL_2$ -- recall that we have assumed that $\GG$ is simply connected). We are thus reduced to the case $\YY=$ the open Bruhat cell and $\ZZ=$ the closed Bruhat cell in $\SSL_2$ (or $\PPGL_2$). Then the isomorphisms $\ZZ\times\UU_\alpha\simeq\YY$ and $\ZZ\times\UU_\beta\simeq \YY$, used in \S \ref{twoorbits} to define compatible choices of eigenmeasures, correspond to nothing else than the decomposition of the open Bruhat 
cell as a product $\UU w\BB$ or $\BB w \UU$, respectively, and hence it is immediate to see that given a class of eigenforms on $\BB$ of the form of Lemma \ref{equivalence}, the induced class of eigenforms on $\BB w \BB$ does not depend on which decomposition we use.
\end{proof}

This proposition allows us to define the notion of ``standard'' eigenmeasures on each orbit $Y$ of maximal rank, which depends only on the choice of standard eigenmeasures on $\mathring X$.

\subsection{Adjoints} \label{adjoints}
Having chosen a point $x_0\in \mathring \XX(\mathfrak o)$, we normalize our family of standard pseudo-rational eigenmeasures on $\mathring X$ by requiring that $\mu(x_0 J)=1$, where $J$ is the Iwahori subgroup. Recall that $x_0J\subset x_0 B_0$ by Axiom \ref{Iwahori}, therefore such a normalization is possible, since all eigenmeasures in the family differ by unramified $B$-eigenfunctions. So we have:

\begin{definition}
 The pseudo-rational $B$-eigenmeasures on $\mathring X$ normalized according to the property $\mu(x_0 J)=1$ will be called \emph{standard}. The corresponding eigenmeasures on any orbit $Y$ of maximal rank, obtained via Proposition \ref{independence}, will also be called \emph{standard}.
\end{definition}

From now on we will only be using the morphisms $S_{\tilde\chi}^Y$ defined as in (\ref{defofS}) with the use of \emph{standard} eigenforms and eigenmeasures.

Recall that we are assuming the existence of a $G$-eigenmeasure on $X$; we fix from now on an eigenmeasure $|\omega_X|$, chosen to be the class of standard eigenmeasures. We denote by $\mathfrak n$ the eigencharacter of the underlying eigenform $\omega_X$ and by $\nu$ the eigencharacter of $|\omega_X|$. Multiplication by this measure yields an identification:
\begin{equation} \label{ident} 
 C_c^\infty(X)\otimes \nu \simeq M_c^\infty(X).
\end{equation}

We can now, via (\ref{ident}), modify $S_{\tilde\chi^{-1}\nu^{-1}}^Y$ into a morphism: 
$$\tilde S_{\tilde\chi^{-1}\nu^{-1}}^Y: M_c^\infty(X)\simeq C_c^\infty(X)\otimes \nu \xrightarrow{S_{\tilde\chi^{-1}\nu^{-1}}^Y\otimes 1} I(\chi^{-1}\nu^{-1})\otimes \nu \simeq I(\chi^{-1})$$
via the canonical isomorphism: $I(\chi^{-1}\nu^{-1})\otimes\nu \simeq I(\chi^{-1})$. From now on we drop the notation $\tilde S^Y_{\tilde\chi^{-1}\nu^{-1}}$, and we will denote $\tilde S^Y_{\tilde\chi^{-1}\nu^{-1}} $ by $S^Y_{\tilde\chi^{-1}\nu^{-1}}$. It should be clear from the context whether we are referring to a morphism from $C_c^\infty(X)$ to $I(\chi^{-1}\nu^{-1})$ or a morphism from $M_c^\infty(X)$ to $I(\chi^{-1})$.

\begin{remark}
 In fact, it is more natural to introduce the morphism $\tilde S_{\tilde\chi^{-1}\nu^{-1}}$ (recall that we have $\YY=\mathring\XX$ when the exponent $^Y$ is omitted) directly as a morphism: $M_c^\infty(X)\to I(\chi^{-1})$, namely, integration of any measure against the character $\delta^{-\frac{1}{2}}\tilde\chi$, considered as a function on $\mathring X$ via the choice of the point $x_0$. The reference to $C_c^\infty(X)$ is artificial and less canonical, since it depends on the choice of $|\omega_X|$. However, to handle the sheaf $M_c^\infty$ on $X$ and describe its restrictions to the smaller orbits, it is convenient to work with functions instead of measures, since restrictions of functions make sense as functions.
\end{remark}

We will denote by $$\Delta^Y_{\tilde\chi}: I(\chi)\to C^\infty(X)$$ the morphism adjoint to $S_{\tilde\chi^{-1}\nu^{-1}}^Y$. (Here we are using the standard pairing between $I(\chi)$ and $I(\chi^{-1})$: $(f_1,f_2)\mapsto \int_K f_1(k)f_2(k)dk$ with $\Vol(K)=1$.) Notice that for $Y=\mathring X$, $\Delta^Y_{\tilde\chi}$ is defined whenever $\chi\in\delta^{\frac{1}{2}} \overline{A_X^*}$. By our assumption that $\mathring X$ carries a $B$-invariant measure, we have $\delta|_{B_y}=\delta_{B_y}$ for a point $y$ on the open orbit. It is known \cite[Th\'eor\`eme 3.4]{BLV} that the unipotent radical of $\BB_y$ is equal to $\LL(\XX)\cap \UU$; therefore, $\delta_{B_y} = \delta_{(X)}|_{B_y}$ and hence:
\begin{equation}\label{deltapx}
 \delta_{P(X)}\in A_X^*
\end{equation}
(more precisely, it belongs to $\overline{A_X^*}$ but, as we remarked in \ref{ssmorphisms}, it has a natural lift to $A_X^*$) and the condition for the open orbit can be written:
\begin{equation}\label{condopen}
 \chi\in\delta_{(X)}^{\frac{1}{2}} \overline{A_X^*}.
\end{equation}

In general, the condition (see (\ref{cond1})) is:
\begin{equation}\label{cond2}
 \chi\nu\delta^{\frac{1}{2}}|_{B_y}=\delta_{B_y}
\end{equation}
for any point $y\in Y$.

We want to obtain an explicit formula for $\ev_\xi\circ\Delta_{\tilde\chi}^Y$, where $\xi$ is some point of $Y$. Assume for now that a point $\xi$ has been chosen, and denote by $\GG_\xi$ its stabilizer.

We fix volume eigen-forms on the spaces $\XX,\GG,\UU\backslash\GG$ as follows:
\begin{itemize}
 \item As mentioned above, the volume form $\omega_X$ on $\XX$ will be standard, hence normalized by the condition: $|\omega_X|(x_0 J)=1$.
 \item On $\UU\backslash \GG$, we fix the invariant volume form $\omega_{U\backslash G}$ such that \linebreak $|\omega_{U\backslash G}|(U\backslash UK)=1$.
 \item On $\GG$ we fix the invariant volume form $\omega_G$  such that $|\omega_G|(K)=1$.
\end{itemize}

These induce differential forms $\omega_\xi$, $\omega_U$ ``along'' $\GG_\xi$-orbits, respectively $\UU$-orbits, on $\GG$ (cf.\ \S \ref{generalities}) which factorize $\omega_G$ with respect to the orbit maps: $$o_\xi:\GG\ni g\mapsto \xi g \in \XX$$ and $$o_U: \GG\ni g\mapsto \UU g \in \UU\backslash \GG.$$ In other words, in the notation of \S \ref{generalities} we have:
$$\omega_G:={\mathfrak n}^{-1} \omega_\xi\wedge o_\xi^*\omega_X = \omega_U\wedge o_U^*(\omega_{U\backslash G}).$$
(This statement is essentially just saying that for a compactly supported function $\phi$ on $G$, considering $G_\xi\backslash G$ as a subset of $X$ via the orbit map, we have:
$$ \int_G \phi(g) = \int_X \int_{G_\xi} (\nu^{-1}\phi)(hx) = \int_{U\backslash G} \int_U \phi(ug)$$
with respect to the corresponding measures.) 

Finally, recall that we have a standard top $\BB$-eigenform $\omega_Y$ on $\YY$, involved in the definition of $S_{\tilde\chi}^Y$ (\ref{defofS}). We use $\omega_\xi$ to define compatible eigenforms $\omega_{\tilde Y}, \omega_{\hat Y}$ on $\tilde\YY:=\BB\GG_\xi\subset\GG$ and $\hat\YY:= \UU\backslash \BB\GG_\xi \subset \UU\backslash \GG$; explicitly, these forms satisfy:
$$\omega_{\tilde Y}:= \iota^*(\omega_\xi\wedge o_\xi^*(\omega_Y)) = \omega_U\wedge o_U^*(\omega_{\hat Y})$$
where $\iota:\GG\to\GG$ is the map $g\mapsto g^{-1}$. Notice that $\omega_{\tilde Y}$, resp.\ $\omega_{\hat Y}$, is a $\BB\times \GG_\xi$-eigenform, resp.\ an $\AA\times \GG_\xi$-eigenform, with eigencharacter equal to $\mathfrak d_{\GG_\xi}$ times the character of $\omega_Y$.

For reasons that will become apparent later, we would like to think of $I(\chi^{-1})$ as living in $C^\infty(U\backslash G)$ (and, later, for almost every $\chi$, in the space  $C_\temp^\infty(U\backslash G)$ of \emph{tempered}, smooth functions on $U\backslash G$, i.e.\ smooth functions which, when multiplied by an invariant measure on $U\backslash G$, become distributions on the Schwartz space which will be defined below). Via the duality between $I(\chi)$ and $I(\chi^{-1})$, the functional $\ev_\xi\circ \Delta_{\tilde\chi}^Y$ can be considered as a generalized vector in $I(\chi^{-1})$ \label{generalizedvector}-- the space of these vectors will be denoted by $I(\chi^{-1})^{-\infty}$ -- or a generalized function on $U\backslash G$ (again: \emph{tempered} with respect to the Schwartz space that we are going to define). Thus, we may think of $\Delta_{\tilde\chi}^Y$ as a morphism: $$M_c^\infty(U\backslash G)\to I(\chi)\to C_\temp^\infty(X).$$ Identifying $M_c^\infty(U\backslash G)$ and $C_c^\infty(U\backslash 
G)$ via the invariant measure that we fixed before we get a morphism: $$C_c^\infty(U\backslash G)\to I(\chi) \to C_\temp^\infty(X),$$ still to be denoted by $\Delta_{\tilde\chi}^Y$. We notice that the first arrow of this morphism can be described as the integral: $$\Phi\mapsto \int_A \Phi(Ua\bullet) \chi^{-1}\delta^{-\frac{1}{2}}(a) da$$ with measure $1$ on $A_0$.

The following is an explicit description of $\Delta_{\tilde\chi}^Y$:

\begin{lemma} \label{lemmaadj}
For $\phi\in C_c^\infty(U\backslash G)$, $\xi\in Y$ and $\chi$ in a suitable region of convergence we have: $$\ev_\xi \circ \Delta_{\tilde\chi}^Y (\phi)= \int_{\hat Y}  \phi(y) \nu\tilde\chi'^{-1}(y) |\omega_{\hat Y}|(y),$$ where $\tilde\chi'$ is the character of $A_Y$ such that $\nu^{-1}\tilde\chi' |\mathfrak c|^{-1}= \delta^{\frac{1}{2}}\tilde\chi$, where $\mathfrak c$ is the eigencharacter of $\omega_{\hat Y}$ under the $\AA$-action.
\end{lemma}

Here we regard $\nu$ as a function on $U\backslash G$, since it is $U$-invariant on $G$. The $B$-eigenfunction $\tilde\chi'$ on $Y$, pulled back to $\tilde Y$, can be considered as an $A\times G_\xi$-eigenfunction on $\hat Y\subset U\backslash G$ with trivial eigencharacter for $G_\xi$. Notice that the eigencharacter of $\omega_{\hat Y}$ for the action of $\GG_\xi$ is $\mathfrak d_{\GG_\xi}^{-1}$, and $\nu|_{G_\xi}=\delta_{G_\xi}$, therefore the stated functional is indeed $G_\xi$-invariant.

\begin{proof}
We notice that with the chosen normalizations we have a commutative diagram of $G$-equivariant maps:
\begin{equation*}
 \begin{CD}
  C_c^\infty(G) \otimes \nu @>o_{\xi*}>> C_c^\infty(X) \otimes \nu \\
@VVV @VVV \\
  M_c^\infty(G) @>o_{\xi*}>> M_c^\infty(X)
 \end{CD}
\end{equation*}
where: the first vertical map is $\Phi\mapsto \Phi(g) \nu(g) |\omega_G|(g)$, the first horizontal map is integration over the fibers of $o_\xi$ with respect to $\omega_\xi$, the second horizontal map is push-forward of measures and the second vertical map is the one described previously, i.e.\ $\phi\mapsto \phi |\omega_X|$. A similar diagram holds for $o_U$ (where the push-forward $o_{U*}$ of functions is with respect to $|\omega_U|$) without tensoring with $\nu$ on the upper row.

Let $\Phi_1, \Phi_2 \in C_c^\infty(G)$ with images $o_{\xi*}(\Phi_1)=\phi_1\in C_c^\infty(X)$, $o_{U*}(\Phi_2)=\phi_2\in C_c^\infty(U\backslash G)$. (In particular, $\phi_1$ is supported in $G_\xi\backslash G\subset X$.) Recall that, by definition $S_{\tilde\chi^{-1}\nu^{-1}}^Y(\phi_1|\omega_X|) = \nu \cdot  S_{\tilde\chi^{-1}\nu^{-1}}(\phi_1)$. We compute:
\begin{eqnarray*} \left< \phi_1 |\omega_X|, \Delta_{\tilde\chi}^Y \phi_2 \right>_X = \left< S_{\tilde\chi^{-1}\nu^{-1}}^Y (\phi_1 |\omega_X|), \phi_2 \right>_{U\backslash G} = \\ = \int_{U\backslash G} \left(\nu(g)\int_Y \phi_1(yg) \tilde\chi'(y) |\omega_Y|(y)\right) \phi_2(g) |\omega_{U\backslash G}|(g) = \\ = \int_G \nu(g) \int_{\tilde Y} \Phi_1(z^{-1}g)\Phi_2(g) \tilde\chi'(z^{-1}) |\omega_{\tilde Y}|(z) |\omega_G|(g)  \\ \left(\textrm{since }\int_{\tilde Y} \Phi_1(z^{-1}) |\omega_{\tilde Y}|(z) = \int_{Y} \int_{G_\xi} \Phi_1(hy) |\omega_\xi|(h) |\omega_Y|(y)\textrm{ by definition}\right)\\ = \int_G \int_{\tilde Y} \nu(zg) \Phi_1(g) \Phi_2(zg) \tilde\chi'^{-1}(z)|\omega_{\tilde Y}|(z) |\omega_G|(g) = \\ = \int_G \nu(g)\Phi_1(g) \left(\int_{\hat Y} \nu(y) \phi_2(y g) \tilde\chi'^{-1}(y) |\omega_{\hat Y}|(y)\right) |\omega_G|(g) =  \\ = \int_{G_\xi\backslash G} \left( \int_{\hat Y}  \phi_2(y g) \nu\tilde\chi'^{-1}(y)|\omega_{\hat Y}|(y)\right) \phi_1(g) |\omega_X|(g).
\end{eqnarray*}

To prove the statement on eigencharacters recall that, if we denote by $\mathfrak c_1$ the eigencharacter of $\omega_Y$ then $\tilde\chi'|\mathfrak c_1| = \tilde\chi \nu \delta^{-\frac{1}{2}}$ from the definition of $S_{\tilde\chi^{-1}\nu^{-1}}^Y$. Given the factorization $\omega_{\tilde Y}= \omega_U\wedge o_U^*(\omega_{\hat Y})$ and the fact that $\omega_U$ has a $\BB$-eigencharacter equal to $\mathfrak d_B$, it follows that $\mathfrak c_1^{-1}=$ the $\BB$-eigencharacter of $\omega_{\tilde Y}$ is equal to $\mathfrak d_B \mathfrak c$.
\end{proof}

\subsection{Type $N$ geometry}\label{sstypeN}

Before we proceed, we must discuss the case of a type $N$ reflection where additional complications arise and we will need to define other intertwining operators. Let $\YY$ be a Borel orbit of maximal rank. Based on Proposition \ref{independence} we have a distinguished $B_0$-orbit $y_0B_0$ on $Y$ (more precisely, on $\YY(\mathfrak o)$), depending, of course, on the $B_0$-orbit of the distinguished point $x_0\in \mathring \XX(\mathfrak o)$. Recall that this allows us to identify the geometric quotient $\YY/\UU$ with the quotient $\AA_Y$ of $\AA$, in a unique way up to translations by $A_0$. We have: $\AA_Y=\AA/\AA_y$, where $\AA_y$ is the stabilizer in $\BB$, modulo $\UU$, of any point $y\in \YY$.

Let $\alpha$ be a simple root such that $(Y,\alpha)$ is of type $N$. We consider the quotient map: $\YY\PP_\alpha \to \YY\PP_\alpha/\mathcal R(\PP_\alpha)\simeq\mathcal N(\TT)\backslash \PPGL_2$ and let $\BB_2$ denote the corresponding Borel subgroup of $\PPGL_2$. (As usual, the above isomorphisms are unique up to integral automorphisms.) 

The space $\XX_2:=\mathcal N(\TT)\backslash\PPGL_2$ is the space of non-degenerate quadratic forms modulo homotheties. The stabilizer in $\BB_2$ (the Borel subgroup of $\PPGL_2$) of a point on the open $\BB_2$-orbit is isomorphic to $\mathbb Z/2$. Therefore, the open $B_2$-orbits are naturally a torsor for $H^1(k,\mathbb Z/2)\simeq k^\times/(k^\times)^2$. The same group parametrizes the isomorphism classes of tori over $k$, and it is easy to see that each open $B_2$-orbit belongs to a different $\PGL_2$-orbit, with all of the possible isomorphism classes of tori appearing as the connected components of stabilizers in the different $\PGL_2$-orbits.

The quotient map: $\YY\PP_\alpha \to \YY\PP_\alpha/\mathcal R(\PP_\alpha)\simeq\mathcal N(\TT)\backslash \PPGL_2$ induces a map from the set of open $B$-orbits on $(\YY\PP_\alpha)(k)$ to the set of open $B_2$-orbits on $X_2$. The latter comes with a distinguished point (the image of $y_0B$) and the ``kernel'' of the composite map (i.e.\ the preimage of the distinguished point): $Y/U=A_Y\to\{B_2\text{-orbits on }X_2\}$ will be denoted by $A_{Y,\alpha}$; it is a subgroup of $A_Y$. Notice that this definition is specific to the root $\alpha$ and this dependence will be a serious obstacle in obtaining a useful explicit formula in those cases. 

Let $\zeta$ be a coset of $A_{Y,\alpha}$ in $A_Y$. The points of $Y$ mapping to $\zeta$ are those belonging to the preimage of a single open $B_2$-orbit on $X_2$; their set will be denoted by $Y_\zeta$. Since the connected component of the isotropy subgroup in the $\PGL_2$-orbit of that $B_2$-orbit is a torus, we can attach certain invariants to the coset $\zeta$. More precisely, we let $D(\zeta)$ be the discriminant of the splitting field of that torus. It is an element of $\mathfrak o\smallsetminus \mathfrak p^2$, well-defined up to the action of $(\mathfrak o^\times)^2$. According as $D(\zeta)\in\mathfrak o^\times$ or not we will say that ``$\zeta$ corresponds to an integral, resp.\ non-integral torus''. This is equivalent to $\zeta$ cointaining (resp.\ not containing) $\mathfrak o$-points. If $\zeta$ corresponds to an integral torus, we will say that it ``corresponds to a split, resp.\ non-split, torus'' according as $D(\zeta)\in (\mathfrak o^\times)^2$ or not.

\subsection{Spaces of morphisms and their bases.}\label{ssbases}
Let $\underline S_\chi^Y$ denote the space of morphisms $C_c^\infty(X)\to I(\chi)$ spanned by the operators $S_{\tilde\chi}^Y$ with $\tilde\chi$ extending $\chi$. Likewise, let $\underline \Delta_\chi^Y$ be the span of the $\Delta_{\tilde\chi}^Y$. The standard bases we will be using for these spaces are the bases of $S_{\tilde\chi}^Y$ and $\Delta_{\tilde\chi}^Y$. (And when we will be expressing linear maps between such spaces as matrices using this basis, unless otherwise stated.) 

Unfortunately, in certain cases it is necessary to use a different basis: 
Let $(Y,\alpha)$ be of type $N$. We defined above a subgroup $A_{Y,\alpha}$ of $A_Y$. Now we define the following basis of $\underline S_{\chi}^Y$: \label{otherbasis} It will be indexed by the set of data $(\tilde\chi,\zeta)$ where $\tilde\chi$ runs over representatives for equivalence classes of extensions of $\chi$ to $A_Y$, two of them considered equivalent if they have the same restriction to $A_{Y,\alpha}$; and $\zeta$ is a coset of $A_{Y,\alpha}$ in $A_Y$. Recall that the phrase ``extensions to $A_Y$'' is by abuse of language, since $\tilde\chi$ is not a character of $A_Y$, cf.\ \S \ref{ssmorphisms}. Let $Y_\zeta$ denote the set of points on $Y$ which map to $\zeta$ under $Y\to Y/U\simeq A_Y$. The corresponding operator will be given by the formula:
\begin{equation}
 \ev_1\circ S_{\tilde\chi,\zeta}^{Y}(\phi) := \int_{Y_\zeta} \phi(y) \tilde\chi'^{-1}(y) |\omega_Y|(y).
\end{equation}
In other words, the integral is the same as for $S_{\tilde\chi}^Y$, except that integration is restricted to the subset $Y_\zeta$. The adjoint of $S_{\tilde\chi^{-1}\nu^{-1},\zeta}^{Y}$ will be denoted by $\Delta_{\tilde\chi,\zeta}^Y$.

\subsection{Composing intertwining operators}\label{sscomposingiop}

Let $T_w$ denote the standard intertwining operators between principal series: $I(\chi)\to I({^w\chi})$, defined by rational continuation in $\chi$ of the following integral:
\begin{equation} \label{stdiop}
  T_w(\phi)(g) = \int_{U\cap w^{-1}Uw\backslash U}  \phi(\tilde w ug) du
\end{equation} (where $\tilde w\in \mathcal N(\AA)(\mathfrak o)$ is a representative for $w$). 
A main result of \cite{Sa2} was:

\begin{theorem}\label{thmSa2}
For an orbit $\YY$ of maximal rank and for almost all $\chi$ satisfying (\ref{cond1}) the space $\underline S_\chi^Y$ is mapped under composition with $T_w$ into the space $\underline S_{{^w\chi}}^{^wY}$. More precisely for $w=w_\alpha$ a simple reflection, the composition is given as follows:
\begin{itemize}
\item If $(Y,\alpha)$ is of type $G$ then $T_w  \underline S_\chi^Y = 0$.
\item If $(Y,\alpha)$ is of type $U$ or $(U,\psi)$ or $T$ then $T_w S_{\tilde\chi}^Y \sim S_{{^w\tilde\chi}}^{{^w Y}}$.
\item If $(Y,\alpha)$ is of type $N$ then $T_w S_{\tilde\chi,\zeta}^Y \sim S_{{^w\tilde\chi,\zeta}}^{Y}$ where $\tilde\chi,\zeta$ are as at the end of the previous paragraph.
\end{itemize}
Here the symbol $\sim$ denotes equality up to a non-zero rational function of $\chi$.
\end{theorem}

This theorem implies that the elements $w\in W_{(X)}\smallsetminus W_{P(X)}$ satisfying the description of Brion's Theorem \ref{Brionsdescription} are all in $W_X$. Indeed, since these elements correspond to a succession of type $U$, $T$ or $N$ reflections, the corresponding intertwining operator $T_w$ does not kill $S_\chi$. On the other hand, $w$ has a reduced decomposition $w=w_1\cdot w_2$ with $w_1\in W_X, w_2\in W_{P(X)}$, so $T_w=T_{w_1}\circ T_{w_2}$, but if $w_2\ne 1$ then $T_{w_2}$ kills $S_\chi$.

Notice that the action of the operator $T_w: I(\chi)\to I({^w\chi})$ extends to generalized sections: $I(\chi)^{-\infty} \to  I({^w\chi})^{-\infty}$ for almost every $\chi$; indeed, under the duality between $I(\chi)$ and $I(\chi^{-1})$ it is a multiple of the adjoint of $T_{w^{-1}}$ applied to $I({^w\chi^{-1}})$, and we extend this adjoint to the nonsmooth linear dual. Therefore, we may apply $T_w$ to the generalized vector $\ev_\xi\circ \Delta_{\tilde\chi}^Y \in I(\chi^{-1})^{-\infty}$ (where $\xi\in X$, see p.\ \pageref{generalizedvector}), and we may define $T_w \Delta_{\tilde\chi}^Y$ to be that morphism: $$I({^w\chi})\to C^\infty(X)$$
for which $$\ev_\xi\circ T_w \Delta_{\tilde\chi}^Y = T_w \left(\ev_\xi\circ \Delta_{\tilde\chi}^Y \right)$$
for every $\xi\in X$. We have a proportionality: $T_w \Delta_{\tilde\chi}^Y \sim \Delta_{\tilde\chi}^Y \circ T_{w^{-1}}$.

By applying the above theorem to adjoints, it immediately follows that:

\begin{corollary}\label{composition}
For an orbit $\YY$ of maximal rank and for almost all $\chi$ satisfying (\ref{cond2}) the space $\underline \Delta_\chi^Y$ is mapped under composition with $T_w$ into the space $\underline \Delta_{{^w\chi}}^{^wY}$. More precisely for $w=w_\alpha$ a simple reflection, the composition is given as follows:
\begin{itemize}
\item If $(Y,\alpha)$ is of type $G$ then $T_w  \underline \Delta_\chi^Y = 0$.
\item If $(Y,\alpha)$ is of type $U$ or $(U,\psi)$ or $T$ then $T_w \Delta_{\tilde\chi}^Y \sim \Delta_{{^w\tilde\chi}}^{{^w Y}}$.
\item If $(Y,\alpha)$ is of type $N$ then $T_w \Delta_{\tilde\chi,\zeta}^Y \sim \Delta_{{^w\tilde\chi,\zeta}}^{Y}$ where $\tilde\chi,\zeta$ are as in \S\ref{otherbasis}.
\end{itemize}
\end{corollary}

In particular, for $\YY=\mathring \XX$ (where we omit the superscript $^Y$ from the notation) we get that $T_w\Delta_{\tilde\chi}$ is not identically zero if and only if $w\in [W/W_{P(X)}]$ (coset representatives of minimal length), that $T_w\underline\Delta_\chi = \underline \Delta_\chi$ if and only if $w\in W_X$ (for generic $\chi$) and that for almost all points in the support of $C_c^\infty(X)^K$ as an $\mathcal H(G,K)$-module, the space of eigenvectors is spanned by the $K$-invariant vectors in the images of $\underline \Delta_\chi$ (for suitable $\chi$), since for almost all points all other intertwining operators can be obtained by composing elements of $\underline\Delta_\chi$ by some $T_w$. (Knop's action is transitive on the set of orbits of maximal rank.) Therefore, the goal of computing eigenvectors of the unramified Hecke algebra now becomes:

\begin{quote}
 \emph{Let $\phi_{K,\chi}$ denote a $K$-invariant vector in $I(\chi)$. Compute its image under $\Delta_{\tilde\chi}$, for every $\tilde\chi$.}
\end{quote}

A priori, this will provide us with all the eigenvectors for almost every point in the support. In \S \ref{secHecke} we will examine, among other things, to what extent this gives a complete answer for every point.

\subsection{Schwartz space and Fourier transforms} \label{ssSchwartz}

We now explain a precise choice of intertwining operators which will make our computations easier and more transparent. In fact, we change the setting slightly and replace unramified principal series by functions on $U\backslash G$.  We recall that there exists a notion of Schwartz space and Fourier transform for this space \cite{BK}:

If $\GG=\SSL_2$ then it is simply the space of compactly supported, locally constant functions on $k^2 \supset k^2\smallsetminus\{0\} \simeq U\backslash G$ with the equivariant Fourier transform. (The Fourier transform obtained by identifying $k^2$ with its dual via a non-degenate symplectic form.) Since our spaces have models over $\mathfrak o$, we can impose some canonical conditions on our Fourier transforms, namely: 
The measure on $k^2$ is such that $\mathfrak o^2$ has measure 1 (this is the correct choice only under our simplifying assumption that $k$ is unramified over $\QQ_p$); we use a symplectic form which in some basis for $\mathfrak o^2$ has the form: $\omega((x,y),(v,w))= xw-yv$; and we use an additive character $\psi$ on $k$ whose conductor is the ring of integers $\mathfrak o$. (We choose and fix such a character.) Using this data, Fourier transform on $k^2$ is given by the formula:
\begin{equation}\label{Fouriertransform}
 \mathcal (Ff)(v,w)= \int_{k^2} f(x,y) \psi(-xw+yv) dx dy.
\end{equation}

In general, one defines for every simple reflection $w_\alpha$ in the Weyl group an equivariant Fourier transform: $\mathcal F_{w_\alpha}: L^2(U\backslash G)\to L^2(U\backslash G)$ as follows: Let $\PP_\alpha$ be the parabolic of semisimple rank one corresponding to the simple root $\alpha$. Then we have $[\PP_\alpha,\PP_\alpha]/\UU_{P_\alpha}\simeq\SSL_2$ and $\UU\backslash[\PP_\alpha,\PP_\alpha]\simeq \mathbb A^2\smallsetminus\{0\}$. Notice that the integral structure on $\GG$ defines an integral structure on $\mathbb A^2\smallsetminus\{0\}$, therefore after fixing the former isomorphism the latter is determined up to multiplication by an element of $\mathfrak o^\times$. Then one defines $\mathcal F_{w_\alpha}$ to be the unique equivariant transform which is equal to the $\SSL_2$-equivariant Fourier transform on the restriction of a continuous function in $L^2(U\backslash G)$ to $U\backslash [P_\alpha,P_\alpha]$. It is known that the equivariant Fourier transforms compose as follows: $\Ff_{w} \circ \Ff_{w'}
= \Ff_{ww'}$, hence they define a unitary action of the Weyl group on $L^2(U\backslash G)$. Braverman and Kazhdan define the Schwartz space to be the smallest subspace which contains $C_c^\infty(U\backslash G)$ and is closed under all Fourier transforms. The Schwartz space comes with a canonical $K$-invariant vector $c_0$, which is stable under all Fourier transforms and whose restriction to the $[L_\alpha,L_\alpha]$-orbit of $U\cdot 1$, for each simple root $\alpha$, is equal, under the above identifications, to the characteristic function of $\mathfrak o^2$ in $k^2$. 

We define the action of $A$ on $L^2(U\backslash G)$ as follows, so that it is \emph{unitary}:
$$ (L_af)(Ug) = \delta^{-\frac{1}{2}}(a) f(Uag).$$
This way we have:
\begin{equation}\label{equivariance}
 \mathcal F_w (L_a f) = L_{wa} \mathcal F_w f.
\end{equation}

We can extend Fourier transforms uniquely to \emph{tempered} distributions (i.e.\ distributions on the Schwartz space) and tempered generalized functions. For almost every $\chi$, the elements of the principal series $I(\chi)$ are tempered, and by the equivariance properties the transform $\mathcal F_w$ maps: $I(\chi)\to I(^w\chi)$. We will denote $\mathcal F_w$, restricted to $I(\chi)$, by the non-script symbol $F_w$.

Equivalently, consider the surjective map: $\mathcal S(U\backslash G) \to I(\chi)$ given by the rational continuation of the integral:
$$ f \mapsto \int_A L_a f \,\,\chi^{-1}(a) da.$$
(We normalize $\Vol(A_0)=1$.) The Fourier transforms $\mathcal F_w$ and certain intertwining operators $F_w$ between the $I(\chi)'s$ fit into the commutative diagram:
\begin{equation*}
\begin{CD}
\mathcal S(U\backslash G) @>\int>>  & I(\chi) \\
@V{\Ff_w}VV & @VV{F_w}V \\
\mathcal S(U\backslash G) @>\int>> & I({^w\chi}),
\end{CD}
\end{equation*}
and $F_w$ is a multiple of the ``standard'' intertwining operator $T_w$ discussed previously. Notice that the $F_w$'s satisfy: $F_w^*=F_{w^{-1}}$ (recall that the smooth dual of $I(\chi)$ is $I(\chi^{-1})$). Thinking of them as adjoints, we may also apply them to the generalized vectors (i.e.\ to the space $I(\chi)^{-\infty}$), as we did with the $T_w$'s.

The reason for working with Schwartz space and Fourier transforms, instead of principal series and their intertwining operators, is that computations will be much easier there; in particular, we will be able to avoid all the mysterious cancellations and simplifications that one discovers a posteriori in all other variants of the Casselman-Shalika method.

\subsection{An Iwahori-invariant vector of small support} The starting point for (all variants of) the Casselman-Shalika method is the observation that it is easy to compute the value of the functional $\Delta_{\tilde\chi}^{Y}$ on translates of an Iwahori-invariant vector of small support. More precisely, let $\Phi_J\in \mathcal S(U\backslash G)$ denote the characteristic function of $U\cdot J$. 

\begin{lemma} \label{Iwahorivolume}
For $x\in A_X^+$ we have: 
$$\Delta_{\tilde\chi}^{Y} (\Phi_J)(x)= \left\{\begin{array}{lr} 0 & \operatorname{if }\YY\ne\mathring\XX \\ \Vol(J)\cdot \tilde\chi\delta^{-\frac{1}{2}} (x) & \operatorname{otherwise.}\end{array}\right.$$
\end{lemma}

\begin{proof}
$$ \Delta_{\tilde\chi}^{Y} (\Phi_J)(x) = \frac{{\left< \Delta_{\tilde\chi}^{Y} (\Phi_J), 1_{xJ} |\omega_X|\right>}_{X}}{|\omega_X|(xJ)} = \frac{{\left<\Phi_J, S_{\tilde\chi^{-1}\nu^{-1}}^{Y} (1_{xJ}|\omega_X|) \right>}_{U\backslash G}}{|\omega_X|(xJ)} = $$ $$=\frac{|\omega_{U\backslash G}|(U\backslash UJ)}{|\omega_X|(xJ)} \cdot \nu(x) \cdot \ev_{U1}\circ S_{\tilde\chi^{-1}\nu^{-1}}^{Y} (1_{xJ}).$$
By Axiom \ref{Iwahori} we have: $xJ\subset xB_0$. Therefore, from the definition of $S_{\tilde\chi^{-1}\nu^{-1}}^{Y}$ (\ref{defofS}) we get: $\ev_{U1} \circ S_{\tilde\chi^{-1}\nu^{-1}}^{Y} (1_{xJ})= 0$ unless $\YY=\mathring \XX$, in which case: $$\nu(x)\ev_{U1} \circ S_{\tilde\chi^{-1}\nu^{-1}}^{Y} (1_{xJ})= |\omega_X|(xJ) \tilde\chi\delta^{-\frac{1}{2}}(x).$$ 
\end{proof}

\subsection{Functional equations} Using Fourier transforms, and the fact that $F_{w^{-1}}^*=F_w$, Corollary \ref{composition} can be restated as follows:

\emph{For every $w\in W$ there is a rational family of linear operators: $\underline b_w^{Y}(\chi): \underline\Delta_\chi^Y\to \underline\Delta_{^w\chi}^{^w Y}$ such that $F_w \Delta = \underline b_w^{Y}(\chi) \Delta$ for $\Delta\in \underline\Delta_\chi^Y$.} Moreover, for $Y=\mathring X$ the operator $\underline b_w(\chi)$ is non-zero if and only if $w\in [W/W_{P(X)}]$, and for $w$ a simple reflection we have explicit bases which make the matrix of $\underline b_w^Y$ diagonal. Since $F_{w_1}\circ F_{w_2} = F_{w_1 w_2}$, the $\underline b_w^Y$ satisfy the cocycle relations: 
\begin{equation}\label{cocycle}
 \underline b_{w_1 w_2}^Y(\chi) = \underline b_{w_1}^{^{w_2}Y}(^{w_2}\chi) \underline b_{w_2}^Y(\chi).
\end{equation}
As a matter of notation, \emph{we will be identifying $\underline b_w^{Y}(\chi)$ with its matrix in the bases $[\Delta_{\tilde\chi}^Y]_{\tilde\chi}, [\Delta_{^w\tilde\chi}^{^wY}]_{\tilde\chi}$}.

Suppose that we knew the matrices $\underline b_w(\chi)$ (recall that when we ignore the orbit as a superscript we mean $\mathring X$). Then it would be easy to compute values of $\Delta_{\tilde\chi}$ applied to every vector of the form $\mathcal F_{w^{-1}} \Phi_J$ as follows:

$$[\Delta_{\tilde\chi}(\mathcal F_{w^{-1}} \Phi_J)(x)]_{\tilde\chi}= [F_w \Delta_{\tilde\chi} (\Phi_J) (x)]_{\tilde\chi}= \underline b_w(\chi) \cdot [\Delta_{^w\tilde\chi}^{^w \mathring X} (\Phi_J) (x)]_{\tilde\chi}.$$

This vanishes unless $^w \mathring \XX = \mathring \XX$, i.e.\ unless $w\in W_X$, in which case from the previous lemma we get: 
\begin{equation}\label{PhiJ}
[\Delta_{\tilde\chi}(\mathcal F_{w^{-1}} \Phi_J)(x)]_{\tilde\chi}=  \underline b_w(\chi) \cdot [\Vol(J)\cdot\,^w\tilde\chi\delta^{-\frac{1}{2}} (x)]_{\tilde\chi}  
\end{equation}

Therefore, the computation of the Hecke eigenvectors will be straightforward if we can perform the following two steps:

\begin{enumerate}\label{sssteps}
\item Compute the functional equations $F_w [\Delta_{\tilde\chi}^{Y}]_{\tilde\chi} = \underline b_w^Y(\chi) [\Delta_{^w\tilde\chi}^{^w Y}]_{\tilde\chi}$.
\item Express a non-zero $K$-invariant Schwartz function as a linear combination of the functions $\mathcal F_{w^{-1}} \Phi_J$, $w\in W$.
\end{enumerate}

\begin{remark}
The second step depends only on the group $G$ and is independent of the particular case that we are examining (i.e.\ independent of the subgroup $H$ or the possible ``additive'' character $\Psi$), therefore it suffices to do it once and for all. It would therefore be possible to retrieve it, for example, from the work of Casselman on zonal spherical functions. However, we construct a $K$-invariant vector in an independent way, which in the sequel will help us avoid all the mysterious cancellations which are usually found in the literature and will lead us directly to simple formulas. Notice also that the requirement that a linear combination of the $\mathcal F_{w^{-1}} \Phi_J$'s be non-zero is automatically satisfied for any non-trivial linear combination. This follows from (\ref{PhiJ}) (applied to any choice of $\HH$) and the linear independence of the characters ${^w\chi}$, for $\chi$ in general position.
\end{remark}


\section{Construction of a $K$-invariant vector}

\subsection{The case of $\SSL_2$} \label{ssKinv} In this section we perform the second step of \S \ref{sssteps}.

We start by considering the case of $\SL_2$. Then $\Phi_J=1_{\mathfrak {p \times o^\times}} = 1_{\mathfrak {p \times o}}- 1_{\mathfrak {p \times p}}$. Since $1_{\mathfrak{p\times p}}$ is $K$-invariant we see that $\Phi_J\equiv 1_{\mathfrak {p \times o}}$ up to $K$-invariants and the equivariant Fourier transform of $\Phi_J$ is, up to $K$-invariants, equal to:
$$\mathcal F 1_{\mathfrak {p \times o}} =  q^{-1} 1_{\mathfrak {o \times p^{-1}}}= L_{\tiny{\left(\begin{array}{cc} \varpi^{-1} & \\ & \varpi \end{array}\right)}} 1_{\mathfrak {p \times o}}.$$

Hence $$L_{\tiny{\left(\begin{array}{cc} \varpi^{-1} & \\ & \varpi \end{array}\right)}} \Phi_J - \mathcal F \Phi_J \equiv L_{\tiny{\left(\begin{array}{cc} \varpi^{-1} & \\ & \varpi \end{array}\right)}} 1_{\mathfrak {p \times o}} - \mathcal F 1_{\mathfrak {p \times o}}\equiv 0$$ up to $K$-invariants; equivalently, the vector:
\begin{equation}\label{Kinv}
\Phi_J - L_{\tiny{\left(\begin{array}{cc} \varpi & \\ & \varpi^{-1} \end{array}\right)}} \mathcal F \Phi_J
\end{equation}
is $K$-invariant.

\subsection{The general case} We now return to an arbitrary $G$. To simplify notation, for each co-root $\check \alpha$ and each $\Phi\in\mathcal S(U\backslash G)$ we will denote $L_{e^{\check\alpha}(\varpi)}\Phi$ by $e^{\check\alpha}\Phi$. This is consistent with the fact that the image of $L_{e^{\check\alpha}(\varpi)}\Phi$ in $I(\chi)$ under the natural morphism (integration with respect to the $A$-action against the character $\chi^{-1}$) is $e^{\check\alpha}(\chi) := \chi(e^{\check\alpha}(\varpi))$ times the image of $\Phi$.

For each simple root $\alpha$ let $K_\alpha\subset K$ denote the inverse image, through the reduction map, of $\PP_\alpha(\mathbb F_q)$. Evidently, $K$ is generated by all the $K_\alpha$ since they contain the Iwahori and representatives for the simple reflections in the Weyl group. Using the $\SL_2$ case that we examined above, it is now easy to show:

\begin{proposition}\label{PhiK}
The vector
$$\Phi_K:=\Vol(J)^{-1}\sum_{w\in W} \prod_{\alpha>0, w^{-1}\alpha<0} (-1) e^{\check\alpha} \mathcal F_w \Phi_J$$
is $K$-invariant.
\end{proposition}

\begin{proof}
It suffices to show that for any simple positive root $\beta$ it is $K_\beta$-invariant.

For the parabolic $P_\beta$ we denote by $W_\beta$ the Weyl group of its Levi, and by $[W/W_\beta]$ the canonical set of representatives of $W/W_\beta$ cosets consisting of representatives of minimal length. Notice that if $w=w' w_\beta$, with $w'\in [W/W_\beta]$, then $\{ \alpha| \alpha>0, w^{-1}\alpha<0\} =  \{\alpha| \alpha>0, w'^{-1}\alpha<0\}\cup \{w'\beta\}$. Hence:

$$\Vol(J)\Phi_K= \sum_{w\in [W/W_\beta]} \prod_{\alpha>0, w^{-1}\alpha<0} (-1)e^{\check\alpha} \left( \mathcal F_w \Phi_J - e^{w\check\beta}\mathcal F_{w w_\beta} \Phi_J\right) = $$
$$ = \sum_{w\in [W/W_\beta]} \prod_{\alpha>0, w^{-1}\alpha<0} (-1)e^{\check\alpha} \mathcal F_w \left( \Phi_J - e^{\check\beta} \mathcal F_{w_\beta}\Phi_J\right), $$
where we used (\ref{equivariance}), and the latter is $K_\beta$-invariant.
\end{proof}

Notice that we can write: $\prod_{\alpha>0,w^{-1}\alpha<0} e^{\check\alpha} (\chi) = e^{{\check\rho}-w{\check\rho}}({\chi})$. Combining this with (\ref{PhiJ}) and changing $w$ to $w^{-1}$ inside the sum we get:

\begin{theorem} \label{maintheorem}
For almost every $\chi$, a basis of eigenvectors of $\mathcal H(G,K)$ on $C^\infty(X)$ with eigencharacter corresponding to $\chi$ is given by the formula:
\begin{equation} \label{mainformula}
[\Omega_{\tilde\chi} (x)]_{\tilde\chi} := [\Delta_{\tilde\chi}(\Phi_K) (x)]_{\tilde\chi} = e^{{\check\rho}}(\chi)\delta^{-\frac{1}{2}}(x)  \sum_{w\in W_X} \sigma(w) e^{-{\check\rho}}({^w\chi}) \underline b_w(\chi) \cdot [{^w\tilde\chi} (x)]_{\tilde\chi} 
\end{equation}
for $x\in A_X^+$, where $\sigma(w)$ is the sign of $w$ as an element of $W$ and the matrices $\underline b_w(\chi)$ are given by the functional equations:
\begin{equation}\label{fe3}
F_{w} [\Delta_{\tilde\chi}]_{\tilde\chi} = \underline b_w(\chi) [\Delta_{{^w\tilde\chi}}]_{\tilde\chi}.
\end{equation}
\end{theorem}

\begin{remarks} \begin{enumerate}
 \item There is some abuse of notation here (which we will consistently practice!), in that the expression $^w\tilde\chi \delta^{-\frac{1}{2}}(x)$ should be taken as a whole and not as a product of two factors, since $^w\tilde\chi$ is not, in general,  a character of $A_X$.
 \item In the case that $(Y,\alpha)$ is never of type $N$ (for $Y$ of maximal rank and $\alpha$ a simple root) then we actually know from Corollary \ref{composition} that the matrices $\underline b_w^Y(\chi)$ are diagonal. If in addition $\tilde\chi$ is unramified we may write: $\tilde\chi(x_{\check\lambda})=e^{-{\check\lambda}}(\tilde\chi)$ where $x_{\check\lambda}=e^{-{\check\lambda}}(\varpi)\in A_X^+$ for any uniformizer $\varpi$, and the formula becomes more pleasant:
\begin{equation}
\Omega_{\tilde\chi} (x_{\check\lambda}) = e^{\check\rho}(\chi){e^{-{\check\lambda}}(\delta^\frac{1}{2})}  \sum_{w\in W_X} \sigma(w) b_w(\tilde\chi) e^{-{\check\rho}+{\check\lambda}}({^w\tilde\chi})
\end{equation}
where the $b_w(\tilde\chi)$ are now scalars given by the functional equations: $F_w \Delta_{\tilde\chi} = b_w(\tilde\chi) \Delta_{^w\tilde\chi}$.
\end{enumerate}
\end{remarks}

The following will be useful later:

\begin{lemma}\label{omegatrivial}
 We have $\Omega_{\delta^\frac{1}{2}}(x)=1$ for every $x$.
\end{lemma}

\begin{proof}
 By our assumption that $\XX$ carries an eigen-volume form, and by the form of the morphisms $S_{\tilde\chi}$, it follows that $S_{\nu^{-1}\delta^{-\frac{1}{2}}}$ furnishes a morphism: $C_c^\infty(X)\to \nu^{-1}$ or, tensoring by $\nu$, a morphism: $M_c^\infty(X)\to 1$ (the trivial representation). Therefore its dual, as a morphism: $I(\delta^\frac{1}{2})\to C^\infty(X)$, factors through the trivial representation. Moreover, the image of $\Phi_K$ via the morphism $\mathcal S(U\backslash G)\to I(\chi)$ is finite at $\delta^\frac{1}{2}$ (see (\ref{sphvector}) below). Therefore, $\Omega_{\delta^\frac{1}{2}}$ is a constant function, and since the coefficient of the trivial character in (\ref{mainformula}) is $1$, we have $\Omega_{\delta^\frac{1}{2}}=1$.
\end{proof}

\begin{remark}
 The lemma does not apply to the case of $C^\infty(X,\mathcal L_\Psi)$, with non-trivial $\Psi$, since there is no $G$-eigenmeasure valued in $\mathcal L_\Psi$ in that case.
\end{remark}

\subsection{Normalization} \label{ssnormalization} 
It is good to keep track of how our eigenfunctions are normalized, so we collect here the properties of $\Omega_{\tilde\chi}=\Delta_{\tilde\chi}(\Phi_K)$:
\begin{description}
 \item[1. Iwahori normalization:] Let $\phi_{J,\chi}\in I(\chi)$ be the vector with $\phi_{J,\chi}|_K= 1_J$. Then $\phi_{J,\chi}$ is the image of $\Phi_J$ and $\Delta_{\tilde\chi}(\phi_{J,\chi})(x_0)=\Vol(J)=\frac{1}{(K:J)}$.
 \item[2. ``Spherical vector'' normalization:] Let $\phi_{K,\chi}\in I(\chi)$ be the unramified vector with $\phi_{K,\chi}(b k) = \chi\delta^\frac{1}{2}(b)$ for every $b\in B$. To determine the value of $\Delta_{\tilde\chi}(\phi_{K,\chi})$ we must determine its relation with the image of $\Phi_K$. Let $\int_\chi$ denote the natural morphism: $\mathcal S(U\backslash G) \to I(\chi)$. We claim: 
\begin{equation}\label{sphvector}
\phi_{K,\chi} = Q^{-1}\prod_{\check\alpha>0}\frac{1-q^{-1}e^{\check\alpha}}{1-e^{\check\alpha}}(\chi) \int_\chi \Phi_K 
\end{equation}
where $Q=\frac{\Vol(K)}{\Vol(Jw_lJ)}=\prod_{\check\alpha>0}\frac{1-q^{-1}e^{\check\alpha}}{1-e^{\check\alpha}}(\delta^\frac{1}{2})$.
To prove this we need a result which we will prove later, together with a result of Casselman and Shalika. Let $\Delta_\chi$ refer to the morphism corresponding to the Whittaker model, i.e.\ $\HH=\UU^-$, the opposite unipotent subgroup, with a generic character $\Psi$. We prove later that: 
$$\Delta_\chi(\Phi_K)(x_{{\check\lambda}}) = {e^{-{\check\lambda}}(\delta^\frac{1}{2})} e^{\check\rho}(\chi)\sum_W \sigma(w) e^{-\check\rho+{\check\lambda}}(^w\chi).$$
The corresponding formula of \cite[Theorem 5.4]{CS} is:
$$\Delta^{CS}_\chi(\phi_{K,\chi})(x_{{\check\lambda}}) = \prod_{\check\alpha>0} \frac{1-q^{-1}e^{\check\alpha}}{1-e^{\check\alpha}}(\chi){e^{-{\check\lambda}}(\delta^\frac{1}{2})} e^{\check\rho}(\chi)\sum_W \sigma(w) e^{-\check\rho+{\check\lambda}}(^w\chi)$$
where $\Delta_\chi^{CS}$ is normalized so that $\Delta_\chi^{CS}(\phi_{J^-,\chi})=1$, where $\phi_{J^-,\chi}$ denotes the element in $I(\chi)$ with $\phi_{J^-,\chi}|_K=1_{B_0J^-}$, where $J^-$ is the Iwahori subgroup corresponding to the opposite Borel. Comparing the two normalizations, since clearly in this case $\Delta_\chi(\phi_{J^-,\chi})= (U_0:U_1) \Delta_\chi(\phi_J)$ (here $U_1$ is the kernel in $U_0$ of the reduction map to $\UU(\FF_q)$) and also since $Q= (U_0:U_1) \Vol(J)$, our claim follows. (Notice that in \cite{CS} the spherical subgroup has not been put in opposite position from the Borel and the representatives for $K$-orbits are all dominant. Therefore, one has to substitute $e^{w_l {\check\lambda}}$ for $a$ in \cite[Theorem 5.4]{CS} -- where $w_l$ is the longest element of the Weyl group -- to get the correct formula.)

For simplicity, we state the following corollary only for the case of trivial arithmetic multiplicity (see Theorem \ref{thmmult}), i.e.\ there is only one open $B$-orbit and $A_X^*\hookrightarrow A^*$:
\begin{corollary}
Assume that there is a unique open $B$-orbit. The spherical function whose value at $x_0$ is equal to the integral (in the domain of convergence for $\Delta_\chi$, and by rational continuation otherwise) of $\phi_{K,\chi}$ over $H$ (with the given normalizations) is: 
\begin{eqnarray} \Delta_{\chi}(\phi_{K,\chi})(x_{{\check\lambda}}) = Q^{-1} {e^{-{\check\lambda}}(\delta^\frac{1}{2})}  \prod_{\check\alpha>0}\frac{1-q^{-1}e^{\check\alpha}}{1-e^{\check\alpha}}(\chi) \cdot \nonumber \\
 \cdot e^{\check\rho}(\chi) \sum_{w\in W_X} \sigma(w) b_w(\chi) e^{-{\check\rho}+{\check\lambda}}({^w\chi}).
\end{eqnarray}
\end{corollary}

 \item[3. ``Basic vector'' normalization:] In their study \cite{BK} of $\mathcal S(U\backslash G)$ Braverman and Kazhdan introduce a ``basic vector'' $c_0\in \mathcal S(U\backslash G)^K$ which generates $\mathcal S(U\backslash G)$ over $\mathcal H(A,A_0)\otimes\mathcal H(G,K)$. It is related to $1_{UK}$ as follows:
$$ 1_{UK} = \prod_{\check\alpha>0} (1-q^{-1} e^{\check\alpha}) c_0$$
where $e^{\check\alpha}$ denotes the corresponding element of $\mathcal H(A,A_0)$ (acting via the normalized left $A$-action on $U\backslash G$) (see \cite[eq.\ (3.15), (3.22)]{BK}). Therefore $\int_\chi c_0 = \prod_{\check\alpha>0} (1-q^{-1} e^{\check\alpha})(\chi) \phi_{K,\chi}$ and from \eqref{sphvector}:
\begin{equation}\label{BKfunction}
c_0= \frac{Q^{-1}\Phi_K}{\prod_{\check\alpha>0} (1-e^{\check\alpha})}.
\end{equation}
Notice that, indeed, using the expression of Proposition \ref{PhiK}, one immediately verifies that $c_0$ is invariant under Fourier transforms.
\end{description}

At this point we can also discuss the relationship between the intertwining operators $F_w$ and the ``classical'' intertwining operators $T_w$ (\ref{stdiop}). By \cite[Theorem 3.1]{C} one has:
\begin{equation}\label{Twphi}T_w \phi_{K,\chi} = \prod_{\check\alpha>0, w\check\alpha<0}\frac{1-q^{-1}e^{\check\alpha}}{1-e^{\check\alpha}}(\chi)\phi_{K,\,{^w}\chi}.
\end{equation}
On the other hand, we have $$F_w (\phi_{K,\chi}) = \prod_{\check\alpha>0, w\check\alpha <0} \frac{1-q^{-1}e^{\check\alpha}}{1- q^{-1}e^{-\check\alpha}}(\chi) \phi_{K,\, {^w\chi}}$$
and therefore: 
\begin{equation}\label{relationFwTw}
 F_w = \prod_{\check\alpha>0, w\check\alpha<0} \frac{1-e^{\check\alpha}}{1- q^{-1}e^{-\check\alpha}}(\chi) T_w.
\end{equation}


\section{Functional equations}\label{secfe}

\subsection{Reduction to rank one} \label{ssred} In this section we compute the proportionality factors in the functional equations:
$$F_{w} [\Delta_{\tilde\chi}^Y]_{\tilde\chi} = \underline b_w^Y(\chi) [\Delta_{{^w\tilde\chi}}^{^wY}]_{\tilde\chi}$$
or equivalently:
$$F_{w} [S_{\tilde\chi^{-1}\nu^{-1}}^Y]_{\tilde\chi} = \underline b_w^Y(\chi) [S_{{^w\tilde\chi^{-1}\nu^{-1}}}^{^wY}]_{\tilde\chi}$$
when $w=w_\alpha$ is a simple reflection and $\YY$ is an orbit of maximal rank. For general $w$, the functional equations follow from the cocycle relations (\ref{cocycle}). Assume $(\YY,\alpha)$ not of type $G$. We have already exhibited basis elements of $\underline\Delta^Y_\chi, \underline\Delta^{^wY}_{^w\chi}$ which map to a multiple of each other under Fourier transform. As a matter of notation, the (scalar) quotient of $F_w\Delta_{\tilde\chi}^Y$ by $\Delta_{{^w\tilde\chi}}^{^wY}$, in cases $U$, $(U,\psi)$ and $T$ will be denoted by $b_w^Y(\tilde\chi)$, and the quotient of $F_w\Delta_{\tilde\chi,\zeta}^{Y}$ by $\Delta_{{^w\tilde\chi},\zeta}^{^wY}$ in case $N$ will be denoted by $b_w^{Y}(\tilde\chi,\zeta)$. Hence, it is enough to compute these scalar quotients for all simple reflections $w=w_\alpha$.

\begin{remark}
 Though our notation below is adapted to the cases $U$ and $T$, the same discussion holds with obvious modifications for cases $N$ and ($U,\psi$).
\end{remark}

To reduce to the case of $\SSL_2$, we proceed as follows:

We can compose $S_{\tilde\chi^{-1}\nu^{-1}}^Y$ with the ``restriction'' morphism to get a sequence of $P_\alpha$-morphisms:
$$ C_c^\infty(X) \xrightarrow{S_{\tilde\chi^{-1}\nu^{-1}}^Y} C^\infty_\temp(U\backslash G) \xrightarrow{\rm{Res}} C^\infty_\temp(U\backslash P_\alpha).$$ We temporarily denote by $S_{\tilde\chi^{-1}\nu^{-1}}'^Y$ the composition. (The dependence on $\alpha$ has been suppressed from the notation.)

We have a commutative diagram of $P_\alpha$-morphisms:
\begin{equation}\begin{CD}
 C_\temp^\infty(U\backslash G) @>{\rm{Res}}>> C_\temp^\infty(U\backslash P_\alpha) \\
@V{\mathcal F_{w_\alpha}}VV @VV{\mathcal F_{w_\alpha}}V \\
 C_\temp^\infty(U\backslash G) @>{\rm{Res}}>> C_\temp^\infty(U\backslash P_\alpha).
\end{CD}\end{equation}

Therefore the morphisms $S_{\tilde\chi^{-1}\nu^{-1}}'^Y$ satisfy \emph{the same functional equations with respect to $w_\alpha$} as the morphisms $S_{\tilde\chi^{-1}\nu^{-1}}^Y$.

Clearly, from the definition of $S_{\tilde\chi^{-1}\nu^{-1}}^Y$, the morphism $S_{\tilde\chi^{-1}\nu^{-1}}'^Y$ factors through restriction of elements of $C_c^\infty(X)$ to $\overline{YP_\alpha}$. To compute the scalar proportionality factors, it is enough to restrict our $P_\alpha$-morphisms to the $P_\alpha$-stable subspace of $C_c^\infty(X)$ of those elements whose restriction to $\overline{YP_\alpha}$ is supported in $YP_\alpha$. For those, we can factorize the morphism as:
\begin{eqnarray*}
  S_{\tilde\chi^{-1}\nu^{-1}}'^Y\otimes 1: C_c^\infty(YP_\alpha)\otimes \nu \to
C_c^\infty((\HH_\alpha\backslash \LL_\alpha)(k), \delta_{G_\xi\cap U_{P_\alpha}\backslash U_{P_\alpha}}) \otimes \nu\to \\ \to \Ind_{B_\alpha}^{L_\alpha}(\chi^{-1}\nu^{-1}\delta^\frac{1}{2})\otimes \nu \simeq \Ind_{B_\alpha}^{L_\alpha}(\chi^{-1}\delta^\frac{1}{2})\subset C^\infty_\temp(U_\alpha\backslash L_\alpha)
\end{eqnarray*} according to the decomposition of eigenforms as in \S \ref{eigenforms}. 
Here the notation is as follows: $\xi$ is a point on $YP_\alpha$ (to be chosen more carefully later), and $\HH_\alpha$ denotes the image of $\GG_\xi\cap \PP_\alpha$ in $\LL_\alpha=\PP_\alpha/\UU_{P_\alpha}$. The first arrow is ``integration along $U_{P_\alpha}$-orbits''.

 We denote the second arrow in the above sequence by $S_{\tilde\chi^{-1}\nu^{-1}}^{Y,\alpha}$, fix invariant measures (chosen compatibly, as we did in \S \ref{adjoints}) on $U_\alpha\backslash L_\alpha$ and $(\HH_\alpha\backslash \LL_\alpha)(k)$ (this admits an invariant measure by inspection of the quasi-affine $\PPGL_2$-spherical varieties), and denote by $\Delta_{\tilde\chi}^{Y,\alpha}$ the adjoint:
$$\Delta_{\tilde\chi}^{Y,\alpha}: \mathcal S(U_\alpha\backslash L_\alpha) \to C^\infty((\HH_\alpha\backslash \LL_\alpha)(k), \delta^{-1}_{G_\xi\cap U_{P_\alpha}\backslash U_{P_\alpha}})\otimes \nu^{-1} \xrightarrow{\sim} $$ $$ \xrightarrow{\sim} C^\infty((\HH_\alpha\backslash \LL_\alpha)(k), \delta^{-1}_{G_\xi\cap U_{P_\alpha}\backslash U_{P_\alpha}}\nu^{-1}).$$

Notice that this factors through the dual of $\Ind_{B_\alpha}^{L_\alpha}(\chi^{-1}\delta^\frac{1}{2})$, which is canonically identified with  $\Ind_{B_\alpha}^{L_\alpha}(\chi e^\alpha\delta^{-\frac{1}{2}})$. The $L_\alpha$-morphisms $\Delta_{\tilde\chi}^{Y,\alpha}$ \emph{satisfy the same functional equations as the morphisms $\Delta_{\tilde\chi}^Y$.} Let $\hat \YY_\alpha:= \UU_\alpha\backslash \BB_\alpha \HH_\alpha\subset \UU_\alpha\backslash \LL_\alpha$. In analogy with Lemma \ref{lemmaadj}, we have the following result (notice that the line bundle $\mathcal L_{ \delta^{-1}_{G_\xi\cap U_{P_\alpha}\backslash U_{P_\alpha}}\nu^{-1}}$ over $\HH_\alpha\backslash \LL_\alpha (k)$ comes with a canonical trivialization of its pull-back to $L_\alpha$, therefore it makes sense to ``evaluate at 1''):
\begin{lemma}\label{lemmaadj2}
For $\phi\in C_c^\infty(U_\alpha\backslash L_\alpha)$ and $\chi$ in a suitable region of convergence we have: 
$$\ev_1 \circ \Delta_{\tilde\chi}^{Y,\alpha} (\phi)= \int_{\hat Y_\alpha} \phi(y) \nu\tilde\chi'^{-1}(y) |\omega_{\hat Y_\alpha}|(y)$$ for a suitable $\AA\times\HH_\alpha$-eigenform $\omega_{\hat Y_\alpha}$ on $\hat\YY_\alpha$ with eigencharacter: $$\mathfrak c_\alpha\times \mathfrak d_{G_\xi\cap U_{P_\alpha}\backslash U_{P_\alpha}}\mathfrak n$$ (unnormalized action of $\AA$) where $\mathfrak c_\alpha$ and $\tilde\chi'$ satisfy: $\nu^{-1}\tilde\chi'|\mathfrak c_\alpha|^{-1}=e^\alpha\delta^{-\frac{1}{2}}\tilde\chi$. 

For $(\YY,\alpha)$ of type $N$ the same holds for $\ev_1 \circ \Delta_{\tilde\chi,\zeta}^{Y,\alpha}$,
provided that $\xi\in Y_\zeta$.
\end{lemma}

\begin{proof}
 The proof is completely analogous to that of Lemma \ref{lemmaadj}. The only thing to notice is that in case $N$, if we consider the image of $\tilde Y_\alpha = (\BB_\alpha\HH_\alpha)(k)$ in $Y$ under the action map $g\mapsto \xi \cdot g^{-1}$, it belongs to $Y_{\alpha,\zeta}$. 
\end{proof}

For the computation, we pick the point $\xi$ as follows: In cases $U$, $(U,\psi)$ and $T$ we choose $\xi\in \YY(\mathfrak o)$, and more precisely we choose $\xi$ in the distinguished $B_0$-orbit (cf.\ Proposition \ref{independence}). The choice of $\xi$ in case $N$ will be discussed in \S \ref{sscaseN}. (We caution the reader that in case $N$ it will not be possible, in general, to choose $\xi$ in the distinguished $B_0$-orbit on $Y$ (cf.\ \S \ref{sscaseN}), and for that reason the function $\nu\tilde\chi'^{-1}(y)$ of Lemma \ref{lemmaadj2} on $\hat Y_\alpha$ will not, in general, be equal to one at $U_\alpha 1$.) To simplify notation, we temporarily denote the distribution
$ \ev_1\circ\Delta_{\tilde\chi}^{Y,\alpha}$ on $U_\alpha\backslash L_\alpha$ (respectively $ \ev_1\circ\Delta_{\tilde\chi,\zeta}^{Y,\alpha}$ in case $N$)  by $\Delta'_{\tilde\chi}$ (resp.\ $\Delta'_{\tilde\chi,\zeta}$). Hence:
$$\Delta'_{\tilde\chi} : \mathcal S(U_\alpha\backslash L_\alpha)\to \CC$$
is an $A\times H_\alpha$-eigendistribution, with eigencharacter (considering the unnormalized action of $A$):
$$\chi^{-1}e^{-\alpha}\delta^\frac{1}{2} \times \delta_{G_\xi\cap U_{P_\alpha}\backslash U_{P_\alpha}}\nu.$$

To complete the reduction to $\SSL_2$, it suffices to consider the factorization: $\LL_\alpha=\ZZ_\alpha \LL_\alpha'$ (where $\ZZ_\alpha$ is the center of $\LL_\alpha$) and to notice that the distribution $\Delta_{\tilde\chi}$ is smooth along $Z_\alpha$-orbits with respect to Haar measure on $Z_\alpha$. More precisely, we can pick a small subgroup-neighborhood $N$ of $1 \in Z_\alpha$ such that:
\begin{itemize}
 \item A neighborhood of $U_\alpha\backslash L_\alpha'$ is isomorphic to $U_\alpha\backslash L_\alpha'\times N$ under the action map.
 \item The restriction of $\Delta'_{\tilde\chi}$ to this neighborhood can be written as: Haar measure on $N$ $\times$ a distribution $\Delta_{\tilde\chi}^2$ on $U_\alpha\backslash L_\alpha'$.
\end{itemize}
Since the kernel of Fourier transform is supported on $U_\alpha\backslash L_\alpha'\times U_\alpha\backslash L_\alpha'$, it is enough for the functional equations to compute the Fourier transform of $\Delta_{\tilde\chi}^2$ (having fixed a Haar measure on $N$ which is independent of $\chi$). We are free to pick the Haar measure on $N$ (and hence vary $\Delta_{\tilde\chi}^2$ up to a constant which is independent of $\tilde\chi$), but we need to make the same choices when relating the distributions $\Delta_{\tilde\chi}^2$ coming from two orbits $\YY,\ZZ$ of maximal rank in the same $\PP_\alpha$-orbit. In particular, the corresponding distributions $\Delta_{\tilde\chi}^2$ will be related to each other in the way discussed in \S \ref{twoorbits}.

\subsection{The result}\label{ssferesult} The computation has now been reduced to the case of $\SSL_2$. To present the functional equations, we need to introduce some extra language: If ${\check\lambda}:\Gm\to\AA_Y$ is a cocharacter, its image is an algebraic subgroup $\MM_{\check\lambda}$ of $\AA_Y$ we say that a character $\tilde\chi$ of $A_Y$ is \emph{${\check\lambda}$-unramified} if it is unramified on $M_{\check\lambda}$. 

Finally, in the split case $T$ we introduce some extra data coming from the geometry of the spherical variety. More precisely, $\YY\PP_\alpha$ is the union of $\YY$ and two $k$-rational divisors $\DD, \DD'$. They define valuations $\check v_D, \check v_{D'}: k(Y)\to \mathbb Z$ which, restricted to $k(Y)^{(B)}$ define homomorphisms: $\varchi(Y)\to\mathbb Z$. These, in turn, can be identified with coweights ${\check v}_D,{\check v}_{D'}$ into $\AA_Y$. It is known \cite[\S 1]{Lu} that ${\check v}_D=-^{w_\alpha}{\check v}_{D'}$ and that ${\check v}_D+{\check v}_{D'}\equiv \check\alpha$ (where $\equiv$ stands for equality with the image of $\check\alpha$ in $\varchi(Y)^*$). As we will see later, the image of ${\check v}_D$ in $\AA_Y$ coincides with the image of the stabilizer in $\BB$ modulo $\UU$ of a point on $\DD$.

We are ready now to state the functional equations for $\Delta_{\tilde\chi}^Y$ (resp.\ $\Delta_{\tilde\chi,\zeta}^Y$) and for $w=w_\alpha$, a simple reflection. The data that we need is: The type of $(\YY,\alpha)$ ($U$, $(U,\psi)$, $T$ split, $T$ non-split or $N$); in case $T$ split, the coweights $\check v_D, \check v_{D'}$ (as we saw, knowledge of one implies knowledge of the other) and the modular character $\delta_{(U_{P_\alpha})_\xi}$, as a character of the stabilizer of $\xi$ in $P_\alpha$; in case $N$, whether the coset $\zeta$ corresponds to an integral/non-integral, split/non-split torus.

Finally, we need to know the character $\tilde\chi$. We remind that $\tilde\chi$ is not really a character of $A_Y$ but rather a character of its preimage in $\AA(\bar k)$ which satisfies the condition:
$$ \tilde\chi\nu\delta^{\frac{1}{2}}|_{B_\xi}=\delta_{B_\xi}.$$
In any case, for a co-root $\check\alpha$ the quantity $e^{\check\alpha}(\tilde\chi)=e^{\check\alpha}(\chi)$ makes sense. It also makes sense in case $T$ to consider the expression $e^{\check v_D}(\tilde\chi\nu\delta^{\frac{1}{2}} \delta_{\left(U_{P_\alpha}\right)_\xi}^{-1})$, as we will explain in \S \ref{sscaseTsplit}.

\begin{theorem} \label{fetheorem}
 In cases $U$, $(U,\psi)$ and $T$ the functional equations read as follows:
\begin{description}
 \item[Case U:]
We have 
$$F_{w_\alpha}\Delta_{\tilde\chi}^Y = \Delta_{^{w_\alpha}\tilde\chi}^{^{w_\alpha}Y} \cdot \left\{\begin{array}{l} -e^{-\check\alpha}\frac{1-q^{-1}e^{-\check\alpha}}{1-e^{-\check\alpha}}(\chi) \textrm{ if }\alpha\textrm{ lowers }Y \\ \\
-e^{-\check\alpha}\frac{1-e^{\check\alpha}}{1-q^{-1}e^{\check\alpha}}(\chi) \textrm{ if }\alpha\textrm{ raises }Y.
\end{array}\right. $$

\item[Case U,$\psi$:]
We have
$$ F_{w_\alpha}\Delta_{\tilde\chi}^Y = \Delta_{^{w_\alpha}\tilde\chi}^{Y}.$$

\item[Case T, split:]
We have
$$ F_{w_\alpha}\Delta_{\tilde\chi}^Y = \Delta_{^{w_\alpha}\tilde\chi}^{Y} \cdot \frac{(1-q^{-1}e^{-{\check v}_D})(1-q^{-1}e^{-{\check v}_{D'}})}{(1-e^{{\check v}_D})(1-e^{{\check v}_{D'}})} (\tilde\chi\nu \delta^{\frac{1}{2}}\delta_{\left(U_{P_\alpha}\right)_\xi}^{-1})  $$
if $\tilde\chi$ is $\check v_D$-unramified, and
$$ F_{w_\alpha}\Delta_{\tilde\chi}^Y = \Delta_{^{w_\alpha}\tilde\chi}^{Y} \cdot e^{- m\check\alpha}(\chi) $$
if it is $\check v_D$-ramified of conductor $\mathfrak p^m$.

\item[Case T, non-split:]
We have
$$ F_{w_\alpha}\Delta_{\tilde\chi}^Y = \Delta_{^{w_\alpha}\tilde\chi}^{Y} \cdot \frac{1-q^{-1}e^{-\check\alpha}(\chi)}{1-q^{-1} e^{\check\alpha}(\chi)}$$
if $\tilde\chi$ is $\frac{\check\alpha}{2}$-unramified, and
$$ F_{w_\alpha}\Delta_{\tilde\chi}^Y = \Delta_{^{w_\alpha}\tilde\chi}^{Y} \cdot e^{-\check\alpha}(\chi)$$
otherwise.

\item[Case $N$:] Here we have four cases, according as the coset $\zeta$ corresponds to a split or non-split, integral or non-integral torus, and in each of these cases we have to distinguish between $\tilde\chi$ being $\frac{\check\alpha}{2}$-ramified or not. The functional equations for this case are given in \S \ref{sscaseN}.
\end{description}

\end{theorem}

\subsection{Preliminaries of the computation} We are going to compute the proportionality constant case-by-case. We will use many times the following fact: Consider usual Fourier transform of functions and distributions on $k$, with respect to our fixed character $\psi$ and the measure $dx$ normalized as: $dx(\mathfrak o)$=1. The Fourier transform of a function $f$ on $k$ will be denoted by $\hat f$. By Tate's thesis, the Fourier transform of the distribution $\chi(x) d^\times x$ (where $dx^\times$ denotes some fixed multiplicative Haar measure on $k^\times$) is equal to:
$$\frac{1-q^{-1}\chi^{-1}(\varpi)}{1-\chi(\varpi)} \chi^{-1}(x) |x| d^\times x = \frac{1-q^{s-1}}{1-q^{-s}} |x|^{1-s}d^\times x$$
if $\chi(x)=|x|^s$ is unramified, and:
$$q^{-m} \tau(\chi) \cdot \chi^{-1}(x) |x| d^\times x$$
if $\chi$ is 
a ramified character of conductor $\mathfrak p^m$, where $\tau(\chi)$ is the Gauss sum: $$\tau(\chi)=\sum_{\epsilon \in \mathfrak o^\times/(1+\mathfrak p^m)} \chi(\varpi^{-m}\epsilon)\psi(\varpi^{-m}\epsilon).$$ 

Indeed, we know a priori from equivariance properties that the Fourier transform of $\chi(x) d^\times x$ must be a multiple of $\chi^{-1}(x) |x| d^\times x$, and that multiple can be computed from the formula: $\left< f, g \right> = \left< \hat f, \hat g\right>$. Tate computes: $\left<1_{\mathfrak o}, \chi d^\times x\right>=:\zeta(1_{\mathfrak o}, \chi |\bullet|) = (1- \chi(\varpi))^{-1}$ if $\chi$ is unramified. If $\chi=\eta \cdot |\bullet|^s$ is ramified, with $\eta$ a unitary character, we compute the zeta integral of $f_m:=\psi(x) 1_{\mathfrak p^{-m}}$ whose Fourier transform is $\hat f_m = q^m 1_{1+\mathfrak p^m}$ and we get: $\left< f_m, \chi d^\times x\right>=\zeta(f_m,\eta |\bullet|^{s}))= q^{ms} \tau(\eta) \Vol(\mathfrak p^m)$ and $\left<\hat f_m, \chi^{-1}|\bullet| d^\times x \right>=\zeta(\hat f_m, \eta^{-1} |\bullet|^{-s+1})= q^m \Vol(\mathfrak p^m)$. 

Notice that here we are using our simplifying assumption that $k$ is unramified over $\QQ_p$; these equations would have to be modified otherwise.

\subsection{\textbf{Case U}} 
Let us identify $\LL_\alpha'$ with $\SSL_2$ and $\UU_\alpha\backslash\LL_\alpha'$ with $\mathbb A^2\smallsetminus\{0\}$ over $\mathfrak o$. The distribution $\Delta_{\tilde\chi}^2$ (defined at the end of \S \ref{ssred}) is here an $(A\cap L_\alpha')\times U_\alpha$-eigendistribution. Based on Lemma \ref{lemmaadj2}, depending on whether $Y$ is raised or lowered by $\alpha$, the distribution $\Delta_{\tilde\chi}^2$ on $U_\alpha\backslash L_\alpha'$ is, up to an integral change of coordinates and up to a common scalar multiple one of the following two:
$$ \Delta_{1, \chi}= |x|^{s_1} dx dy \textrm{ and } \Delta_{2, \chi}=  |y|^{s_2} \delta_0(x) dy $$
where $s_1,s_2\in \CC$ is such that $q^{s_1+1}=q^{s_2}=e^{-\check\alpha}(\chi\nu)=e^{-\check\alpha}(\chi)$, and $\delta_0$ denotes the delta measure at $\{0\}\subset k$.

Now we see immediately that $$F_{w_\alpha} \Delta_{1, \chi}= \frac{1-q^{s_1}}{1-q^{-s_1-1}} \Delta_{2, {^{w_\alpha}\chi}}= \frac{1-q^{-1}e^{-\check\alpha}(\chi)}{1-e^{\check\alpha}(\chi)} \Delta_{2, {^{w_\alpha}\chi}}$$
and $$F_{w_\alpha} \Delta_{2, \chi}= \frac{1-q^{s_2}}{1-q^{-s_2-1}} \Delta_{1, {^{w_\alpha}\chi}}= \frac{1-e^{-\check\alpha}(\chi)}{1-q^{-1}e^{\check\alpha}(\chi)} \Delta_{1, {^{w_\alpha}\chi}}$$

Hence, for $(\alpha,Y)$ of type U, we have: 
\begin{equation}
\boxed{b_{w_\alpha}^Y(\tilde\chi)= \left\{\begin{array}{l} -e^{-\check\alpha}\frac{1-q^{-1}e^{-\check\alpha}}{1-e^{-\check\alpha}}(\chi) \textrm{ if }\alpha\textrm{ lowers }Y \\ \\
-e^{-\check\alpha}\frac{1-e^{\check\alpha}}{1-q^{-1}e^{\check\alpha}}(\chi) \textrm{ if }\alpha\textrm{ raises }Y.
\end{array}\right. }
\end{equation}

\begin{example} \label{groupcase}
Let $\GG'=\GG\times \GG$, $\HH= \GG^\diag$, and consider the spherical variety $\XX=\GG=\HH\backslash \GG'$ of $\GG'$ with action $(g_1,g_2)\cdot x = g_1^{-1} x g_2$. Let us choose a Borel subgroup $\BB'=\BB^-\times \BB$, so that $\HH\BB'$ is open. The little Weyl group is generated by $w_\alpha w_{\tilde\alpha}$, where $-\alpha$ is a simple root of $\BB^-$ in the first copy of $\GG$ and $\tilde\alpha$ the corresponding root of $\BB$ in the second copy. The admissible characters $\chi'$ are of the form: $\chi^{-1}\otimes\chi$. (We put $\chi^{-1}$ in the first copy so that the eigenfunction $\Omega_{\chi^{-1}\otimes\chi}$ is proportional to the matrix coefficient $\left<\pi(g) v, \tilde v\right>$ with $\pi=I(\chi)$ and $v,\tilde v$ the unramified vectors.) Notice that $w_{\tilde\alpha}$ lowers the open orbit to some orbit $\YY_\alpha$, which $w_\alpha$ then raises back to the open orbit. Hence we compute: $e^{-\check\alpha} e^{\check{\tilde\alpha}}\cdot b_{w_\alpha w_{\tilde\alpha}}(\chi') = \frac{1-q^{-1}e^{-
\check\alpha}}{1-e^{-\check\alpha}}(\chi)\frac{1-e^{-\check\alpha}}{1-q^{-1}e^{-\check\alpha}}(\chi^{-1})$ and the formula reads:
$$ \Omega_{\chi'} (g_{\check\lambda}) ={e^{-{\check\lambda}}(\delta^\frac{1}{2})} \sum_W \prod_{\check\alpha>0. w\check\alpha<0} \frac{1-q^{-1}e^{-\check\alpha}}{1-e^{-\check\alpha}}\frac{1-e^{\check\alpha}}{1-q^{-1}e^{\check\alpha}}(\chi) e^{{\check\lambda}}(^w\chi)$$
which up to a factor which is independent of ${\check\lambda}$ is equal to:
$$e^{-{\check\lambda}}(\delta^\frac{1}{2}) \sum_{w\in W} \prod_{\check\alpha>0} \frac{1-q^{-1}e^{\check\alpha}}{1-e^{\check\alpha}} e^{{\check\lambda}} (^w\chi).$$
This is, of course, Macdonald's formula for zonal spherical functions which was reproven by Casselman in \cite[Theorem 4.2]{C}. (To compare with the result in \emph{loc.cit.}, recall that our ${\check\lambda}$ is anti-dominant.)

\end{example}

\subsection{\textbf{Case U,$\psi$}} Here the distribution $\Delta_{\tilde\chi}^2$ has eigencharacter $\psi$ under the action of $U$. Under the same identifications as above we have, up to an integral change of coordinates:
$$ \Delta^2_{\tilde\chi} = |x|^s\psi^{-1}\left(\frac{y}{x}\right) dx dy$$
where $s\in \CC$ is such that $q^{s+1}=e^{-\check\alpha}(\chi)$.

Here we see that $$F_{w_\alpha} |x|^s \psi^{-1}\left(\frac{y}{x}\right) dx dy = |x|^{-s-2}\psi^{-1}\left(\frac{y}{x}\right) dx dy$$ in other words:
$$ F_{w_\alpha} \Delta^2_\chi = \Delta^2_{^{w_\alpha}\chi}$$
and hence:
\begin{equation}\label{casepsi}
\boxed{b^Y_{w_\alpha}(\chi)=1}.
\end{equation}

\begin{example}
Let $H=U$ with the standard Whittaker character. Then $W_X=W$, all simple roots are of type (U,$\psi$), and we have:
$$\Omega_\chi (g_{\check\lambda}) = {e^{-{\check\lambda}}(\delta^\frac{1}{2})} e^{\check\rho}(\chi)\sum_W \sigma(w) e^{-\check\rho+{\check\lambda}}(^w\chi).$$
This is the Shintani-Casselman-Shalika formula \cite{CS}. (Again, to compare with the result in \emph{loc.cit.}, recall that our ${\check\lambda}$ is anti-dominant.)
\end{example}

\begin{example}
Let $G=\GL_{2n}$ and let $H$ be a conjugate of the Shalika subgroup such that $HB$ is open, equipped with the Shalika character $\Psi$. The Shalika subgroup is the semidirect product of $\GL_n$, embedded diagonally in the $\GL_n\times\GL_n$ parabolic, with the unipotent radical of this parabolic; the Shalika character is the complex character $(g,u)\mapsto \psi(\tr u)$ of this subgroup. In this case, if we enumerate $w_1,\dots,w_{2n-1}$ the simple reflections in the Weyl group of $G$, the little Weyl group is generated by the elements $w_i w_{2n-i}$ ($1\le i \le n-1$) and $w_n$, hence can be identified with the Weyl group of the subgroup $\Sp_{2n}(\CC)$ of the dual group $\check G=\GL_{2n}(\CC)$. The character $\chi$ is of the form $\chi_0\otimes 
{^{w_0}\chi_0^{-1}}$, where $\chi_0$ is a character of the maximal torus of $\GL_n$, and $w_0$ denotes the longest Weyl element of $\GL_n$. The Knop action of the element $w_i w_{2n-1}$ is by lowering $\mathring\XX$ to some orbit $\YY_i$ and then raising $\YY_i$ back to the open orbit, so the corresponding factor is: $e^{\check\alpha} e^{\check{\tilde\alpha}}(\chi)b_{w_\alpha w_{\tilde\alpha}}(\chi) = \frac{1-q^{-1}e^{-\check\alpha}}{1-e^{-\check\alpha}}\frac{1-e^{\check\alpha}}{1-q^{-1}e^{\check\alpha}}(\chi)$ (where $\alpha,\tilde\alpha$ are the roots corresponding to $w_i$, $w_{2n-i}$, and we have used the fact that $\chi$ is of the aforementioned form). On the other hand, the pair $(\mathring X,w_n)$ is of type ($U, \psi$), hence the corresponding factor is $e^{\check\alpha}(\chi) b_\alpha(\chi)=\frac{1-e^{\check\alpha}}{1-e^{-\check\alpha}} (\chi)$. Notice that $\chi$ can be considered as an element of the maximal torus of $\Sp_{2n}(\CC)\subset\check G$, and the coroots $\check\alpha$ which appear in 
the above factors are the roots of $\Sp_{2n}(\CC)$. Hence the unramified Shalika function is given by:

$$\Omega_\chi (g_{\check\lambda}) = {e^{-{\check\lambda}}(\delta^\frac{1}{2})}\sum_{W_X} \prod_{\alpha\in \Phi_{\Sp_{2n}}^S, \alpha>0, w\alpha<0} \frac{1-q^{-1}e^{-\check\alpha}}{{1-e^{-\check\alpha}}} \frac{1-e^{\check\alpha}}{1-q^{-1}e^{\check\alpha}}(\chi) $$  $$\cdot\prod_{\alpha\in\Phi_{\Sp_{2n}}^L, \alpha>0, w\alpha<0}  \frac{1-e^{\check\alpha}}{1-e^{-\check\alpha}}(\chi)  e^{{\check\lambda}}({^w\chi})$$
which, up to a factor independent of ${\check\lambda}$, is:
$$e^{-{\check\lambda}}(\delta^\frac{1}{2})\sum_{W_X} \sigma(w)\prod_{\alpha\in \Phi_{\Sp_{2n}}^S}(1-q^{-1}e^{\check\alpha}) e^{-{\check\rho}+{\check\lambda}}({^w\chi})$$
($\sigma(w)$ denoting the sign of $w$ as an element of $W$), the formula of \cite[Theorem 2.1]{Sa1}. The symbols $\Phi_{\Sp_{2n}}^S$ and $\Phi_{\Sp_{2n}}^L$ denote the sets of short and long roots of $\Sp_{2n}$, respectively. (Once more, one needs to substitute ${\check\lambda}$ by $w_{\check\lambda}{\check\lambda}$ to arrive at the formula of \emph{loc.cit.} Notice that here $e^{-{\check\lambda}}=e^{w_l {\check\lambda}}$ where $w_l$ is the longest element of the Weyl group, so the signs of the coweights inside of the $W_X$-sum will be inverted.)
\end{example}

\subsection{Case $T$ split} \label{sscaseTsplit} The goal of the discussion that follows is to identify the distribution $\Delta_{\tilde\chi}^2$ in case $T$, split. The following is an easy fact:

\begin{lemma}
 In case $T$, the group $\AA\cap\HH_\alpha$ is central in $\LL_\alpha$.
\end{lemma}

\begin{proof}
 Its image in $\LL_\alpha/\mathcal R(\LL_\alpha)\simeq \PPGL_2$ is trivial.
\end{proof}

This implies that the antidiagonal copy: $g\mapsto (g^{-1},g)$ of this group is the kernel of the morphism: $\AA\times \HH_\alpha\to\Aut(\UU_\alpha\backslash\LL_\alpha)$. Denote the quotient by $\overline{\AA\times\HH_\alpha}$, i.e.:
$$\overline{\AA\times\HH_\alpha}:=(\AA\times\HH_\alpha)/ (\AA\cap\HH_\alpha)^{\rm{adiag}} \hookrightarrow \Aut(\UU_\alpha\backslash\LL_\alpha).$$

Let $\check v_D, \check v_{D'}:\Gm\to \AA_Y$ denote the cocharacters defined in \S \ref{ssferesult}. We will prove:

\begin{proposition} \begin{enumerate}
 \item There is a natural way to identify $\check v_D, \check v_{D'}$ with cocharacters $\widetilde{\check v_D}, \widetilde{\check v_{D'}}$ into $\overline{\AA\times\HH_\alpha}$ such that:
\begin{itemize}
\item[i)] Their compositions with the map $\overline{\AA\times\HH_\alpha}\to\AA_Y$ are equal to $\check v_D,\check v_{D'}$, respectively.
\item[ii)] Their images are automorphisms which preserve the subvariety $\UU_\alpha\backslash\LL_\alpha'$. 
\item[iii)] Under an $\mathfrak o$-isomorphism: $\UU_\alpha\backslash\LL_\alpha'\simeq \mathbb A^2_{x,y}$ and up to an integral change of coordinates, the automorphisms $\check v_D(a), \check v_{D'}(a)$ restrict to $\UU_\alpha\backslash\LL_\alpha'$ as: $(x,y)\mapsto (ax,y)$ and $(x,y)\mapsto(x,ay)$.
\end{itemize}
 \item The distribution $\Delta_{\tilde\chi}^2$ is an eigendistribution for $\widetilde{\check v_D}(k^\times), \widetilde{\check v_{D'}}(k^\times)$, with eigencharacter equal to the restriction of $\tilde\chi^{-1}e^{-\alpha}\delta^\frac{1}{2} \otimes \delta_{G_\xi\cap U_{P_\alpha}\backslash U_{P_\alpha}}\nu$
 to their image.
\end{enumerate}
\end{proposition}

Notice that the statement about the eigencharacter makes sense, because we have $\tilde\chi^{-1}e^{-\alpha}\delta^\frac{1}{2} = \delta_{G_\xi\cap U_{P_\alpha}\backslash U_{P_\alpha}}\nu$ on $A\cap H_\alpha$. Indeed, the right hand side is equal to $\delta_{U_{P_\alpha}} \delta_{\left(U_{P_\alpha}\right)_\xi}^{-1}$ and in this case we have: $\left(U_{P_\alpha}\right)_\xi=U_\xi$. Therefore, the equality follows from (\ref{cond2}). By a slight abuse of notation, we will be denoting the pull-back of this character to $k^\times$ by $\tilde\chi^{-1}\nu^{-1}\delta^{-\frac{1}{2}} \delta_{\left(U_{P_\alpha}\right)_\xi}\circ e^{\check v_D}$ (and similarly for $D'$).

\begin{proof}
 Let $z\in \DD$ be any point. Consider the series of quotients: $\BB\to \AA\to \AA_Y$. Let $\BB'\subset\BB$, $\AA'\subset\AA$ be the preimages of $\check v_D(\Gm)$. We claim:

\emph{The subgroup $\BB'$ coincides with the stabilizer in $\BB$ of the $\UU$-orbit of $z$.}

An element $b\in\BB$ is in the stabilizer of $z\UU$ if and only if $z$ and $z\cdot b$ cannot be separated by a regular $\BB$-eigenfunction on $\DD$. The pair $(\YY,\alpha)$ being of type $T$, every $\BB$-eigenfunction on $\DD$ extends to a $\BB$-eigenfunction on $\YY\PP_\alpha$. By definiton of $\check v_D$, the eigencharacters of eigenfunctions which restrict non-trivially to $\DD$ are precisely those which are orthogonal to $\check v_D$. This proves the claim.

Under the identification: $\YY \PP_\alpha \simeq (\PP_\alpha)_\xi\backslash \PP_\alpha$ (following from a choice of point $\xi$), we may let $\hat\DD$ denote the orbit of $\AA\times\HH_\alpha$ on $\UU\backslash \PP_\alpha=\UU_\alpha\backslash \LL_\alpha$ corresponding to the $\BB$-orbit $\DD$ on $\YY\PP_\alpha$. Then the above claim translates to the following: 

\emph{For every $\hat z\in \hat\DD$ the stabilizer of $z$ in ${\AA\times\HH_\alpha}$ is isomorphic to $\AA'$.}

Clearly, this isomorphism does not depend on the choice of $z$. Since the antidiagonal copy of $\AA\cap\HH$ acts trivially, we get:

\emph{For every $\hat z\in \hat\DD$ the stabilizer of $z$ in $\overline{\AA\times\HH_\alpha}$ is canonically isomorphic to $\check v_D(\Gm)$.}

Hence the cocharacter $\widetilde{\check v_D}$ of the proposition. 

We have a morphism: $\UU_\alpha\backslash\LL_\alpha\to \LL_\alpha'\backslash\LL_\alpha$. Therefore the stabilizer, in $\AA\times\LL_\alpha$, of any point on $\UU_\alpha\backslash\LL_\alpha'$ stabilizes $\UU_\alpha\backslash\LL_\alpha'$, as well. Hence claim (i).

Identify $\UU_\alpha\backslash\LL_\alpha'\simeq\mathbb A^2_{x,y}\smallsetminus\{0\}$. For a suitable choice of $\mathfrak o$-isomorphisms, we can identify the non-open orbits of $\rm{stab}_{\overline{\AA\times\HH_\alpha}}(\UU_\alpha\backslash\LL_\alpha')$ with the two coordinate axes. Let $f$ be a $\BB$-eigenfunction on $\HH_\alpha\backslash\LL_\alpha$ with eigencharacter $e^\alpha$. Equivalently, $f\circ\iota$ (where $\iota$ denotes inversion in the group) can be thought of as an $\AA\times\HH_\alpha$-eigenfunction on $\UU_\alpha\backslash\LL_\alpha$ with eigencharacter $e^\alpha\times 1$. From the definitions, $f\circ\widetilde{\check v_D}(a)=a f$. Since $\widetilde{\check v_D}(a)$ stabilizes the points of a coordinate axis and preserves the other axis, it follows that $\widetilde{\check v_D}(a)$ is the automorphism: $(x,y)\mapsto (ax,y)$ or $(x,y)\mapsto (y,ax)$. 

The statement about the eigencharacter follows from the eigencharacter of $\Delta^{Y,\alpha}_{\tilde\chi}$, cf.\ Lemma \ref{lemmaadj2}.
\end{proof}

Therefore, up to an integral change of variables we have on $\mathbb A^2_{x,y}$: 

$$ \Delta_{\tilde\chi}^2 = \eta_D(x) \eta_{D'} (y) d^\times x d^\times y $$
where $\eta_D = \tilde\chi\nu\delta^{\frac{1}{2}}\delta_{\left(U_{P_\alpha}\right)_\xi}^{-1}\circ\check v_D$, and similarly for $\eta_{D'}$. Its Fourier transform is, therefore, as follows:

\begin{enumerate}
\item If $\tilde\chi$ is $\check v_{D}$--unramified (equivalently, since $\check v_D+\check v_{D'}=\check\alpha$, if it is $\check v_{D'}$-unramified):
$$ F_{w_\alpha} \Delta_{\tilde\chi}^2= 
 \frac{(1-q^{-1}e^{-{\check v}_D})(1-q^{-1}e^{-{\check v}_{D'}})}{(1-e^{{\check v}_D})(1-e^{{\check v}_{D'}})} (\tilde\chi\nu \delta^{\frac{1}{2}}\delta_{\left(U_{P_\alpha}\right)_\xi}^{-1}) \Delta_{^{w_\alpha}\tilde\chi}^2$$
Consequently:
\begin{equation} \label{tsplitetaunramified}
\boxed{b_{w_\alpha}^Y(\tilde\chi)=\frac{(1-q^{-1}e^{-{\check v}_D})(1-q^{-1}e^{-{\check v}_{D'}})}{(1-e^{{\check v}_D})(1-e^{{\check v}_{D'}})} (\tilde\chi\nu \delta^{\frac{1}{2}}\delta_{\left(U_{P_\alpha}\right)_\xi}^{-1}).}
\end{equation}

\begin{remark}
The vanishing of the terms $(1-e^{{\check v}_D})(\tilde\chi\nu \delta^{\frac{1}{2}}\delta_{\left(U_{P_\alpha}\right)_\xi}^{-1})$ and $(1-e^{{\check v}_{D'}})(\tilde\chi\nu \delta^{\frac{1}{2}}\delta_{\left(U_{P_\alpha}\right)_\xi}^{-1})$ in the denominator is precisely the condition for the orbit of $D$ (resp.\ $D'$) to support a morphism from $I(\tilde\chi)$, in other words a condition for $\Delta_{\tilde\chi}^{\mathring D}$ (resp.\ $\Delta_{\tilde\chi}^{\mathring D'}$) to be defined. This, of course, was expected -- see \cite{Ga} or \cite[\S 4.6]{Sa2}.
\end{remark}

\item If $\tilde\chi$ is $\check v_D$-ramified:
$$ F_{w_\alpha} \Delta_{\tilde\chi}^2 = q^{-2m}\tau(\eta_D)\tau(\eta_{D'}) \eta_D(-1) \Delta_{^{w_\alpha}{\tilde\chi}}^2.$$

We can simplify the above expression as follows: From the fact that $\hat{\hat f} (x) = f(-x)$ and the formula $\widehat{\eta(x)d^\times x}=q^{-m}\tau(\eta)\eta^{-1}(x)|x|d^\times x$ for a ramified character $\eta$ of conductor $\mathfrak p^m$ it follows that $\tau(\eta)\tau(\eta^{-1}) = \eta(-1) q^m$. Applying this to $\eta=\eta_D$ and taking into account that $\eta_{D'}= \tilde\chi\nu\delta^{\frac{1}{2}}\delta_{\left(U_{P_\alpha}\right)_\xi}^{-1}\circ e^{\check\alpha} \cdot \eta_D^{-1}$ $\Rightarrow \tau(\eta_{D'})=\chi\nu\delta^{\frac{1}{2}}\delta_{\left(U_{P_\alpha}\right)_\xi}^{-1}\circ e^{\check\alpha}(\varpi^{-m}) \cdot \tau(\eta^{-1}) = q^m e^{-m\check\alpha}(\chi) \tau(\eta^{-1})$ we get:

$$ F_{w_\alpha} \Delta_{\tilde\chi}^2 = e^{-m\check\alpha}(\chi) \Delta_{^{w_\alpha}{\tilde\chi}}^2.$$

Hence:
\begin{equation} \label{tsplitetaramified}
\boxed{b_{w_\alpha}^Y(\tilde\chi)=e^{- m\check\alpha}(\chi) .}
\end{equation}
\end{enumerate}

\begin{example}\label{tripleproduct1}
Let $\GG=(\PPGL_2)^3$ and $\HH$ a $\GG(\mathfrak o)$-conjugate of the diagonal copy of $\PPGL_2$ such that $\HH\BB$ is open.  For instance, if we take for $\BB$ three copies of the group of upper triagonal matrices, then $\HH$ can be the diagonal copy of $\PPGL_2$ conjugated by the element $\left(\left( \begin{array}{cc} 1 \\ & 1\end{array}\right),\left( \begin{array}{cc} & 1 \\ -1 & \end{array}\right),\left( \begin{array}{cc} 1 \\ 1 & 1\end{array}\right) \right)$. Here the little Weyl group is the whole Weyl group and $A_X^*=A^*$. Let $\alpha_1,\alpha_2, \alpha_3$ denote the simple positive roots in each copy respectively. Then we easily see that $(\mathring X,\alpha_i)$ is of type $T$ and that there are three orbits $D_1, D_2, D_3$ of codimension one, with ${\check v}_{D_i}= \frac{1}{2} (-\check\alpha_i + \sum_{j\ne i} \check\alpha_j)$ and the $D_j, j\ne i$ being the two orbits in the $P_{\alpha_i}$-orbit of $\mathring X$. Moreover, $\delta_{B_\xi}$ is trivial. Therefore:
$$b_{\alpha_i}(\chi)= \frac{(1-q^{-\frac{1}{2}}e^{\frac{-\check\alpha_i+\check\alpha_j-\check\alpha_k}{2}})(1-q^{-\frac{1}{2}}e^{\frac{-\check\alpha_i-\check\alpha_j+\check\alpha_k}{2}})} {(1-q^{-\frac{1}{2}}e^{\frac{\check\alpha_i+\check\alpha_j-\check\alpha_k}{2}})(1-q^{-\frac{1}{2}}e^{\frac{\check\alpha_i-\check\alpha_j+\check\alpha_k}{2}})} (\chi)$$
and, up to a factor independent of ${\check\lambda}$ (see also Example \ref{tripleproduct2}), $\Omega_\chi(g_{\check\lambda})$ is equal to:
$$ {e^{-{\check\lambda}}(\delta^\frac{1}{2})} \sum_W \sigma(w) (1-q^{-\frac{1}{2}}e^{\frac{\check\alpha_1+\check\alpha_2-\check\alpha_3}{2}}) (1-q^{-\frac{1}{2}}e^{\frac{\check\alpha_1-\check\alpha_2+\check\alpha_3}{2}})\cdot$$ $$\cdot(1-q^{-\frac{1}{2}}e^{\frac{-\check\alpha_1+\check\alpha_2+\check\alpha_3}{2}})(1-q^{-\frac{1}{2}}e^{\frac{\check\alpha_1+\check\alpha_2+\check\alpha_3}{2}}) e^{-{\check\rho}+{\check\lambda}}(^w\chi). $$
\end{example}

\subsection{\textbf{Case T non-split}} The analysis here is simpler than in the split case, because all one-dimensional non-split tori of $\LL_\alpha$ are contained in $\LL_\alpha'$ and therefore $\HH_\alpha\cap \LL_\alpha'$ is a spherical subgroup of $\LL_\alpha'$. Under an $\mathfrak o$-identification of $\UU_\alpha\backslash \LL_\alpha'$ with $\mathbb A^2_{x,y}\smallsetminus\{0\}$, the group $\HH_\alpha\cap \LL_\alpha'$ is the special orthogonal group of a non-degenerate, non-split binary quadratic form $Q$, and therefore:
$$\Delta_{\tilde\chi}^2 = \eta(Q(x,y)) d x d y$$
where $\eta = \tilde\chi\delta^{-\frac{1}{2}} \circ e^{\frac{\check\alpha}{2}}$. By our assumption that the action of the group on $\XX$ is smooth over $\mathfrak o$, up to an integral change of coordinates we have $Q(x,y)= x^2+\kappa y^2$, where $-\kappa$ is a non-square unit element. We can write $\eta= \eta_2 \cdot |\bullet|^{\frac{s}{2}}$ where $q^{s+1}=e^{-\check\alpha}(\chi)$ and $\eta_2$ is either the trivial character or a ramified quadratic character.

A fact which will make our computations easier is that if $-\kappa$ is not a square then $|x^2+\kappa y^2|=\max\{|x^2|,|\kappa y^2|\}$. Since, moreover, $\kappa$ is a unit,  the value of $\eta_2(x^2+\kappa y^2)$ only depends on the restriction of $\eta_2$ on units. We will compute the constant in the functional equations through the Parseval formula:
$$ \left< \Delta_{\tilde\chi}^2, f\right> = \left< \mathcal F \Delta_{\tilde\chi}^2, \mathcal F f\right>$$
with $f=\mathcal F f = 1_{\mathfrak o^2}$.

A basic relation which will be used repeatedly is
$$ \int_{|x|<|y|\le 1} |y|^s dx dy = \frac{q^{-1}(1-q^{-1})}{1-q^{-s-2}}.$$

\subsubsection*{$\tilde\chi$ is $\frac{\check\alpha}{2}$-unramified.}

We compute: 
$$ \left< \Delta_{\tilde\chi}^2, 1_{\mathfrak o^2}\right> = \int_{\mathfrak o^2} |x^2+\kappa y^2|^\frac{s}{2} dx dy =\int_{|x|<|y|\le 1} |y|^s dx dy + \int_{|y|<|x|\le 1} |x|^s dx dy + $$
$$+ \int_{|x|=|y|\le 1} |x|^s dx dy= 
 2\frac{q^{-1}(1-q^{-1})}{1-q^{-s-2}} + \frac{(1-q^{-1})^2}{1-q^{-s-2}} = \frac{1-q^{-2}}{1-q^{-s-2}}.$$

Similarly, for $\Delta_{^{w_\alpha}\tilde\chi}^2= |x^2+\kappa y^2|^{-\frac{s}{2}-1}$ we will get:
$$ \left< \Delta_{^w\tilde\chi}^2, 1_{\mathfrak o^2}\right>= \frac{1-q^{-2}}{1-q^{s}}$$
and the quotient of the two is:
$$b_{w_\alpha}^Y(\tilde\chi)= \frac{1-q^s}{1-q^{-s-2}}.$$

Hence:

\begin{equation} \label{tunramifiedetaunramified}
\boxed{b_{w_\alpha}^Y({\tilde\chi}) = \frac{1-q^{-1}e^{-\check\alpha}(\chi)}{1-q^{-1} e^{\check\alpha}(\chi)}.}
\end{equation}

\subsubsection*{$\tilde\chi$ is $\frac{\check\alpha}{2}$-ramified.}

It is an easy exercise in harmonic analysis and Gauss sums to verify the following identity:
$$ \int_{\mathfrak o^\times} \eta_2(\kappa+x^2) = \left\{\begin{array}{ll} 0 & \textrm{, if }-1\in (k^\times)^2 \\
									-2q^{-1} & \textrm{, otherwise.}\end{array}\right.$$
(Recall that $-\kappa$ is not a square.)

Now we compute:
$$\left< \Delta_{\tilde\chi}^2, 1_{\mathfrak o^2}\right>= \int_{\mathfrak o^{2}} \eta_2(x^2+\kappa y^2) |x^2+\kappa y^2|^\frac{s}{2} dx dy = $$ $$=\eta_2(\kappa) \int_{|x|<|y|\le 1} |y|^s dx dy + \int_{|y|<|x|\le 1} |x|^s dx dy +  \int_{|x|=|y|\le 1} \eta_2(x^2+\kappa y^2) |y|^s dx dy  =$$
$$= (1- \eta_2(-1))\frac{q^{-1}(1-q^{-1})}{1-q^{-s-2}} + \int_{\mathfrak o} |y|^{s+1} dy \int_{\mathfrak o^\times} \eta_2(x^2+\kappa y^2) = $$
$$ =  (1-\eta_2(-1)) \frac{q^{-1}(1-q^{-1})}{1-q^{-s-2}} + \left\{\begin{array}{ll} 0 & \textrm{, if }-1\in (k^\times)^2 \\
									-2q^{-1} \frac{1-q^{-1}}{1-q^{-s-2}} & \textrm{, otherwise.}\end{array}\right. = 0.$$

Since $1_{\mathfrak o^2}$ generates all $K$-invariant Schwartz functions on $k^2$ over the dilation action of $k^\times$, the value of the distribution $\Delta_{\tilde\chi}^2$ on any $K$-invariant function is zero. (There is no contradiction here; this just says that the corresponding eigenfunction $\Omega_{\tilde\chi}$ on $C^\infty(\TT\backslash\SSL_2)$ has value zero at the chosen point -- not that it vanishes identically.) We will use instead the function $\Phi_J=1_{\mathfrak p \times \mathfrak o^\times}$ whose Fourier transform has been computed in \S \ref{ssKinv}:
$$ \Delta_{\tilde\chi}^2(\Phi_J) = \int_{\mathfrak p \times \mathfrak o^\times} \eta_2(x^2+\kappa y^2) |x^2+\kappa y^2|^\frac{s}{2} dx dy = q^{-1}(1-q^{-1}) \eta_2(\kappa)$$
and we have $\Ff (\Phi_J) \equiv L_{\check\alpha(\varpi^{-1})} \Phi_J$ modulo $K$-invariants, whence by the vanishing of $\Delta_{\tilde\chi}^2$ on $K$-invariants:
$$ b_{w_\alpha}({\tilde\chi}) q^{-1}(1-q^{-1}) \eta_2(\kappa) = b_{w_\alpha} ({\tilde\chi}) \left< \Delta_{^w{\tilde\chi}}^2, \Phi_J\right>= \left< \Ff \Delta_{\tilde\chi}^2, \Phi_J\right> = $$
$$= \left< \Delta_{\tilde\chi}^2, \Ff \Phi_J \right> = \left< \Delta_{\tilde\chi}^2, L_{\check\alpha(\varpi^{-1})} \Phi_J \right> = e^{-\check\alpha}(\chi) q^{-1}(1-q^{-1}) \eta_2(\kappa).$$

Therefore:
\begin{equation} \label{tunramifiedetaramified} \boxed{b_{w_\alpha}^Y(\tilde\chi) = e^{-\check\alpha}(\chi).}
\end{equation}

\begin{remark}\label{remarkpartofchi}
 Notice that in the functional equations for the case $T$ non-split all that matters is whether $\tilde\chi$ is a ramified extension of $\chi$ and not which unramified/ramified extension it is (for instance, the cocharacter $e^{\frac{\check\alpha}{2}}$ does not appear). This is no coincidence: In this case the space $(\TT\backslash\SSL_2)(k)$ has two $\SL_2$-orbits (isomorphic to each other), and therefore if instead of the distribution $S_{\tilde\chi^{-1}\nu^{-1}}$, given by an integral over all open $B_2$-orbits, we considered the restriction of that distribution to a single $\SL_2$-orbit, then its Fourier transform should also be supported on the same $\SL_2$-orbit. But these distributions which are supported on one $\SL_2$-orbit do not ``see'' the extension of $\chi$ to the whole $A_Y$, but only to a subgroup of it, and one can check that that subgroup can detect ramification (it contains $\AA_Y(\mathfrak o)$) but not the precise extension of $\chi$.
\end{remark}

\subsection{\textbf{Case N}} \label{sscaseN}

Notice that for a one-dimensional torus $\TT$ which is a spherical subgroup of $\LL_\alpha$, if $\mathcal N_{\LL_\alpha}(\TT)\ne \mathcal Z_{\LL_\alpha}(\TT)$ then $\TT\subset \LL_\alpha'$, and therefore in case $N$ the variety $\HH_\alpha^0\backslash\LL_\alpha$ is of type $T$ with associated cocharacters $\check v_{D_{1/2}}=\frac{\check\alpha}{2}$ if $\HH_\alpha^0$ is split. To compute the functional equations for $\Delta_{\tilde\chi,\zeta}$ (\S \ref{ssbases}), we first describe how to choose the point $\xi\in Y_\zeta$. Notice that $Y_\zeta$ may not contain integral points, therefore $\HH_\alpha^0$ cannot always be assumed to be a smooth subgroup scheme over $\mathfrak o$ -- hence, not all of the cases which we will encounter in case $N$ have already been encountered in case $T$.

First of all, let $\xi_0\in \YY(\mathfrak o)$ be a point in the ``standard'' $B_0$-orbit on $\YY(\mathfrak o)$, and $\zeta_0$ the corresponding coset of $A_{Y,\alpha}$. We identify $\LL_\alpha/\mathcal R(\LL_\alpha)$ with $\PPGL_2$ and use $\xi_0$ to define a map: $\YY\PP_\alpha\to \XX_2:=\mathcal N(\TT)\backslash\PPGL_2$ (for an integral torus $\TT$ corresponding to the class of $\xi_0$). The latter is the space of binary quadratic forms modulo homotheties. Without loss of generality, $\xi_0$ maps to the quadratic form $x^2+ay^2$ for some $a\in -D(\zeta_0)\subset\mathfrak o^\times$. We choose $\xi\in Y_\zeta$ such that under the above map $\xi$ maps to the homothety class of the quadratic form $x^2+by^2$ with $b\in -D(\zeta)$. (Recall from \S \ref{sstypeN} that the discriminant $D(\zeta)$ is by its definition an element of $\mathfrak o\smallsetminus \mathfrak p^2$.) 

From the definition of $\Delta_{\tilde\chi,\zeta}:= \ev_1 \circ \Delta^{Y,\alpha}_{\tilde\chi,\zeta}$ and Lemma \ref{lemmaadj2} we deduce that the distribution  $\Delta_{\tilde\chi,\zeta}^2$ on $U_\alpha\backslash L_\alpha'$ is given by:
\begin{equation}\label{D2N}
 \Delta_{\tilde\chi,\zeta}^2 = \tilde\chi'(\xi) \eta(Q(x,y)) dx dy
\end{equation}
where, $\eta = \chi\delta^{-\frac{1}{2}} \circ e^{\frac{\check\alpha}{2}}$ and $Q$ is the quadratic form stabilized by $\HH_2^0$ and $\tilde\chi'$ is the character of Lemma \ref{lemmaadj2}. (The factor $\tilde\chi'(\xi)$ is non-canonical, as is the distribution $\Delta_{\tilde\chi,\zeta}^2$, cf.\ \S \ref{ssred}; what is important is how this factor varies as $\tilde\chi$ varies.) By our choice of point $\xi$, up to an integral change of coordinates and up to a constant we have $Q(x,y)=x^2+\varpi y^2$, where $\varpi=-D(\zeta)$.

To determine the functional equations for $\Delta_{\tilde\chi,\zeta}^2$, if $\zeta$ corresponds to an integral torus, we can use our computations for the corresponding cases of type $T$, with $\check v_D=\frac{\check\alpha}{2}$. However, notice that now there is an extra factor of 
$$\frac{\tilde\chi'(\xi)}{e^{-\alpha}\cdot ^{w_\alpha}\tilde\chi'(\xi) } = \tilde\chi\circ e^{-\frac{\check \alpha}{2}} \left(\frac{D(\zeta)}{D(\zeta_0) }\right).$$
In the above, we took into account that if $\tilde\chi'$ is the character used in the expression (\ref{D2N}) for $\Delta_{\tilde\chi,\zeta}^2$, then the character in the expression for $\Delta_{^{w_\alpha}\tilde\chi,\zeta}^2$ is $e^{-\alpha}\cdot^{w_\alpha}\tilde\chi'$. But the quotient $\frac{\tilde\chi'(\xi)}{e^{-\alpha}\,(^{w_\alpha}\tilde\chi')(\xi) }$ can be written as $\frac{\tilde\chi}{^{w_\alpha}\tilde\chi}$ (which in this case is a character of $A_Y$). By the definition of $\xi$, if $\bar\xi$ denotes its image in $A_Y$ we get:
$$\frac{\tilde\chi}{^{w_\alpha}\tilde\chi}(\xi) = \tilde\chi\left(\frac{\bar\xi}{^{w_\alpha}\bar\xi}\right) = \tilde\chi\circ e^{-\frac{\check\alpha}{2}}\left(\frac{b}{a}\right) = \tilde\chi\circ e^{-\frac{\check\alpha}{2}}\left(\frac{D(\zeta)}{D(\zeta_0) }\right).$$

Hence, if $\zeta$ corresponds to a split integral torus and $\tilde\chi$ is $\frac{\check\alpha}{2}$-unramified then 
\begin{equation}
 \boxed{ b_{w_\alpha}^Y(\tilde\chi,\zeta)=  \left( \frac{1-q^{-\frac{1}{2}}e^{-\frac{\check\alpha}{2}}(\tilde\chi)}{1- q^{-\frac{1}{2}}e^\frac{\check\alpha}{2}(\tilde\chi)} \right)^2},
\end{equation}
if $\zeta$ corresponds to a split integral torus and $\tilde\chi$ is $\frac{\check\alpha}{2}$-ramified then 
\begin{equation}
\boxed{b_{w_\alpha}^Y(\tilde\chi,\zeta)= \tilde\chi\circ e^\frac{-\check\alpha}{2}\left(\frac{D(\zeta)}{D(\zeta_0)}\right)  e^{- \check\alpha}(\chi)},
\end{equation}
if $\zeta$ corresponds to a non-split integral torus and $\tilde\chi$ is $\frac{\check\alpha}{2}$-unramified then
\begin{equation}
\boxed{b_{w_\alpha}^Y(\tilde\chi,\zeta) =  \frac{1-q^{-1}e^{-\check\alpha}(\chi)}{1-q^{-1} e^{\check\alpha}(\chi)}},
\end{equation}
while if $\zeta$ corresponds to a non-split integral torus and $\tilde\chi$ is $\frac{\check\alpha}{2}$-ramified then
\begin{equation} 
\boxed{b_{w_\alpha}^Y(\tilde\chi,\zeta) = \tilde\chi\circ e^\frac{-\check\alpha}{2}\left(\frac{D(\zeta)}{D(\zeta_0)}\right)  e^{-\check\alpha}(\chi)}.
\end{equation}

There remains to compute the equations for the case that $\zeta$ does not correspond to an integral torus. Here we will denote $b$ by $\varpi$ to remind of the fact that $\varpi=-D(\zeta)$ will be a uniformizing element in $k$, uniquely defined up to multiplication by $(\mathfrak o^\times)^2$. Hence $Q(x,y)=x^2+\varpi y^2$.

Now we notice the following: The form $Q(x,y)$ takes values not in the whole of $k$ but only in $\{1,\varpi\}\cdot (\mathfrak o^\times)^2$. This implies that the values of $\eta(Q(x,y))$ will be the same for two characters $\eta$ agreeing on $\{1,\varpi\}\cdot (\mathfrak o^\times)^2$, and since $\eta^2$ is unramified there exists an unramified such character. To compute the Fourier transforms, we may therefore without loss of generality assume that $\eta$ is unramified. The reader should compare this with Remark \ref{remarkpartofchi}: Here it will not matter whether $\tilde\chi$ is ramified or not, but only the value of its pull-back via $e^{\frac{\check\alpha}{2}}$ at $\varpi=-D(\zeta)$.

Hence, we may write $\eta(\bullet)=|\bullet|^{\frac{s}{2}}$ where $q^{\frac{s}{2}}=\tilde\chi\delta^{-\frac{1}{2}} \circ e^{-\frac{\check\alpha}{2}} (-D(\zeta)) = q^{-\frac{1}{2}}\tilde\chi \circ e^{-\frac{\check\alpha}{2}} (-D(\zeta))$.
Now we compute:
$$\left<\Delta^2_{\tilde\chi,\zeta},1_{\mathfrak o^2}\right> = \tilde\chi'(\xi) \int_{\mathfrak o^2} |x^2+\varpi y^2|^\frac{s}{2} dx dy = $$ $$=\tilde\chi\nu\delta^{-\frac{1}{2}}(\bar\xi)  \left(\int_{|x|<|y|\le 1} q^{-\frac{s}{2}}|y|^s dx dy + \int_{|y|<|x|\le 1} |x|^s dx dy + \int_{|x|=|y|\le 1} |x|^s dx dy \right)= $$
$$= \tilde\chi\nu\delta^{-\frac{1}{2}}(\bar\xi) \left((1+q^{-\frac{s}{2}})\frac{q^{-1}(1-q^{-1})}{1-q^{-s-2}} + \frac{(1-q^{-1})^2}{1-q^{-s-2}}\right) = $$ $$=\tilde\chi\nu\delta^{-\frac{1}{2}}(\bar\xi) \frac{(1-q^{-1})(1+q^{\frac{-s-2}{2}})}{1-q^{-s-2}}= \tilde\chi\nu\delta^{-\frac{1}{2}}(\bar\xi) \frac{1-q^{-1}}{1-q^\frac{-s-2}{2}}.$$

Similarly, the integral of $\Delta_{^{w_\alpha}\chi,\zeta}^2= e^{-\alpha}\cdot {^{w_\alpha}\tilde\chi'}(\xi)  |x^2+\varpi y^2|^{-\frac{s}{2}-1}dxdy$:
\begin{equation*}
\left<\Delta^2_{^{w_\alpha}\tilde\chi,\zeta},1_{\mathfrak o^2}\right> = e^{-\alpha}\cdot {^{w_\alpha}\tilde\chi'}(\xi)   \frac{1-q^{-1}} {1-q^{\frac{s}{2}}} .
\end{equation*}

The quotient of the two is:
$$b_{w_\alpha}^Y(\tilde\chi,\zeta)= \tilde\chi\circ e^\frac{-\check\alpha}{2}\left(\frac{D(\zeta)}{D(\zeta_0)}\right) \frac{1-q^{\frac{s}{2}}}{1-q^{\frac{-s-2}{2}}}.$$
Therefore:
\begin{equation} \label{caseNnonintegral}
\boxed{b_{w_\alpha}^Y(\tilde\chi,\zeta) = \tilde\chi\circ e^\frac{-\check\alpha}{2}\left(\frac{D(\zeta)}{D(\zeta_0)}\right)  \frac{1-q^{-\frac{1}{2}}\tilde\chi \circ e^{-\frac{\check\alpha}{2}} (-D(\zeta))}{1-q^{-\frac{1}{2}}\tilde\chi \circ e^{\frac{\check\alpha}{2}} (-D(\zeta))}}
\end{equation}


\part{The formula with $L$-values}

\section{Simple spherical reflections} \label{secsimplespherical}

From now on we assume that $\XX$ is a spherical variety without type-$N$ reflections, i.e.\ such that for no orbit $\YY$ of maximal rank and simple root $\alpha$ the pair $(\YY,\alpha)$ is of type $N$. Moreover, the character $\delta_{(X)}^{-\frac{1}{2}}\tilde\chi$ of $A_X$ is assumed to be unramified. In the case of a parabolically induced $\XX$ with a non-trivial character $\Psi=\psi\circ\Lambda$, we assume that $\Lambda$ (or $\Psi$) is \emph{generic}: that means that $\Lambda$ belongs to the open orbit of the corresponding Levi on the additive characters of the unipotent radical. 

\subsection{The root system of a spherical variety}\label{ssrootsystem}

Up to now we have developed an algorithm for computing the constants in formula (\ref{mainformula}) of Theorem \ref{maintheorem} by calculating the functional equations for all orbits $Y$ of maximal rank and simple roots $\alpha$ \emph{of $G$}. However, the formula really depends only on the functional equations for $\Delta_{\tilde\chi}$, the morphism defined by the open $\BB$-orbit. Those are parametrized by the little Weyl group $W_X$. 

It is known that the faithful action of $W_X$ on $\varchi(\XX) \otimes \QQ$ (equivalently on $\mathcal Q=\Hom(\varchi(\XX),\QQ)$) is generated by reflections, and the cone $\mathcal V\subset\mathcal Q$ of invariant valuations is a fundamental domain for its action on $\mathcal Q$. In fact $W_X$ is the Weyl group of a root system, and the rank of this root system is called the \emph{rank} of the spherical variety $\XX$. (We will use this term only when it is clear that we are not referring to the other notion of rank -- namely, the rank of $\XX$ as a $\BB$-variety.) This root system is defined as follows \cite{KnAut}: One considers in $\varchi(\AA_X) \otimes \QQ$ the cone (negative) dual to $\mathcal V$, i.e.\ the set $\{\chi \in \varchi(\XX)\otimes\QQ |\left<v,\chi\right>\le 0\text{ for every }v\in\mathcal V\}$. This is a strictly convex cone; let us call its extremal rays the \emph{spherical root half-lines} of $\XX$. The simple roots of the root system of \emph{loc.cit.\ }are the generators for the 
intersections of $\varchi(\XX)$ with the spherical root half-lines; these simple roots are called the \emph{spherical roots} of $\XX$. It seems that this root system is not quite the correct one for the purposes of representation theory, therefore we describe in \cite{SV} (and recall below) a variant $\Phi_X$ of that, on the same vector space and with the same Weyl group, but with roots of different length. We point out that in the literature the term ``spherical roots'' is used for what is the set of \emph{simple} roots in the corresponding root system. We will comply with this convention here, both for the spherical roots and for the normalized spherical roots that will be defined below.

For the case of $C^\infty(X,\mathcal L_\Psi)$, we explained in \cite{Sa2} how to associate similar invariants based on Knop's definitions for non-spherical varieties. Notice the following:

\begin{lemma}\label{genericcharacter}
 Let $\Lambda:\UU_{P^-}\to \Ga$ be an additive character of the unipotent radical of a parabolic $\PP^-$, normalized by a spherical subgroup $\MM$ of the Levi $\LL$. We assume that $\PP^-$ is opposite to a standard parabolic and $\MM\cdot (\LL\cap\BB)$ is open in $\LL$. Then $\Lambda$ is completely determined by its restriction to the simple root subgroups $\UU_{-\alpha}$, $\alpha\in\Delta\smallsetminus\Delta_L$. A generic character is non-trivial on all those subgroups.
\end{lemma}

\begin{proof}
 For the first statement, it suffices to show that if $\Lambda$ is zero on all those simple root subgroups, then it is trivial. Given that $[\UU_{-\alpha},\UU_\beta]=1$ for all distinct $\alpha,\beta\in\Delta$, the subgroups $\UU_{-\alpha}$, $\alpha\in\Delta\smallsetminus\Delta_L$, are normalized by $\BB_L:=\LL\cap \BB$. In combination with the fact that $\MM$ stabilizes $\Lambda$, $bm\cdot \Lambda$ is zero on those subgroups, for every $b\in \BB_L$, $m\in\MM$. But $\BB_L\MM$ is open in $\LL$, hence, $l\cdot \Lambda$ is zero on all those subgroups for all $l\in \LL$. The only character with this property is the trivial one.

For the second claim, the same argument shows that if $\Lambda$ is zero on one of the $\UU_{-\alpha}$'s, then the same is true for $l\cdot \Lambda$, for every $l\in\LL$. Such a $\Lambda$ cannot be generic.
\end{proof}

To describe the root system $\Phi_X$, since it has the same Weyl group as that of Knop, it is enough to describe how to choose a set of simple roots on the spherical root half-lines; those simple roots will be called \emph{normalized spherical roots}. 
 Recall Brion's description of generators for $W_{(X)}$ (Theorem \ref{Brionsdescription}). To repeat, adding the case of a non-trivial $\mathcal L_\Psi$, every $w$ in a set of generators of $W_{(X)}$ can be written as $w=w_1^{-1} w_2 w_1$ where:
\begin{itemize}
 \item $^{w_1}\mathring X =: Y$ with $\operatorname{codim}(Y)= l(w_1)$. (In the notation of \ref{ssBrion}, $w_1^{-1}\in W(\YY)$.)
 \item $w_2$ is either of the following three:
\begin{enumerate}
 \item equal to $w_\alpha$, where $\alpha$ is a simple root such that $(Y,\alpha)$ is of type $T$ (or $N$, if we allow type $N$ reflections), or
 \item equal to $w_\alpha w_\beta$, where $\alpha,\beta$ are two orthogonal simple roots which both lower $Y$ to the same orbit $Y'$, or
 \item equal to $w_\alpha$, where $\alpha$ is a simple root such that $(Y,\alpha)$ is of type $(U,\psi)$. (By our assumption that $\Psi$ be generic and Lemma \ref{genericcharacter}, in this case $w_1$ is trivial.)
\end{enumerate}
\end{itemize} 

It is easy to see that such an element $w$ induces a reflection on $\varchi(\XX)\otimes \mathbb Q$, which implies that every spherical root half-line contains a root of $\GG$ or the sum of two orthogonal roots which are simple with respect to some choice of basis for the root system of $\GG$. Moreover, these two cases are mutually exclusive, hence we can define, for every spherical root $\gamma$, the corresponding \emph{normalized spherical root} $\gamma'$ which is equal to the unique positive multiple of $\gamma$ with the property that:
\begin{itemize}
 \item either $\gamma'$ is a root of $\GG$;
\item or it is the sum of two roots of $\GG$ which are orthogonal to each other and simple with respect to some choice of basis for the root system.
\end{itemize}

It is shown in \cite{SV} that the set of those $\gamma'$, for $\gamma$ ranging over all spherical roots of $\XX$, is the set of simple roots for a root system $\Phi_X$ with Weyl group $W_X$.
The dual root system $\check\Phi_X$ is expected to be the root system of the dual group $\check G_{X,GN}$ that Gaitsgory and Nadler \cite{GN} have associated to the spherical variety. In \cite{SV} we show that if there are no reflections of type $N$ then the data $(\varchi(\XX)^*, \check\Phi_X,\varchi(\XX), \Phi_X)$ form a root datum and hence define uniquely up to isomorphism a complex group $\check G_X$ with a canonical maximal torus $A_X^*$ and root system $\check\Phi_X$; it will be called ``the dual group of $\XX$''.

The set of normalized spherical roots will be denoted by $\Delta_X$ (not to be confused with $\Delta(\XX)$), and the set of positive elements in $\Phi_X$ by $\Phi_X^+$. The rank of this root system is called the \emph{rank} of the spherical variety. We will see that the normalized spherical roots (or rather, their co-roots) naturally appear in our formula for Hecke eigenvectors. It is clearly enough to compute the proportionality factors $b_{w_\gamma}$ for all $\gamma$ in $\Delta_X$. For those, we introduce the following notation: $\PP_\gamma, \LL_\gamma$ are the standard parabolic and Levi $\PP_{\supp(\gamma)}, \LL_{\supp(\gamma)}$, and $\XX_\gamma$ the spherical variety $\mathring \XX \PP_\gamma/\UU_{P_\gamma}$ for $\LL_\gamma$. Here and throughout, $\supp(\gamma)$ denotes the minimal set of simple roots such that $\gamma$ lies in their linear span.

\subsection{Type of a spherical root and Brion paths}

Since Brion's description of generators for $W_{(X)}$ will be very important in our definitions later, we mention a few more facts on it. First of all, we introduce the following language:

Let us consider the weak order graph $\mathfrak G$, forgetting the orientation of the edges, and consider paths $\mathfrak g$ corresponding to the description of $W_{(X)}$-generators in Brion's theorem, that is the path consists of a sequence of adjacent edges of the form $e_1 e_2 \dots e_n E_1 E_2 e_n \dots e_2 e_1$, where:
\begin{itemize}
 \item one vertex of the edge $e_1$ is $\mathring \XX$, and for all edges $e_1,\dots,e_n$, if $\YY$ is a vertex of the edge and $\alpha$ is its label then $(\YY,\alpha)$ is of type $U$;
 \item if $\YY$ denotes the lower vertex of the edge $e_n$ then one of the following happens:
\begin{enumerate}
 \item either $E_1=E_2$ and it is labelled by a simple root $\alpha$ such that $(\YY,\alpha)$ is of type $T$, $N$ or $(U,\psi)$;
 \item or $E_1$ and $E_2$ are labelled by orthogonal simple roots $\alpha,\beta$ lowering $\YY$ to the same $\BB$-orbit.
\end{enumerate}
\end{itemize}

(For simplicity, in the case of a non-trivial $\mathcal L_\Psi$ we use the weak order graph for the underlying variety without reference to $\mathcal L_\Psi$.)

Let $\alpha_1,\dots,\alpha_n$ denote the labels of $e_1,\dots,e_n$, respectively. Let $w_2=w_\alpha$, in the first case, and $w_2=w_\alpha w_\beta$ in the second. If the path $\mathfrak g$ corresponds to a reduced decomposition of the corresponding Weyl group element $w=w_{\alpha_1}\cdots w_{\alpha_n} w_2 w_{\alpha_n}\cdots w_{\alpha_1}$, i.e.\ if $l(w)=2n+l(w_2)$, then that path will be called a \emph{Brion path} corresponding to $w$. A fact that we will use is that Brion's description for a set of generators of $W_{(X)}$ applies to the simple reflections in $W_X$ corresponding to spherical roots:

\begin{proposition}\label{Brionapplies}
 For every spherical root $\gamma$, the corresponding element $w_\gamma\in W_X$ admits a Brion path.
\end{proposition}

\begin{proof}[Proof (Sketch)] The proof will eventually rest on a case-by-case analysis of low-rank spherical varieties which will be presented later in this section.
Given a simple spherical root $\gamma$, we consider the spherical variety $\mathring\XX\cdot\PP_\gamma/\mathcal R(\PP_\gamma)$. We claim that it is a spherical variety of rank at most two:
\begin{lemma}\label{rankatmosttwo}
 A spherical variety for a group $\GG$ with the property that there is a spherical root $\gamma$ such that $\PP_\gamma =\GG$ is of rank at most two. Equivalently, for any spherical variety and any spherical root $\gamma$ of it, the span of the support of $\gamma$ contains at most two spherical roots.
\end{lemma}
\begin{proof}
 Assume that $\XX$ is a (homogeneous) spherical variety for a group $\GG$ with a spherical root $\gamma$ such that $\PP_\gamma = \GG$, and that $\XX$ has rank three or more. Let $\gamma'$ be a spherical root different from $\gamma$, then there is a spherical variety for $\GG=\PP_\gamma$ whose spherical roots are $\gamma$ and $\gamma'$ and which has a non-trivial torus of $\GG$-automorphisms. (Indeed, we may assume by dividing by a group of automorphisms of $\XX$ that $\XX$ admits a wonderful compactification, see \cite{Wa}. Then there is a non-open $\GG$-orbit in this wonderful compactification whose spherical roots are $\gamma$ and $\gamma'$, and the open $\GG$-orbit in its normal bundle has the properties stated.) By inspection of the tables of \cite{Wa} (see Theorem \ref{AW} below), no such spherical variety exists.
\end{proof}

Now if $\mathring\XX\cdot\PP_\gamma/\mathcal R(\PP_\gamma)$ is of rank one, Brion's description applies to an element of $W_{(X)}$ which maps to the unique non-trivial element in the quotient $\overline{W_X}$ but, as we remarked after Theorem \ref{Brionsdescription}, that element has to belong to $W_X$. If it is of rank two, we can see by inspection of the few wonderful varieties of rank two where one spherical root is included in the span of the support of the other (see again Theorem \ref{AW}) that, again, Brion's description applies to the elements of $W_X$ corresponding to simple spherical roots.
\end{proof}

Depending on which of the enumerated cases above appears:
\begin{enumerate}
 \item the weight $\gamma'=w_1^{-1}\alpha$ will be called a \emph{normalized spherical root of type $T$, $N$ or $(U,\psi)$}, respectively.
 \item the weight $\gamma'=w_1^{-1} (\alpha + \beta)$ will be called a \emph{normalized spherical root of type $G$}. (The name ``type $G$'' originates from the fact that this is the spherical root for the variety $\SSL_2\backslash\SSL_2\times\SSL_2$. This shouldn't cause any confusion with ``reflections of type $G$'', since the latter do not generate any spherical roots.)
\end{enumerate}

We notice that the type of $\gamma'$ is anambiguous because it does not depend on the choice of path. Indeed, spherical roots of type $G$ are characterized as not being proportional to any root of $\GG$; spherical roots of type $(U,\psi)$ are, by assumption, simple roots in $G$ so there is a unique choice of Brion path for them (it is possible to distinguish those roots even without assuming that $\psi$ is generic, but we omit that); the remaining cases are all of type $T$ or type $N$, according as $\gamma'\in \varchi(\XX)$ or not.

Finally, for the discussion which follows recall that $\BB_{\bar k}$-invariant irreducible divisors in $\XX_{\bar k}$ are called ``colors''. We notice that a color may not be defined over $k$, as happens in the case of $\TT\backslash\SSL_2$ with $\TT$ a non-split torus. In fact, this is essentially the only case where a color is not defined over $k$. More precisely, every color is contained in $(\mathring \XX\PP_\alpha)_{\bar k}$ for some simple root $\alpha$, and if $(\mathring\XX,\alpha)$ is of type $U$, $T$ split or $N$ then the absolutely irreducible divisors in $(\mathring \XX\PP_\alpha)_{\bar k}$ are defined over $k$. We have already encountered colors in \S \ref{ssferesult} and we have seen that each color $\DD$ induces a valuation on $k(\XX)$, which by restriction to $\BB$-semiinvariants (eigenfunctions) defines an element $\check v_D\in \Hom(\varchi(\XX),\mathbb Z)$. (This is clearly well-defined, even if the color is not defined over $k$.)

\subsection{The values of the cocycles for simple spherical reflections}

For the statement of the results it is convenient to introduce the constants $B_w$ (for $w\in W_X$), defined as:

\begin{equation}\label{Bw}
 B_w(\tilde\chi) = \prod_{\check\alpha\in\check \Phi^+, w\check\alpha<0} (-e^{\check\alpha}(\chi)) b_w(\delta_{(X)}^{\frac{1}{2}}\tilde\chi).
\end{equation}

The $B_w$ are cocycles: $W_X \to \CC(A_X^*)$. Using those, Theorem \ref{maintheorem} becomes:

\begin{equation}\label{formulawithBw}
  \Omega_{\delta_{(X)}^\frac{1}{2}\tilde\chi}(x) = \delta_{P(X)}^{-\frac{1}{2}}(x) \sum_{w\in W_X} B_w(\tilde\chi) \, ^w\tilde\chi(x).
\end{equation}

Our goal at this point is to compute $B_w$ for $w=w_\gamma$, where $w_\gamma$ is a simple spherical reflection. In what follows, we denote by $\rho$ the half-sum of positive roots, and by $\rho_{P(X)}$ the half-sum of roots in the unipotent radical of $\PP(\XX)$; as before, $\check\rho$ denotes the half-sum of positive co-roots. We are aiming at the following:

\begin{statement}\label{STATEMENT}
 Let $\mathfrak g$ be a Brion path corresponding to the simple spherical reflection $w_\gamma$, where $\gamma$ is a \emph{normalized} spherical root. Let ${\check\theta}$ be the valuation induced by the codimension-one orbit in the path, viewed as an element of $\Hom(\varchi(\XX),\mathbb Z)$. Then, according to the type of the root $\gamma$ we have:
\begin{itemize}
 \item If $\gamma$ is of type $G$, then ${\check\theta}=\check\gamma$ and:
\begin{equation}\label{STATEMENTG}
 B_{w_\gamma}(\tilde\chi)=-e^{\check\gamma} \frac{1-q^{-\left<\check\gamma,\rho_{P(X)}\right>}e^{-\check\gamma}}{1-q^{-\left<\check\gamma,\rho_{P(X)}\right>}e^{\check\gamma}}(\chi).
\end{equation}
 \item If $\gamma$ is of type $T$ split, then $\left<{\check\theta},\gamma\right>=1$ and:
\begin{equation}\label{STATEMENTT}
B_{w_\gamma}(\tilde\chi) = - e^{\check\gamma} \frac{1-q^{-\left<{\check\theta},\rho_{P(X)}\right>} e^{-{\check\theta}}} {1-q^{\left<{\check\theta},\rho_{P(X)}\right>-\left<\check\rho,\gamma\right>}e^{{\check\theta}}} 
\frac{1-q^{\left<{\check\theta},\rho_{P(X)}\right>-\left<\check\rho,\gamma\right>}e^{w_\gamma{\check\theta}}}{1-q^{-\left<{\check\theta},\rho_{P(X)}\right>}e^{-w_\gamma{\check\theta}}}(\tilde\chi).
\end{equation}
 \item If $\gamma$ is of type $T$ non-split, then $\check\theta=\frac{\check \gamma}{2}$ and:
\begin{equation}\label{STATEMENTTnonsplit}
B_{w_\gamma}(\tilde\chi) = - e^{\check\gamma} \frac{1-q^{-\left<{\check\theta},\rho_{P(X)}\right>} e^{-{\check\theta}}} {1+q^{\left<{\check\theta},\rho_{P(X)}\right>-\left<\check\rho,\gamma\right>}e^{{\check\theta}}} 
\frac{1+q^{\left<{\check\theta},\rho_{P(X)}\right>-\left<\check\rho,\gamma\right>}e^{w_\gamma{\check\theta}}}{1-q^{-\left<{\check\theta},\rho_{P(X)}\right>}e^{-w_\gamma{\check\theta}}}(\tilde\chi).
\end{equation}
 \item If $\gamma$ is of type $(U,\psi)$, then $\check\theta=\check\gamma$ and: \begin{equation}\label{STATEMENTPSI} B_{w_\gamma}(\tilde\chi)=-e^{\check\gamma}(\chi).\end{equation}
\end{itemize}
In case $T$, the exponents $-\left<{\check\theta},\rho_{P(X)}\right>$ and $\left<{\check\theta},\rho_{P(X)}\right>-\left<\check\rho,\gamma\right>$ are equal, unless possibly if $\check\theta=\frac{\check\gamma}{2}$. In any case, these exponents are negative half-integers. 
\end{statement}

\begin{remarks}\begin{enumerate}
 \item Recall that $\delta_{P(X)}\in A_X^*$ (\ref{deltapx}) and therefore the exponential of $\left<\check\gamma,\rho_{P(X)}\right>$ makes sense.
 \item  A priori it is not clear that being of type $T$ split or non-split is a well-defined property of $\gamma$, independent of the path. A posteriori, this is true if the above statement is true.
 \item We remark that the choice of lift $\tilde\chi$ of $\chi$ matters only in cases $T$ split or non-split.
\end{enumerate}
\end{remarks}

We can only prove this statement by reducing to certain low-rank spherical varieties, and we have performed the corresponding computation only for classical groups:

\begin{theorem} \label{statementtrue}
 Let $\XX$ be a spherical variety without type-$N$ roots and such that for every $\gamma\in \Delta_X$ the Levi $\LL_{\gamma}$ is a classical group. Then for every normalized spherical root $\gamma$ of $\XX$ and any Brion path corresponding to $w_\gamma$ Statement \ref{STATEMENT} is true.
\end{theorem}

\subsection{Glueing paths from simple varieties}

Our goal for the rest of this section is to prove Theorem \ref{statementtrue}. Let $\gamma\in\Delta_X$, and choose a Brion path $\mathfrak g$ which corresponds to $w_\gamma$. Let $\ZZ\ne \mathring \XX$ be a vertex in the path $\mathfrak g$ which is the endpoint of two distinct edges $e_1,e_2$ in the path. Let $\delta,\epsilon$ be the simple roots of $\GG$ labelling $e_1$ and $e_2$, and consider the spherical variety $\ZZ\PP_{\delta\epsilon}/\mathcal R(\PP_{\delta\epsilon})$ for the group $\LL_{\delta\epsilon}$. If $\GG_1,\GG_2,\dots$ are reductive groups of rank two, and $\XX_1,\XX_2,\dots$ spherical varieties thereof, we will say that the path $\mathfrak g$ is ``glued'' from $\XX_1,\XX_2,\dots$ if all spherical varieties $\ZZ\PP_{\delta\epsilon}/\mathcal R(\PP_{\delta\epsilon})$ obtained this way are contained in the list of $\XX_1,\XX_2,\dots$. As a matter of language, if the path does not contain any vertex attached to two distinct edges (which is the case if $\gamma\in \Delta$), then we consider the path 
$\mathfrak g$ to be glued from any list of such spherical varieties.

We have:
\begin{proposition}\label{gluedfrom}
 Assume that there is a path $\mathfrak g$ corresponding to $w_\gamma$ which over the algebraic closure is glued out of the following varieties:
$$\PPGL_2\backslash\PPGL_2\times\PPGL_2, \GGL_2\backslash\SSL_3, $$ $$\SSp_2\times\SSp_2\backslash \SSp_4, (\Gm\ltimes\Ga)\times\SSp_2\backslash\SSp_4, \Gm\times\SSp_2\backslash\SSp_4,  $$
$$\PP_\alpha\backslash\SSL_3 \text{ (where }\alpha\text{ is a simple root) and }$$
$$\PP_\alpha\backslash\SSp_4\text{ (where }\alpha\text{ is the long simple root)}.$$
Then $\mathfrak g$ satisfies Statement \ref{STATEMENT}, except possibly for the last assertion (on the exponents).
\end{proposition}

\begin{remark}
 The first two rows in the list consist of all spherical varieties without spherical roots of type $N$, for simply connected groups of rank two other than $\GG_2$, and with the property that both simple roots are in the support of a spherical root (see Theorem \ref{AW} below). Therefore, unless $\mathring\XX$ is the only node of the path $\mathfrak g$, the bottom part of the path will necessarily consist of one of these varieties. Hence, if the group has no $\GG_2$-factors, the content of the assumption is contained in the next lines, where one restricts the horospherical varieties which can glue the path.
\end{remark}

Moreover:

\begin{proposition}\label{ranktwo}
 If $\gamma\in\Delta_X$ with $\LL_\gamma$ a classical group then for any Brion path $\mathfrak g$ corresponding to $w_\gamma$ there is a path $\mathfrak g'$ corresponding to $w_\gamma$ such that $\mathfrak g$ and $\mathfrak g'$ have the same first edge and $\mathfrak g'$ satisfies the assumptions of Proposition \ref{gluedfrom}. 

Moreover, in case $T$ the numbers $-\left<{\check\theta},\rho_{P(X)}\right>$ and $\left<{\check\theta},\rho_{P(X)}\right>-\left<\check\rho,\gamma\right>$ are equal, unless possibly if $\check\theta=\frac{\check\gamma}{2}$. In any case, these numbers are negative half-integers. 
\end{proposition}

Since Statement \ref{STATEMENT} only depends on the first edge of $\mathfrak g$, this proves Theorem \ref{statementtrue}. The goal of the rest of this section is to prove Propositions \ref{gluedfrom} and \ref{ranktwo}.

\subsection{Computation for the simple varieties.}

Proposition \ref{gluedfrom} will follow from the following:

\begin{proposition}\label{gluedfrom2}
Assume that the Brion path $\mathfrak g$ in the weak order graph, corresponding to $w_\gamma$ for a normalized spherical root $\gamma$, satisfies the assumptions of Proposition \ref{gluedfrom}. Let $\ZZ$ be a node of $\mathfrak g$ other than the lowest one, $w$ the element of the Weyl group corresponding to the path below $\ZZ$. Hence $w=w_1^{-1} w_2 w_1$, where $w_1$ lowers $\ZZ$ to an orbit $\YY$ and $w_2=w_\beta w_{\tilde\beta}$ or $w_2=w_\beta$, according as $\gamma$ is of type $G$ or $T$. Let $\alpha$ be the label of an edge in $\mathfrak g$ which lowers $\ZZ$. Let $`b_w^Z(\tilde\chi) = \prod_{\check\epsilon\in\check \Phi^+, w\check\epsilon<0} (-e^{\check\epsilon}(\chi)) b_w^Z(\tilde\chi)$. Then:
\begin{itemize}
 \item If the root $\gamma$ is of type $G$ then 
\begin{equation}
 `b_w^Z(\tilde\chi) = - e^{\check\gamma'}(\chi) \frac{1-q^{-1}e^{-\check\alpha}}{1-q^{-1}e^{-w\check\alpha}}(\chi)
\end{equation}
where $\check\gamma'=w_1^{-1}\check\beta$. Moreover, we have that $\check\alpha|_{\varchi(\ZZ)}=\check\gamma'|_{\varchi(\ZZ)}$. (Recall that as an element of $\varchi(\ZZ)^*$, this does not depend on the choice between $\beta$ and ${\tilde\beta}$, and similarly $e^{\check\gamma'}(\chi)$ does not depend on this choice for characters $\chi$ of $A$ for which $\Delta_{\tilde\chi}^Z$ is defined.)

 \item If the root $\gamma$ is of type $T$ split, then
\begin{equation}\label{STATEMENTT-alt}
 `b_w^Z(\tilde\chi) = - e^{\check\gamma'}(\chi) \frac{1-q^{-1}e^{-\check\alpha}}{1-q^{1-\left<\check\rho,\gamma'\right>}e^{\check\alpha}} \frac{1-q^{1-\left<\check\rho,\gamma'\right>}e^{w\check\alpha}}{1-q^{-1}e^{-w\check\alpha}}(\chi)
\end{equation}
where $\gamma'=w_1^{-1}\beta$, $\check\gamma'=w_1^{-1}\check\beta$. Moreover, in this case $\gamma'\in\varchi(\ZZ)$ and $\left<\check\alpha,\gamma'\right>=1$.
 
 \item If the root $\gamma$ is of type $T$ non-split, then
\begin{equation}\label{STATEMENTTnonsplit-alt}
 `b_w^Z(\tilde\chi) = - e^{\check\gamma'}(\chi) \frac{1-q^{-1}e^{-\check\alpha}}{1+q^{1-\left<\check\rho,\gamma'\right>}e^{\check\alpha}} \frac{1+q^{1-\left<\check\rho,\gamma'\right>}e^{w\check\alpha}}{1-q^{-1}e^{-w\check\alpha}}(\chi)
\end{equation}
where $\gamma'=w_1^{-1}\beta$, $\check\gamma'=w_1^{-1}\check\beta$. Moreover, in this case $\gamma'\in\varchi(\ZZ)$ and $\check\alpha|_{\varchi(\XX)}=\frac{\check\gamma'}{2}$.
\end{itemize}
\end{proposition}

Let us first see why this proves Proposition \ref{gluedfrom}:

\begin{proof}[Proof of Proposition \ref{gluedfrom}]
  First of all, if $(\mathring \XX,\alpha)$ is of type $T$ split, then setting $\delta_{(U_{P_\alpha})_\xi}=\delta_{(X)}$ and substituting $\tilde\chi$ by $\delta_{(X)}^\frac{1}{2}\tilde\chi$ in formula (\ref{tsplitetaunramified}) we get: $$e^{\check v_D}(\delta_{(X)}^\frac{1}{2}\tilde\chi\delta^\frac{1}{2}\delta_{(X)}^{-1})=e^{\check v_D}(\tilde\chi\delta_{P(X)}^\frac{1}{2}),$$ and similarly for $\check v_{D'}$. Given the property that $B_{w_\alpha}(\delta_{P(X)}^\frac{1}{2})=0$ (by Lemma \ref{omegatrivial}),
it follows that $e^{\check v_D}(\delta_{P(X)})=q^{-1}$ for a suitable naming of $\DD, \DD'$, and the fact that $\check v_{D'}+\check v_D \equiv \check\alpha$ (the symbol $\equiv$ stands to remind that we are talking about the image of $\check \alpha$ in $\varchi(\XX)^*$) and $e^{\check\alpha}(\delta_{P(X)})=e^{\check\alpha}(\delta)=q^{-2}$ implies that the same is true for $\check v_{D'}$ in place of $\check v_D$. This verifies (\ref{STATEMENTT}) in this case. (In case we are considering a line bundle $\mathcal L_\Psi$ over $X$, 
we still have $B_{w_\alpha}(\delta_{P(X)}^\frac{1}{2})=0$ by using, for instance, the variety $\XX$ without the character $\Psi$.)

Similarly, if $(\mathring \XX,\alpha)$ is of type $T$ non-split, then (\ref{STATEMENTTnonsplit}) follows from (\ref{tunramifiedetaunramified}) and if $(\mathring \XX,\alpha)$ is of type $(U,\psi)$ then (\ref{STATEMENTPSI}) follows from (\ref{casepsi}).

In all other cases, if $\alpha$ is the label of the first edge in the path $\mathfrak g$ then $(\mathring \XX, \alpha)$ is of type $U$ and the last proposition applies. Notice that $\varchi(\XX)\perp \Delta(\XX)$, therefore $e^{\check\gamma}(\chi\delta_{(X)}^\frac{1}{2})=e^{\check\gamma}(\chi)$. From the fact that $\delta_{P(X)}\in A_X^*$ we get in case of a root of type $G$ that: $e^{\check\alpha}(\delta_{(X)}^\frac{1}{2})= e^{\check\alpha}(\delta^\frac{1}{2} \delta^{-\frac{1}{2}}_{P(X)})=q^{-1} e^{\check\gamma}(\delta_{P(X)}^{-\frac{1}{2}})$ and this proves (\ref{STATEMENTG}). Finally, (\ref{STATEMENTT}) and (\ref{STATEMENTTnonsplit}) follow from (\ref{STATEMENTT-alt}), resp.\ (\ref{STATEMENTTnonsplit-alt}), because ${\check\theta}\equiv \check\alpha$ and $e^{\check\alpha}(\delta_{(X)}^\frac{1}{2}\delta_{P(X)}^\frac{1}{2})=q^{-1}$.
\end{proof}

\begin{proof}[Proof of Proposition \ref{gluedfrom2}]
 We prove it by induction on the dimension of $\ZZ$. We start with the case $G$. First of all, if $\ZZ$ is equal to $\YY$ and $\alpha=\beta,\tilde\beta$ are the two orthogonal roots lowering $\ZZ$ to another orbit then we have computed in Example \ref{groupcase} that in this case $\check\gamma'\equiv\check\alpha$ and:
\begin{equation*}
 `b^Z_w(\tilde\chi)= - e^{\check\alpha} \frac{1-q^{-1} e^{-\check\alpha}}{1-q^{-1} e^{\check\alpha}}(\chi).
\end{equation*}

Now assume that $\ZZ$ is lowered by $\alpha$ to some vertex $\VV$ of $\mathfrak g$, and let $\beta$ be an edge lowering $\VV$. By the induction hypothesis, $$`b^V_{w_\alpha w w_\alpha}(^{w_\alpha}\tilde\chi)= - e^{w_\alpha\check\gamma'} \frac{1-q^{-1} e^{-\check\beta}}{1-q^{-1} e^{\check\beta}}(^{w_\alpha}\chi).$$ Using our functional equations for roots of type $U$, we compute:
$$ `b^Z_w(\tilde\chi)= - e^{\check\gamma'} \frac{1-q^{-1} e^{-\check\alpha}}{1- e^{-\check\alpha}}(\chi) \cdot \frac{1-q^{-1} e^{-w_\alpha\check\beta}}{1-q^{-1} e^{w_\alpha\check\beta}}(\chi) \cdot \frac{1-e^{-w\check\alpha}}{1-q^{-1}e^{-w\check\alpha}}(\chi).$$

To complete the proof in this case, it suffices to prove that if $\VV\PP_{\alpha\beta}/\mathcal R(\PP_{\alpha\beta})$ is among the list of varieties in the assumptions of the proposition, then: 
\begin{equation}\label{inductionstep}
 e^{-\check\alpha}(\chi)=q^{-1}e^{-w_\alpha\check\beta}(\chi)
\end{equation}

This will be proved via a \emph{case-by-case} computation below. Finally, since by the induction assumption we have $\check\beta|_{\varchi(\VV)} = w_\alpha\check\gamma'|_{\varchi(\VV)}$ and $\varchi(\ZZ)=w_\alpha \varchi(\VV)$ it follows from (\ref{inductionstep}) that $\check\alpha|_{\varchi(\ZZ)}=\check\gamma'|_{\varchi(\ZZ)}$.

In cases $T$ split and non-split, we will see in the \emph{case-by-case} analysis that the statement of Proposition \ref{gluedfrom2} is correct when $\ZZ=$the node $\VV$ of $\mathfrak g$ which is immediately higher than the lowest node.
 Now, if the hypothesis is satisfied by a node $\VV$ with $\beta$ a root lowering $\VV$ and $\alpha$ a root raising it to a node $\ZZ$, then (in the case $T$ split, the non-split case being completely analogous):
\begin{eqnarray*} `b^Z_w(\tilde\chi)=  \frac{1-q^{-1} e^{-\check\alpha}}{1- e^{-\check\alpha}}(\chi) \cdot  (- e^{w_\alpha\check\gamma'}) \frac{1-q^{-1}e^{-\check\beta}}{1-q^{1-\left<\check\rho,w_\alpha\gamma'\right>}e^{\check\beta}} \cdot \\
 \cdot  \frac{1-q^{1-\left<\check\rho,w_\alpha\gamma'\right>}e^{(w_\alpha ww_\alpha)\check\beta}}{1-q^{-1}e^{-(w_\alpha w w_\alpha)\check\beta}}(^{w_\alpha}\chi) \cdot \frac{1-e^{-w\check\alpha}}{1-q^{-1}e^{-w\check\alpha}}(\chi).
\end{eqnarray*}

Once again, it suffices to prove (\ref{inductionstep}). Indeed, then the factor ${1- e^{-\check\alpha}}(\chi)$ cancels ${1-q^{-1}e^{-\check\beta}}(^{w_\alpha}\chi)$. Moreover, since $w_\alpha\gamma'\in\varchi(\VV)\iff \gamma'\in\varchi(\ZZ)$ we get $\left<\gamma',\check\alpha\right> = \left<\gamma',w_\alpha\check\beta\right> = \left<w_\alpha\gamma',\check\beta\right>=1$ by (\ref{inductionstep}) and the induction hypothesis. By (\ref{inductionstep}) we also have $e^{-\check\alpha}(^{w^{-1}}\chi)=q^{-1}e^{-w_\alpha\check\beta}(^{w^{-1}}\chi)$ and therefore the factor ${1-e^{-w\check\alpha}}(\chi)$ cancels the factor ${1-q^{-1}e^{-(w_\alpha w w_\alpha)\check\beta}}(^{w_\alpha}\chi)$. Finally, we have $$q^{1-\left<w_\alpha\gamma',\check\rho\right>}e^{\check\beta}(^{w_\alpha}\chi) = q^{1-\left<\gamma',\check\rho\right>} q^{\left<\gamma',\check\alpha\right>}e^{w_\alpha\check\beta}(\chi) = q^{1-\left<\gamma',\check\rho\right>} e^{\check\alpha}(\chi)$$ since $\gamma'\in\varchi(\ZZ)$ and hence $\left<\check\alpha,\
gamma'\right>=\left<\frac{\check\gamma}{2},\gamma\right>=1$. This completes, up to the case-by-case analysis that follows, the proof of the proposition.
\end{proof}

Now we study one-by-one the ``simple'' varieties of Proposition \ref{gluedfrom}, in order to establish the remaining points of its proof. More precisely, let $\ZZ$ be an orbit in a given path $\mathfrak g$ and $\alpha,\beta$ two simple roots such that $\ZZ\PP_{\alpha\beta}/\mathcal R(\PP_{\alpha\beta})$ is one of the varieties listed in Proposition \ref{gluedfrom}. Let $\HH_{\alpha\beta}$ be the stabilizer mod $\UU_{P_{\alpha\beta}}$ of a point on $\ZZ\PP_{\alpha\beta}$ (hence, by definition, $\HH_{\alpha\beta}$ is a subgroup of the reductive quotient $\LL_{\alpha\beta}$), and set  $\HH'=[\HH_{\alpha\beta},\HH_{\alpha\beta}]$ and $\GG'=[\LL_{\alpha\beta},\LL_{\alpha\beta}]$. We can prove the remaining statements of Proposotion \ref{gluedfrom2} by looking at the variety $\HH'\backslash\GG'$; the latter may not be spherical for $\GG'$, in which case we will treat it as a spherical variety for $\GG'\,\,\times $ the connected component of $(\mathcal N_{\LL_{\alpha\beta}}(\HH_{\alpha\beta}) \cap \GG')/\HH'$.

Recall that the formulas of section \ref{secfe} (and in particular (\ref{tsplitetaunramified})) do not depend just on the space $\HH'\backslash\GG'$ but also on some modular characters of subgroups of $\PP_{\alpha\beta}$. Fortunately, however, it will turn out that all coweights which appear in our computation are in the span of the coroots of $\GG'$, thus rendering irrelevant the difference between those subgroups of $\PP_{\alpha\beta}$ and their image in $\LL_{\alpha\beta}$.

\begin{remark}
 Assume that $(\ZZ,\alpha)$ is of type $T$ split, with $\DD,\DD'$ the divisors of smaller rank in $\ZZ\PP_\alpha$. While $\check v_D$ is an element of $\varchi(\ZZ)^*$, it will be convenient in the computations which follow to choose a lift of $\check v_D$ to $\Hom(\varchi(\AA),\QQ)$ -- to be denoted again by $\check v_D$ (and similarly for $\check v_{D'}$). This will not affect the computation of $e^{\check v_D}(\tilde\chi\nu \delta^{\frac{1}{2}}\delta_{\left(U_{P_\alpha}\right)_\xi}^{-1})$ of (\ref{tsplitetaunramified}) since, as we already remarked, $\tilde\chi\nu \delta^{\frac{1}{2}}\delta_{\left(U_{P_\alpha}\right)_\xi}^{-1}\in A_Z^*$, but it allows us to split it into a product of factors, for instance we will be able to evaluate $e^{\check v_D}(\delta^{\frac{1}{2}}\delta_{\left(U_{P_\alpha}\right)_\xi}^{-1})$. Moreover, we will for simplicity replace the symbol $\equiv$ by $=$, since it will be clear that whenever a simple co-root $\check\alpha$ appears we mean its image in $\mathcal Q$.
\end{remark}

\subsection{$\SSL_2\backslash\SSL_3$.}

Below is Knop's graph for the orbits of maximal rank:

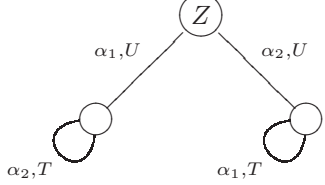
\entrymodifiers={++[o][F-]}
\xymatrix{
*{} & *{} & *{} & Z \ar@{-}[dl]_{\alpha_1,U} \ar@{-}[dr]^{\alpha_2,U} \\ 
*{} & *{} & \ar@{-}@(d,l)[]^{\alpha_2,T}  & *{} & \ar@{-}@(d,l)[]^{\alpha_1,T} & *{} &  *{} \\
}\noindent

We will use the path on the left to compute the coefficient $`b^Z_{w_\gamma}$. It is easy to compute that for the orbits $D,D'$ corresponding to the bottom-left reflection of type $T$ we have: $\check v_D=-\check\alpha_1, \check v_{D'}=\check\alpha_1+\check\alpha_2$, $e^{\check v_D}(\delta^\frac{1}{2}\delta_{(U_{P_{\alpha_2}})_\xi}^{-1})= q^{-1}, e^{\check v_{D'}}(\delta^\frac{1}{2}\delta_{(U_{P_{\alpha_2}})_\xi}^{-1})=1$. Notice that, because $\check v_D\ne\check v_{D'}$, the spherical root of this variety is necessarily of type $T$ split. Therefore, we get:
$$`b_{w}^Z(\chi) = \frac{1-q^{-1}e^{-\check\alpha_1}}{1-e^{-\check\alpha_1}}(\chi) \cdot $$ $$\cdot
(-e^{\check\alpha_2})\frac{1-e^{\check\alpha_1}}{1-q^{-1}e^{-\check\alpha_1}} \cdot  \frac{1-q^{-1}e^{-\check\alpha_1-\check\alpha_2}}{1-e^{\check\alpha_1+\check\alpha_2}}(^{w_{\alpha_1}}\chi) \cdot 
 \frac{1-e^{-w\check\alpha}}{1-q^{-1}e^{-w\check\alpha}}(\chi)$$
$$=-e^{\check\alpha_1+\check\alpha_2} \frac{1-q^{-1}e^{-\check\alpha_1}}{1-q^{-1}e^{\check\alpha_1}}\frac{1-q^{-1}e^{-\check\alpha_2}}{1-q^{-1}e^{\check\alpha_2}}(\chi).$$
This agrees with the statement of Proposition \ref{gluedfrom2}.

\subsection{$\SSp_2\backslash\SSp_4$.}\label{sp4modsp2}

This is a spherical variety for the group $\Gm\times\SSp_4$. (Of course, $\SSp_2=\SSL_2$, but we write $\SSp_2$ to emphasize the way it embeds -- namely, into the stabilizer of a 2-dimensional symplectic subspace.) Here is Knop's graph for orbits of maximal rank:

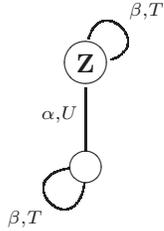
\xymatrix{
*{} & *{} & *{} &  \ZZ \ar@{-}@(u,r)[]^{\beta,T} \ar@{-}[d]_{\alpha,U}  \\ 
*{} & *{}  & *{} & \ar@{-}@(d,l)[]^{\beta,T} & *{} & *{} & *{} \\
}\noindent
where $\alpha$ is the long root and $\beta$ is the short one.

We are interested in the spherical root $\alpha+\beta$, since the spherical root $\beta$ belongs to the type $\TT\backslash\SSL_2$. One computes that if $D,D'$ are the orbits of co-rank one corresponding to the reflection of type $T$ at the bottom, then $\check v_D=-\check\alpha$, $\check v_{D'}=\check\alpha+\check\beta$, $e^{\check v_D}(\delta^\frac{1}{2}\delta_{U_{(P_{\beta}})_\xi}^{-1})=q^{-1}$, $e^{\check v_{D'}}(\delta^\frac{1}{2}\delta_{U_{(P_{\beta}})_\xi}^{-1})=1$. Again, for the same reason as above, the spherical root is necessarily of type $T$ split. Therefore:$$ `b_w^Z(\chi) = \frac{1-q^{-1}e^{-\check\alpha}}{1-e^{-\check\alpha}}(\chi) \cdot $$ $$\cdot
(-e^{\check\beta})\frac{1-e^{\check\alpha}}{1-q^{-1}e^{-\check\alpha}} \cdot  \frac{1-q^{-1}e^{-\check\alpha-\check\beta}}{1-e^{\check\alpha+\check\beta}}(^{w_{\alpha}}\chi) \cdot 
 \frac{1-e^{-w\check\alpha}}{1-q^{-1}e^{-w\check\alpha}}(\chi)=$$
$$=-e^{2\check\alpha+\check\beta} \frac{1-q^{-1}e^{-\check\alpha}}{1-q^{-1}e^{\check\alpha}}\frac{1-q^{-1}e^{-\check\alpha-\check\beta}}{1-q^{-1}e^{\check\alpha+\check\beta}}(\chi)$$
which agrees with Proposition \ref{gluedfrom2}.

\subsection{$\Ga\times\SSp_2\backslash\SSp_4$.}

Here the subgroup $\Ga$ is the unipotent radical of the ``second'' $\SSp_2$ inside $\SSp_4$. Knop's graph is as follows in this case:

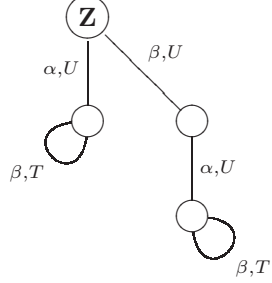
\xymatrix{
*{} & *{} & *{} &  \ZZ \ar@{-}[dr]^{\beta,U} \ar@{-}[d]_{\alpha,U}  \\ 
*{} & *{} & *{} & \ar@{-}@(d,l)[]^{\beta,T} & \ar@{-}[d]^{\alpha,U} & *{} & *{} \\
*{} & *{} & *{} &           *{}         &  \ar@{-}@(d,r)[]_{\beta,T}\\
}\noindent
where $\alpha$ is the long root and $\beta$ is the short one.

The left-most path is identical to the one of the previous example, including the same values for $\check v_D, e^{\check v_D}(\delta^\frac{1}{2}\delta_{U_{(P_{\alpha_2}})_\xi}^{-1})$ etc.,  whence again:
$$ `b_w^Z(\chi) = -e^{2\check\alpha+\check\beta} \frac{1-q^{-1}e^{-\check\alpha}}{1-q^{-1}e^{\check\alpha}}\frac{1-q^{-1}e^{-\check\alpha-\check\beta}}{1-q^{-1}e^{\check\alpha+\check\beta}}(\chi).$$
The path on the right does not correspond to a simple reflection $w_\gamma$, because it defines a decomposition of $w_\gamma$ which is longer than the length of $w_\gamma$.

\subsection{$\SSp_2\times\SSp_2\backslash\SSp_4$.}\label{sp4modsp2xsp2}

Here we have two cases over $k$, namely: 
\begin{itemize}
 \item  the variety $\SSp_2\times\SSp_2\backslash\SSp_4$ and 
 \item the variety $(\operatorname{Res}_{E/k}\SSp_2)\backslash\SSp_4$ where $E/k$ is a quadratic field extension and $\operatorname{Res}_{E/k}$ denotes restriction of scalars. 
\end{itemize}
Here is Knop's graph for orbits of maximal rank:

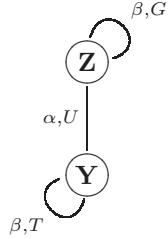
\xymatrix{
*{} & *{} & *{} &  \ZZ \ar@{-}@(u,r)[]^{\beta,G} \ar@{-}[d]_{\alpha,U}  \\ 
*{} & *{}  & *{} & \YY \ar@{-}@(d,l)[]^{\beta,T} & *{} & *{} & *{} \\
}\noindent
where $\alpha$ is the long root and $\beta$ is the short one. The reflection of type $T$ at the bottom is of split type in the first case and of non-split type in the second.

Again, for the divisors of co-rank one at the bottom we have: $\check v_D=-\check\alpha$, $\check v_{D'}=\check\alpha+\check\beta$, $e^{\check v_D}(\delta^\frac{1}{2}\delta_{U_{(P_{\beta}})_\xi}^{-1})=q^{-1}$, $e^{\check v_{D'}}(\delta^\frac{1}{2}\delta_{U_{(P_{\beta}})_\xi}^{-1})=1$. Therefore, as before we have in the split case:

$$ `b_w^Z(\chi) = -e^{2\check\alpha+\check\beta} \frac{1-q^{-1}e^{-\check\alpha}}{1-q^{-1}e^{\check\alpha}}\frac{1-q^{-1}e^{-\check\alpha-\check\beta}}{1-q^{-1}e^{\check\alpha+\check\beta}}(\chi).$$
which agrees with Proposition \ref{gluedfrom2}.

Notice that here $\chi=e^\frac{\beta}{2} \chi'$ with $\chi'\in A_Z^*$ and $\check\alpha\equiv\check\alpha+\check\beta=\frac{\check\gamma'}{2}$ on $\varchi(\ZZ)$, therefore the above can also be written:
$$`b_w^Z(\chi) = -e^{\check\gamma'} \frac{1-q^{-\frac{3}{2}}e^{-\frac{\check\gamma'}{2}}}{1-q^{-\frac{3}{2}}e^{\frac{\check\gamma'}{2}}}\frac{1-q^{-\frac{1}{2}}e^{-\frac{\check\gamma'}{2}}}{1-q^{-\frac{1}{2}}e^{\frac{\check\gamma'}{2}}}(\chi').$$

Similarly, in the non-split case we will have:
$$`b_w^Z(\chi) = -e^{\check\gamma'} \frac{1-q^{-\frac{3}{2}}e^{-\frac{\check\gamma'}{2}}}{1-q^{-\frac{3}{2}}e^{\frac{\check\gamma'}{2}}}\frac{1+q^{-\frac{1}{2}}e^{-\frac{\check\gamma'}{2}}}{1+q^{-\frac{1}{2}}e^{\frac{\check\gamma'}{2}}}(\chi').$$

\subsection{Horospherical varieties.}

For the varieties $\PP_\alpha\backslash\SSL_3$ or $\PP_\alpha\backslash\SSp_4$ of Proposition \ref{gluedfrom}, the graphs look as follows: 

\indent 
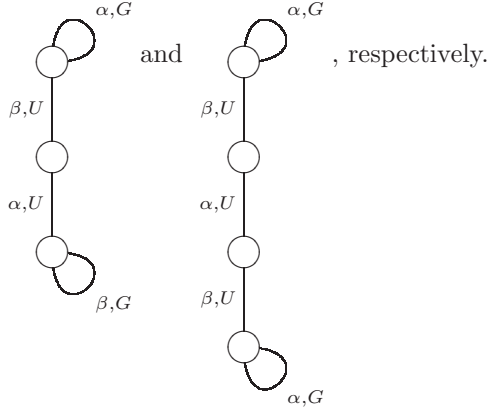
\xymatrix{
\ar@{-}@(u,r)[]^{\alpha,G} \ar@{-}[d]_{\beta,U} \\ \ar@{-}[d]_{\alpha,U} & *{}\\ \ar@{-}@(d,r)[]_{\beta,G}& *{}  }
and 
\xymatrix{
\ar@{-}@(u,r)[]^{\alpha,G} \ar@{-}[d]_{\beta,U} \\ \ar@{-}[d]_{\alpha,U}& *{} \\ \ar@{-}[d]_{\beta,U}& *{}\\ \ar@{-}@(d,r)[]_{\alpha,G}& *{} }, respectively. 
\\The validity of (\ref{inductionstep}) at every intermediate node is easily verified.

\subsection{Spherical varieties of rank at most two.}

We now come to the proof of Proposition \ref{ranktwo}. The starting point is the classification of wonderful varieties of rank one by Akhiezer \cite{Ak} and of rank two by Wasserman \cite{Wa}.

\begin{theorem}\label{AW}
 Below is a complete list of the isomorphism classes of homogeneous spherical varieties $\XX'=\HH'\backslash \GG'$ over $\bar k$ with the following properties:
\begin{itemize}
 \item $\GG'$ is semisimple, simply connected, and $\HH'$ is equal to the connected component of its normalizer.
 \item There is a spherical root $\gamma$ whose support is the set of simple roots of $\GG'$.
 \item There are no spherical roots of type $N$.
\end{itemize}
(For simplicity and since everything is over the algebraic closure, we do not use boldface letters in what follows. We also write $k^n$ to denote an $n$-dimensional vector space.)
$$\begin{array}{ccc}
 \GG' && \HH' \\
&\text{Rank one:}\\
\SL_{n+1} && \GL_n \\
\SL_2\times \SL_2 && \SL_2^\diag \\
\SL_4 && \Sp_4 \\
\Spin_{2n+1} && \Spin_{2n} \\
\Spin_{2n+1} && (\SL_n\times\Gm)\ltimes \wedge^2 k^n \\
\Spin_7 && G_2 \\
\Sp_{2n} && \Sp_2\times\Sp_{2n-2} \\
\Sp_{2n} && (\Gm\ltimes k)\times\Sp_{2n-2} \\
\Spin_{2n} && \Spin_{2n-1} \\
F_4 && \Spin_9 \\
G_2 && \SL_3 \\
G_2 && (\Gm\times\SL_2)\ltimes (k^1\oplus k^2)\\
&\text{Rank two:} \\
\Spin_9 && \Spin_7\\
\Sp_{2n} && \Gm\times\Sp_{2n-2}\\
G_2 && (\Gm\times\SL_2)\ltimes k^2.
\end{array}$$
(In the case of $G'=\Spin_{2n}$ and $H'=\Spin_{2n-1}$, $n=4$, this includes the ``non-obvious'' embedding of $\Spin_7$ obtained from the ``obvious'' one by applying the outer automorphism of $\Spin_8$. The spherical subgroup $\Spin_7\hookrightarrow\Spin_9$ is not inside of a Levi subgroup, but obtained via this ``non-obvious'' map to $\Spin_8$.)
\end{theorem}

\begin{proof}
If we add the extra condition ``the rank of $\XX'$ is at most two'', then it follows by inspection of the tables of \cite{Wa}. As we explained in Lemma \ref{rankatmosttwo}, the same tables imply that for any spherical root $\gamma$ the variety $\XX_\gamma$ cannot have rank larger than two.
\end{proof}

We now come to the proof of Proposition \ref{ranktwo}. Given a spherical variety $\XX$ and a Brion path $\mathfrak g$ for a simple spherical reflection $w_\gamma$, the labels of all edges in $\mathfrak g$ are simple roots belonging to the support of $\gamma$; indeed, this follows from the fact that the decomposition of $w_\gamma$ defined by $\mathfrak g$ is reduced. Therefore, for the proof of Proposition \ref{ranktwo} we may replace $\XX$ by the spherical variety $\mathring \XX \PP_\gamma/\mathcal R(\PP_\gamma)$, which, as we have seen (Lemma \ref{rankatmosttwo}), is of rank at most two. We may also replace the adjoint group of $\LL_\gamma$ by its simply connected cover, and therefore we have reduced the proof of the proposition to the spherical varieties $\XX'$ of the groups $\GG'$ of the above list (when $\GG'$ is a classical group).
Hence, it suffices to check that for each of those varieties, each normalized spherical root $\gamma$ thereof and every Brion path $\mathfrak g'$ corresponding to $w_\gamma$ there is a Brion path $\mathfrak g$ corresponding to $w_\gamma$ with the same first edge as $\mathfrak g'$ which is glued out of the spherical varieties of Proposition \ref{gluedfrom}. We notice that the spherical roots for the examples below can be read off from the tables of \cite{Wa}, and that for a quasi-affine spherical variety $\XX$ the set $\Delta(\XX)$ consists precisely of the simple roots of $\GG$ which are orthogonal to $\varchi(\XX)$ \cite[Lemma 3.1]{KnAs}. There is nothing to check for $\Gm\backslash\SSL_2$ and $\SSL_2\backslash\SSL_2\times\SSL_2$.

\subsection{$\GGL_n\backslash\SSL_{n+1}$, $n\ge 2$.}\label{ssglsl}

The normalized spherical root here is $\gamma=\alpha_1+\alpha_2+\dots+\alpha_n$, we have $\Delta(\XX)=\{\alpha_2,\alpha_3,\dots,\alpha_{n-1}\}$ and a path corresponding to it is the following:

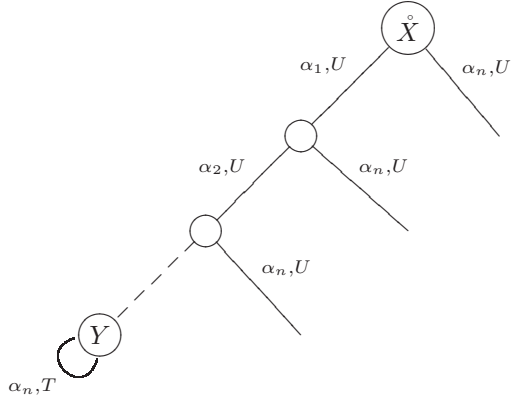
\xymatrix{
*{} & *{} & *{} & \mathring X \ar@{-}[dl]_{\alpha_1,U} \ar@{-}[dr]^{\alpha_n,U} & *{}\\ 
*{} & *{} & \ar@{-}[dl]_{\alpha_2,U} \ar@{-}[dr]^{\alpha_n,U}& *{}& *{}& *{}& *{}\\
*{} & \ar@{-}[dr]^{\alpha_n,U} & *{}& *{}& *{}& *{}\\
Y \ar@{-}@{--}[ur] \ar@{-}@(d,l)[]^{\alpha_n,T}  & *{}& *{}& *{}\\
}\noindent
where the arrows not shown are reflections of type $G$.

Clearly, this path is glued from $\SSL_2\backslash\SSL_3$ (at the bottom) and $\PP_\alpha\backslash\SSL_3$. There is also a similar path starting with $\alpha_n$. We mention that here we have:

\begin{equation}
B_{w_\gamma}(\chi)= - e^{\check\gamma}  \frac{(1-q^{-\frac{n}{2}}e^{-\check\alpha_1})(1-q^{-\frac{n}{2}}e^{-\check\alpha_n})} {(1-q^{-\frac{n}{2}}e^{\check\alpha_1})(1-q^{-\frac{n}{2}}e^{\check\alpha_n})}(\chi) . 
\end{equation}

\subsection{$\SSp_4\backslash\SSL_4$.}

Here $\Delta(\XX)=\{\alpha_1,\alpha_3\}$ and the normalized spherical root is $\gamma=\alpha_1+2\alpha_2+\alpha_3$.
Knop's graph is as follows:

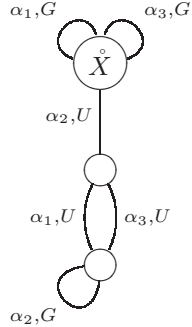
\xymatrix{
*{} & *{} & *{} & \mathring X \ar@{-}@(u,l)[]_{\alpha_1,G} \ar@{-}[d]_{\alpha_2,U} \ar@{-}@(u,r)[]^{\alpha_3,G} \\ 
*{} & *{} & *{} &  \ar@{-}@/_/[d]_{\alpha_1,U}  \ar@{-}@/^/[d]^{\alpha_3,U} & *{} & *{} &  *{}\\
*{} & *{}  & *{} & \ar@{-}@(d,l)[]^{\alpha_2,G} & *{} & *{} & *{} \\
}

Therefore, this path is glued from varieties of the form $\PPGL_2\backslash \PPGL_2\times\PPGL_2$ and $\PP_\alpha\backslash\SSL_3$. Here we have:
\begin{equation}
B_{w_\gamma}(\chi) = -e^{\check\gamma}\frac{1-q^{-2}e^{-\check\gamma}}{1-q^{-2}e^{\check\gamma}}(\chi).
\end{equation}

\subsection{$\SSpin_{2n}\backslash\SSpin_{2n+1}$.}

We denote by $\alpha_1,\dots,\alpha_{n-1}$ the long roots and by $\alpha_n$ the short one. The normalized spherical root is $\gamma=\alpha_1+\alpha_2+\dots+\alpha_n$, and $\Delta(\XX)=\{\alpha_2,\dots,\alpha_n\}$. A Brion path is the following:

\xymatrix{
*{} & *{} & *{} & \mathring X \ar@{-}[dl]_{\alpha_1,U} & *{}\\ 
*{} & *{} & \ar@{-}[dl]_{\alpha_2,U} & *{}& *{}& *{}& *{}\\
*{} & \ar@{-}[dr]_{\alpha_n,U}& *{}& *{}& *{}& *{}\\
Y \ar@{-}@{--}[ur]^{\alpha_{n-1},U} \ar@{-}@(d,l)[]^{\alpha_n,T}  & *{}& *{}& *{}\\
}\noindent

It is glued from $\SSp_2\times\SSp_2\backslash\SSp_4$ (at the bottom) and $\PP_\alpha\backslash\SSL_3$. We have:
\begin{equation}
B_{w_\gamma}(\chi)= - e^{\check\gamma}  \frac{(1-q^{-n}e^{-\frac{\check\gamma}{2}})(1-q^{-\frac{1}{2}}e^{-\frac{\check\gamma}{2}})} {(1-q^{-n}e^{\frac{\check\gamma}{2}})(1-q^{-\frac{1}{2}}e^{\frac{\check\gamma}{2}})}(\chi) . 
\end{equation}

\subsection{$(\SSL_n\times\Gm)\ltimes\wedge^2 k^n\backslash\SSpin_{2n+1}$.}

With notation as above, the normalized spherical root is the same as above (namely, $\gamma=\alpha_1+\dots+\alpha_n$) and $\Delta(\XX)=\{\alpha_2,\dots,\alpha_{n-1}\}$. A Brion path is the following:

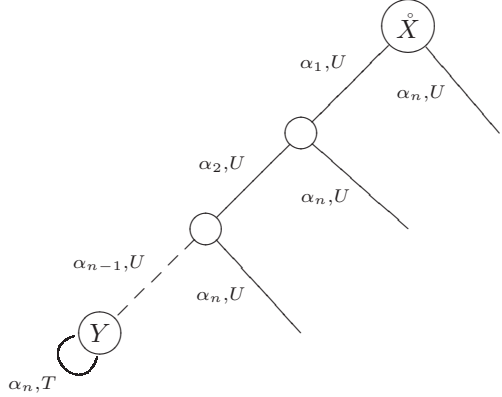
\xymatrix{
*{} & *{} & *{} & \mathring X \ar@{-}[dl]_{\alpha_1,U}  \ar@{-}[dr]_{\alpha_n,U}& *{}\\ 
*{} & *{} & \ar@{-}[dl]_{\alpha_2,U}  \ar@{-}[dr]_{\alpha_n,U}& *{}& *{}& *{}& *{}\\
*{} & \ar@{-}[dr]_{\alpha_n,U}& *{}& *{}& *{}& *{}\\
Y \ar@{-}@{--}[ur]^{\alpha_{n-1},U} \ar@{-}@(d,l)[]^{\alpha_n,T}  & *{}& *{}& *{}\\
}\noindent

It is glued from $(\Gm\ltimes\Ga)\times\SSp_2\backslash\SSp_4$ (at the bottom) and $\PP_\alpha\backslash\SSL_3$. The edge $\alpha_n$ on the top-right cannot lead to $w_\gamma$ since $\alpha_n$ is orthogonal to $\gamma$ and hence the length of any such path would be larger than the length of $w_\gamma$.

Here we have:
\begin{equation}
B_{w_\gamma}(\chi)= - e^{\check\gamma}  \frac{(1-q^{-\frac{n}{2}}e^{-\check\alpha_1})(1-q^{-\frac{n}{2}}e^{-\check\alpha_1-\check\alpha_n})} {(1-q^{-\frac{n}{2}}e^{\check\alpha_1})(1-q^{-\frac{n}{2}}e^{\check\alpha_1+\check\alpha_n})}(\chi) . 
\end{equation}

\subsection{$\GG_2\backslash\SSpin_7$.}

The normalized spherical root is $\gamma=\alpha_1+2\alpha_2+3\alpha_3$, and $\Delta(\XX)=\{\alpha_1,\alpha_2\}$. Knop's graph looks as follows:

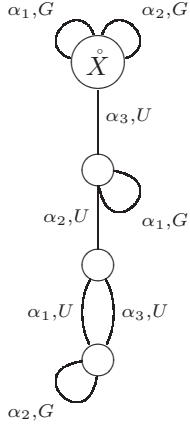
\xymatrix{
 *{} &*{} &\mathring X \ar@{-}@(u,l)[]_{\alpha_1,G} \ar@{-}@(u,r)[]^{\alpha_2,G} \ar@{-}[d]^{\alpha_3,U} &*{} \\
 *{} &*{} &\ar@{-}@(d,r)[]_{\alpha_1,G} \ar@{-}[d]_{\alpha_2,U} &*{} &*{} \\
 *{} &*{} &\ar@{-}@/_/[d]_{\alpha_1,U}  \ar@{-}@/^/[d]^{\alpha_3,U} & *{} & *{} &  *{}\\
 *{} &*{} &\ar@{-}@(d,l)[]^{\alpha_2,G} & *{} & *{} & *{} 
}\noindent

It is glued from the varieties $\PP_\alpha\backslash\SSp_4$ (where $\alpha$ denotes the long root), $\PP_\alpha\backslash\SSL_3$ and $\PPGL_2\backslash \PPGL_2\times\PPGL_2$. We have:
\begin{equation}
B_{w_\gamma}(\chi) = -e^{\check\gamma}\frac{1-q^{-3}e^{-\check\gamma}}{1-q^{-3}e^{\check\gamma}}(\chi).
\end{equation}

\subsection{$\SSp_2\times\SSp_{2n-2}\backslash\SSp_{2n}$.}

Denote by $\alpha_1,\dots,\alpha_{n-1}$ the short roots and by $\alpha_n$ the long one. Then the normalized spherical root is $\gamma= \alpha_1+2\alpha_2+\dots+2\alpha_{n-1}+\alpha_n$, and $\Delta(\XX)=\{\alpha_1,\alpha_3,\alpha_4,\dots,\alpha_n\}$. A Brion path is as follows:

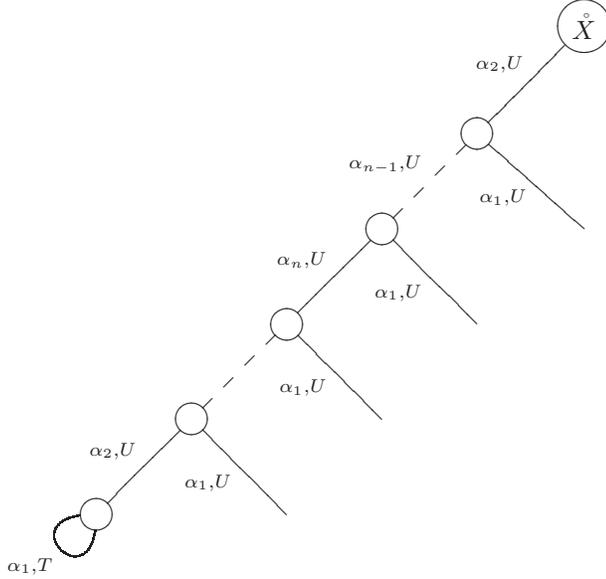
\xymatrix{
*{} &*{} &*{} & *{} & *{} & \mathring X \ar@{-}[dl]_{\alpha_2,U} & *{}\\ 
*{} &*{} &*{} & *{} &   \ar@{-}[dr]_{\alpha_1,U}& *{}& *{}& *{}& *{}\\
*{} &*{} &*{} & \ar@{-}[dl]_{\alpha_n,U} \ar@{-}@{--}[ur]^{\alpha_{n-1},U} \ar@{-}[dr]_{\alpha_1,U}  & *{}& *{}& *{}\\
*{} &*{} & \ar@{-}[dr]_{\alpha_1,U} &*{} &*{} \\
*{} & \ar@{-}@{--}[ur] \ar@{-}[dl]_{\alpha_2,U} \ar@{-}[dr]_{\alpha_1,U}&*{} &*{} &*{} \\
\ar@{-}@(d,l)[]^{\alpha_1,T}  & *{}& *{}& *{}\\
}\noindent

It is glued from $\GGL_2\backslash\SSL_3$ (at the bottom), $\PP_\alpha\backslash\SSp_4$ (where $\alpha$ is the long root) and $\PP_\alpha\backslash\SSL_3$, unless $n=2$, in which case it coincides with $\SSp_2\times\SSp_2\backslash\SSp_4$. We have here:
\begin{equation}
B_{w_\gamma}(\chi)= - e^{\check\gamma}  \frac{(1-q^{-n+\frac{3}{2}}e^{-\frac{\check\gamma}{2}})(1-q^{-n+\frac{1}{2}}e^{-\frac{\check\gamma}{2}})} {(1-q^{-n+\frac{3}{2}}e^{\frac{\check\gamma}{2}})(1-q^{-n+\frac{1}{2}}e^{\frac{\check\gamma}{2}})}(\chi) . 
\end{equation}

\subsection{$(\Gm\ltimes k)\times\SSp_{2n-2}\backslash\SSp_{2n}$.}

The normalized spherical root is the same as above, and $\Delta(\XX)=\{\alpha_3,\alpha_4,\dots,\alpha_n\}$. A Brion path is the following:

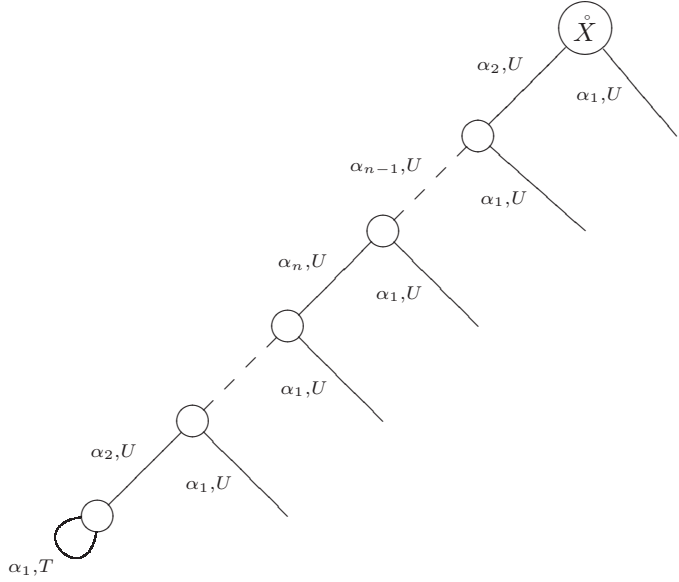
\xymatrix{
*{} &*{} &*{} & *{} & *{} & \mathring X \ar@{-}[dl]_{\alpha_2,U} \ar@{-}[dr]_{\alpha_1,U} & *{}\\ 
*{} &*{} &*{} & *{} &   \ar@{-}[dr]_{\alpha_1,U}& *{}& *{}& *{}& *{}\\
*{} &*{} &*{} & \ar@{-}[dl]_{\alpha_n,U} \ar@{-}@{--}[ur]^{\alpha_{n-1},U} \ar@{-}[dr]_{\alpha_1,U}  & *{}& *{}& *{}\\
*{} &*{} & \ar@{-}[dr]_{\alpha_1,U} &*{} &*{} \\
*{} & \ar@{-}@{--}[ur] \ar@{-}[dl]_{\alpha_2,U} \ar@{-}[dr]_{\alpha_1,U}&*{} &*{} &*{} \\
\ar@{-}@(d,l)[]^{\alpha_1,T}  & *{}& *{}& *{}\\
}\noindent
which, again, is glued from $\GGL_2\backslash\SSL_3$ (at the bottom), $\PP_\alpha\backslash\SSp_4$ (where $\alpha$ is the long root) and $\PP_\alpha\backslash\SSL_3$ (unless $n=2$, a case which was treated before). The edge labelled $\alpha_1$ on the top right cannot lead to $w_\gamma$, since $\alpha_1$ is orthogonal to $\gamma$.

Here we have:
\begin{equation}
B_{w_\gamma}(\chi)= - e^{\check\gamma}  \frac{(1-q^{-n+1}e^{-\check\alpha_2})(1-q^{-n+1}e^{-\check\alpha_1-\check\alpha_2})} {(1-q^{-n+1}e^{\check\alpha_2})(1-q^{-n+1}e^{\check\alpha_1+\check\alpha_2})}(\chi) . 
\end{equation}

\subsection{$\SSpin_{2n-1}\backslash\SSpin_{2n}$.}

Denote the simple roots by $\alpha_1,\alpha_2,\dots,\alpha_n$, where $\alpha_{n-2}$ is the root neighboring with three others ($\alpha_{n-3},\alpha_{n-1}$ and $\alpha_n$). Assume $n\ge 3$, the case $n=3$ having been treated in \S \ref{ssglsl}. The normalized spherical root is $\gamma= 2\alpha_1+\dots+2\alpha_{n-2}+\alpha_{n-1}+\alpha_n$, and $\Delta(\XX)=\{\alpha_2,\dots,\alpha_n\}$. Knop's graph looks as follows (omitting type $G$ reflections):

\xymatrix{
 *{} & \mathring X \ar@{-}[d]_{\alpha_1,U}  & *{}\\ 
 *{} & \ar@{-}[d]_{\alpha_2,U}\\
 *{} & \ar@{-}[d]_{\alpha_n,U}& *{}& *{}& *{}& *{}\\
 *{} &\ar@{-}@/_/[d]_{\alpha_1,U}  \ar@{-}@/^/[d]^{\alpha_3,U} & *{} & *{} &  *{}\\
 *{} &\\
}\noindent

It is glued from $\PPGL_2\backslash\PPGL_2\times\PPGL_2$ and $\PP_\alpha\backslash\SSL_3$.
We have:
\begin{equation}
B_{w_\gamma}(\chi) = -e^{\check\gamma}\frac{1-q^{-n+1}e^{-\check\gamma}}{1-q^{-n+1}e^{\check\gamma}}(\chi).
\end{equation}

\subsection{$\SSpin_7\backslash\SSpin_9$.} 

There are 2 normalized spherical roots, $\gamma=\alpha_1+\alpha_2+\alpha_3+\alpha_4$ and $\alpha_2+2\alpha_3+3\alpha_4$ but, as in \S \ref{sp4modsp2}, we are interested only in $\gamma$. Here $\Delta(\XX)=\{\alpha_2,\alpha_3\}$.

A Brion path for $w_\gamma$ is the following:

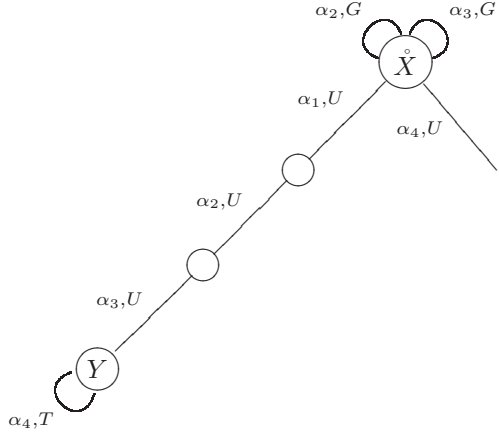
\xymatrix{
*{} & *{} & *{} & \mathring X \ar@{-}@(u,l)[]_{\alpha_2,G} \ar@{-}@(u,r)[]^{\alpha_3,G} \ar@{-}[dl]_{\alpha_1,U}  \ar@{-}[dr]_{\alpha_4,U}& *{}\\ 
*{} & *{} & \ar@{-}[dl]_{\alpha_2,U}  & *{}& *{}& *{}& *{}\\
*{} & \ar@{-}[dl]_{\alpha_3,U} & *{}& *{}& *{}& *{}\\
Y \ar@{-}@(d,l)[]^{\alpha_4,T}  & *{}& *{}& *{}\\
}\noindent
where we have only drawn the arrows of interest.

It is glued from $(\Gm\ltimes \Ga)\times\SSp_2\backslash\SSp_4$ (at the bottom) and $\PP_\alpha\backslash\SSL_3$. The arrow labelled $\alpha_4$ on the top cannot lead to $w_\gamma$ since $\alpha_4$ is orthogonal to $w_\gamma$. We have:
\begin{equation}
B_{w_\gamma}(\chi)= - e^{\check\gamma}  \frac{(1-q^{-2}e^{-\check\alpha_1})(1-q^{-2}e^{-\check\alpha_1-\check\alpha_4})} {(1-q^{-2}e^{\check\alpha_1})(1-q^{-2}e^{\check\alpha_1+\check\alpha_4})}(\chi) . 
\end{equation}

\subsection{$\Gm\times\SSp_{2n-2}\backslash\SSp_{2n}$.}

There are 2 normalized spherical roots, $\alpha_1$ and $\gamma=\alpha_1+2\alpha_2+\dots+2\alpha_{n-1}+\alpha_n$ but, as in \S \ref{sp4modsp2}, we are interested only in $\gamma$.
We have $\Delta(\XX)=\{\alpha_3,\alpha_4,\dots,\alpha_n\}$. A Brion path for $w_\gamma$ is the following:

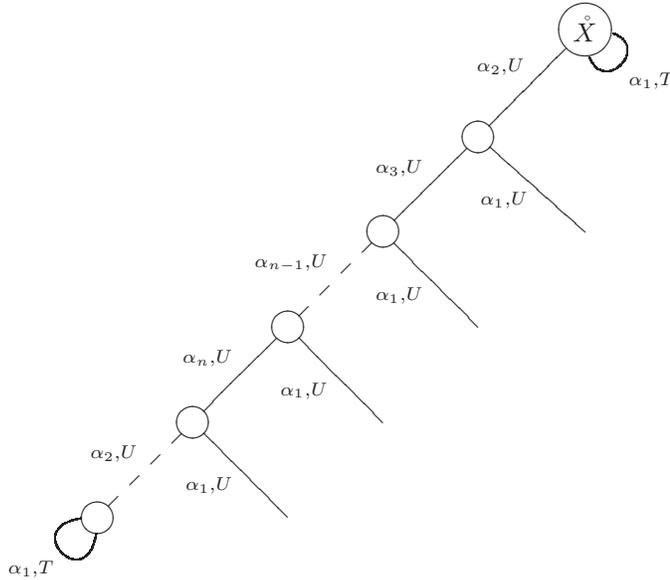
\xymatrix{
*{} &*{} &*{} & *{} & *{} & \mathring X \ar@{-}[dl]_{\alpha_2,U} \ar@{-}@(d,r)_{\alpha_1,T}& *{}\\ 
*{} &*{} &*{} & *{} & \ar@{-}[dl]_{\alpha_3,U}  \ar@{-}[dr]_{\alpha_1,U}& *{}& *{}& *{}& *{}\\
*{} &*{} &*{} &  \ar@{-}[dr]_{\alpha_1,U}& *{}& *{}& *{}& *{}\\
*{} &*{} &\ar@{-}[dl]_{\alpha_n,U} \ar@{-}@{--}[ur]^{\alpha_{n-1},U} \ar@{-}[dr]_{\alpha_1,U}  & *{}& *{}& *{}\\
*{} & \ar@{-}[dr]_{\alpha_1,U} &*{} &*{} \\
\ar@{-}@{--}[ur]^{\alpha_2,U} \ar@{-}@(d,l)[]^{\alpha_1,T}  & *{}& *{}& *{}\\
}\noindent
which is glued from $\SSL_2\backslash\SSL_3$ (at the bottom), $\PP_\alpha\backslash\SSp_4$ (where $\alpha$ is the long root) and $\PP_\alpha\backslash\SSL_3$ (unless $n=2$, which was treated before). We have:
\begin{equation}
B_{w_\gamma}(\chi)= - e^{\check\gamma}  \frac{(1-q^{-n+1}e^{-\check\alpha_2})(1-q^{-n+1}e^{-\check\alpha_1-\check\alpha_2})} {(1-q^{-n+1}e^{\check\alpha_2})(1-q^{-n+1}e^{\check\alpha_1+\check\alpha_2})}(\chi) . 
\end{equation}


\section{The formula}

The goal of this section is to develop a more useful and explicit formula, in which the cocycles of our previous formula have been substituted by simpler expressions, and where a certain quotient of $L$-values plays a distinguished role. In order to do this, we need to make a combinatorial assumption on the colors of a spherical variety, which is very easy to check in each particular case, but we do not know how to prove in general. This assumption is certainly not true in all cases, but we expect it to be true in the case of affine homogeneous spherical varieties. 

Throughout this section we assume that $\XX$ is a spherical variety without type $N$ roots and such that for every $\gamma\in \Delta_X$ the Levi $\LL_\gamma$ is a classical group; in particular, all normalized spherical roots satisfy Statement \ref{STATEMENT}.

\subsection{Relevant and virtual colors} \label{sscolors}

We call a color $\DD$ of $\XX$ \emph{relevant} if it is the codimension-one orbit in a Brion path corresponding to $w_\gamma$, for $\gamma$ a (normalized) spherical root. In this case we say that $\DD$ \emph{belongs} to $\gamma$. A color may belong to more than one spherical root; that happens only if those spherical roots are of type $T$ split, since in the rest of cases the valuation induced by the color determines the spherical root by Statement \ref{STATEMENT}. Vice versa, 

\begin{lemma}
 If $\XX$ is affine, then all colors are relevant. 
\end{lemma}

\begin{proof}[Proof (Sketch)]
Let $\DD$ be an irrelevant color with valuation $\check\theta$, and let $\alpha$ be a simple root raising it to $\mathring \XX$. If $(\mathring\XX,\alpha)$ were of type $T$ then $\DD$ would be relevant, so $(\mathring\XX,\alpha)$ is of type $U$, and hence $\check\theta=\check\alpha$. If $\alpha\in\supp(\gamma)$ for some spherical root $\gamma$, then $\XX_\gamma$ is of rank at most two and one sees by a case-by-case analysis of spherical varieties of rank at most two that $\check\alpha\perp \gamma$. On the other hand, if $\alpha\notin \supp(\gamma)$ then $\left<\check\alpha,\gamma\right>\le 0$. Hence $\left<\check\alpha,\gamma\right>\le 0$ for all $\gamma\in\Delta_X$, which means that $\check\theta=\check\alpha\in\mathcal V$. But if $\XX$ is affine homogeneous then \cite[\S 6]{KnLV} the cone generated by valuations of colors is strictly separated from $\mathcal V$ by a hyperplane, a contradiction.
\end{proof}

The following is also established by reduction to a case-by-case analysis:

\begin{lemma} \label{whichcolorbelongs}
 A color $\DD$ with corresponding valuation $\check\theta$ belongs to a simple spherical root $\gamma$ if and only if $\left<\check\theta,\gamma\right>>0$. 
\end{lemma}

\begin{proof}[Proof (Sketch)]
 If a color $\DD$ belongs to a spherical root $\gamma$ then there exists a simple root $\alpha$ in the support of $\gamma$ that raises $\DD$ to $\mathring\XX$, i.e.\ such that $\mathring\XX\cdot \PP_\alpha \supset \DD$. Vice versa:

\begin{lemma}
 Let $\gamma$ be a spherical root and $\DD$ a color with corresponding valuation $\check v_D$. If $\left<\check v_D,\gamma\right>>0$ then $\mathring\XX\cdot\PP_\gamma\supset \DD$. 
\end{lemma}

\begin{proof}
Assume that this is not the case, hence $\DD$ is $\PP_\gamma$-stable. Let $\YY=\XX/\!/\UU_{P_\gamma}:= \spec k[\XX]^{\UU_{P_\gamma}}$. The image of $\DD$ in $\YY$ is an $\LL_\gamma$-stable divisor, but on the other hand the valuation it induces (still to be denoted by $\check v_D$) satisfies $\left<\check v_D,\gamma\right>>0$, and $\gamma$ is a spherical root of $\YY$. This means that $\check v_D$ does not belong to the cone of invariant valuations for $\YY$ (the negative-dual of the cone of spherical roots), a contradiction.
\end{proof}

Returning to the proof of Lemma \ref{whichcolorbelongs}, we have shown that under either assumption the image of $\DD$ is a color in the rank-one or rank-two spherical variety $\mathring\XX\PP_\gamma/\mathcal R(\PP_\gamma)$, and by inspection of those we see that $\DD$ belongs to $\gamma$ if and only if $\left<\check\theta,\gamma\right>>0$.
\end{proof}

Finally, we notice the following:
\begin{lemma}
 If a color belongs to a spherical root of type $G$, type $T$ non-split, or type $(U,\psi)$, then it belongs only to that spherical root; vice versa, each normalized spherical root $\gamma$ has a unique color belonging to it, unless it is of type $T$, in which case it may have two colors with corresponding valuations satisfying $\check v_D + \check v_{D'} = \gamma$.
\end{lemma}

\begin{proof} The first claim is true because the valuation that the color induces in those cases is proportional to the corresponding spherical coroot, and hence $\le 0$ on all other spherical roots. 
 
The second claim follows by inspection of the spherical varieties of Theorem \ref{AW}.
\end{proof}

Therefore, one can define the \emph{type} of a relevant color as being one among the following: $G$, $T$ split, $T$ non-split or $(U,\psi)$. 

We denote the set of relevant colors by $\mathcal D_R$. We now define the set of \emph{virtual weighted colors} by adding to the set of relevant colors some extra information which is related to Statement \ref{STATEMENT}. The set of virtual weighted colors $\mathcal D_V$ is a set equipped with a map ``$\inv$'' into the set of triples $(\check\theta,\sigma,r)$, where $\check\theta$ is an element of $\varchi(\XX)^*$, $\sigma=\pm$ is a sign and $r\in\frac{1}{2}\mathbb Z$, and minimal with respect to the following properties: 
\begin{enumerate}
\item There is an injection $\phi:\mathcal D_R\to \mathcal D_V$ such that $\inv\circ\phi(\DD)= (\check\theta,\sigma,r)$ where:
\begin{enumerate}
 \item $\check\theta =\check v_D$;
 \item $\sigma = +$, unless $\DD$ is not defined over $k$; the latter happens only when $\DD$ is raised to $\mathring\XX$ by a simple root $\alpha$ of type $T$ non-split, in which case there are two colors with the same valuation in $\mathring\XX\PP_\alpha$, and in that case we set $\sigma=+$ for one of them and $\sigma=-$ for the other;
 \item $r=\left<\check\theta,\rho_{P(X)}\right>$.
\end{enumerate}
 Notice that different relevant colors may give rise to the same triples. If a color belongs to a normalized spherical root $\gamma$, we will also say that the corresponding virtual weighted color \emph{belongs} to that root.
 \item For every normalized spherical root $\gamma$ of type $T$ split and color $\DD$ (with valuation $\check\theta$) belonging to $\gamma$ there is an element $D'\in \mathcal D_V$, distinct from $\phi(\DD)$, such that $\inv(D')=(-{^{w_\gamma}\check\theta},+,\left<\check\rho,\gamma\right>-\left<{\check\theta},\rho_{P(X)}\right>)$. (Notice that this triple may already have been accounted for by a relevant color; however, this will not always be the case -- see, for instance, the examples of \S \ref{sp4modsp2} and \S \ref{sp4modsp2xsp2}.) In that case, we will say that $D'$ (as well as $\phi(\DD)$, by the previous property) \emph{belongs} to $\gamma$.
 \item For every normalized spherical root $\gamma$ of type $T$ non-split and color $\DD$ (with valuation $\check\theta$) belonging to $\gamma$ there is an element $D'\in \mathcal D_V$ such that $\inv(D')=(-{^{w_\gamma}\check\theta},-,\left<\check\rho,\gamma\right>-\left<{\check\theta},\rho_{P(X)}\right>)$. (Again, this may have been accounted for already, though not always -- cf.\ \S \ref{sp4modsp2xsp2}.) In that case, we will say that $D'$ (as well as $\phi(\DD)$, by the first property) \emph{belongs} to $\gamma$.
\end{enumerate}
From now on we will allow ourselves to loosely identify the set $\mathcal D_V$ with a set of triples $(\check\theta,\sigma,r)$ which may be appearing ``with multiplicity''.

The group $W_X$ acts on triples $(\check\theta,\sigma,r)$ as above by acting on the first component. We let $\Theta$ denote the minimal $W_X$-invariant ``set of such triples'' (again, allowing multiplicities) such that:
\begin{itemize}
 \item $\Theta$ contains all virtual weighted colors.
 \item If $({\check\theta},\sigma,r)\in\Theta$ then $(-{\check\theta},\sigma,r)\in\Theta$ (with the same multiplicity).
\end{itemize}
By abuse of notation, we will sometimes write $\check\theta\in\Theta$ and mean a triple as above, in which case we will denote by $\sigma_{\check\theta}$, $r_{\check\theta}$ the corresponding sign and constant. We denote by $\Theta^+$ the set of $\check\theta\in\Theta$ such that $\left<{\check\theta},\eta\right>\ge 0$ for every $\eta\in\varchi(\XX)$ which appears in $k[\XX]^{(\BB)}$. Equivalently, since the regular $\BB$-eigenfunctions are precisely those rational $\BB$-eigenfunctions which don't blow up on any of the colors, $\Theta^+$ consists of all $\check\theta\in\Theta$ such that ${\check\theta}$ lies in the cone generated by the valuations defined by all colors. (This cone is strictly convex since $\XX$ is assumed quasi-affine.) We also write $\check\theta>0$ if $\check\theta\in\Theta^+$, and ${\check\theta}<0$ if $-{\check\theta}\in\Theta^+$.

Our main assumption regarding the set $\Theta$ is the following:

\begin{statement}\label{posneg}
For every ${(\check\theta,\sigma,r)}\in\Theta$ we either have ${\check\theta}>0$ or ${\check\theta}<0$.  For every $\gamma\in \Delta_X$ the set $\{(\check\theta,\sigma,r)\in\Theta^+ |w_\gamma{\check\theta}<0\}$ consists precisely of the virtual weighted colors belonging to $\gamma$.
\end{statement}

This assumption is certainly not true for every spherical variety. However, I conjecture that it is true if $\XX=\HH\backslash \GG$ is (homogeneous and) affine, i.e.\ $\HH$ is reductive. The benefit of this stament is that it is very easy to check in each particular case; its drawback is that it gives no insight into what it might mean. A geometric understanding of this stament may provide a better understanding of the nature of the $L$-functions which are about to appear. 

\subsection{The formula}\label{ssfinalformula}

The final form of our formula is the following:

\begin{theorem} \label{improvedformula}
Assume that all simple spherical roots of $\XX$ satisfy Statement \ref{STATEMENT} and that the set $\Theta$ satisfies Statement \ref{posneg}. Then:
\begin{equation}\label{omegaoverbeta}
\Omega_{\delta_{(X)}^\frac{1}{2}\tilde\chi} (x_0) = \omega(\tilde\chi) \beta({\tilde\chi})
\end{equation}
where
\begin{equation}
 \beta({\tilde\chi}):= \frac{\prod_{\check\gamma\in\check\Phi_X^+} (1-e^{\check\gamma})}{\prod_{\check\theta\in\Theta^+}(1- \sigma_{\check\theta}q^{-r_{\check\theta}} e^{\check\theta})}({\tilde\chi}),
\end{equation}
and $\omega\in \CC[A_X^*]^{W_X}$. If $\XX$ is affine, then $\omega$ is the constant:
\begin{equation} \label{constantc}
 c := \beta(\delta_{P(X)}^\frac{1}{2})^{-1}.
\end{equation}
In either case: 
\begin{equation}\label{Schur}
 \frac{\Omega_{\delta_{(X)}^\frac{1}{2}\tilde\chi} (x_{\check\lambda})}{\beta({\tilde\chi})} =  \delta_{P(X)}^{-\frac{1}{2}}(x_{\check\lambda}) \prod_{\Theta^+} (1- \sigma_{\check\theta}q^{-r_{\check\theta}}T_{\check\theta}) s_{{\check\lambda}}({\tilde\chi})
\end{equation}
where $s_{{\check\lambda}} = \frac{\sum_{W_X} (-1)^w e^{\check\rho - w\check\rho + w{\check\lambda}}}{\prod_{\check\gamma>0} (1-e^{\check\gamma})}$ is the Schur polynomial \emph{indexed by lowest weight} (that is, if ${\check\lambda}$ is anti-dominant then $s_{\check\lambda}$ is the character of the irreducible representation of $\check G_X$ with lowest weight ${\check\lambda}$) and $T_{\check\theta}$ denotes the formal operator: $T_{\check\theta} s_{{\check\lambda}} = s_{{\check\theta}+{\check\lambda}}$.
\end{theorem}

\begin{remark}
 In the case of a pair $(X,\mathcal L_\Psi)$, where $X$ is parabolically induced from an \emph{affine} spherical variety and $\Psi$ is a generic character, then it should be considered that we are in the same case as that of an ``affine variety''. The statements above pertaining to affine varieties hold, except for the one about the value of the constant $c$, due to failure of Lemma \ref{omegatrivial}.
\end{remark}

The proof of the theorem will be completely combinatorial. The only geometric information used is the following fact (cf.\ \cite[\S 6]{KnLV}):
\begin{quote}
The spherical variety $\XX=\HH\backslash\GG$ is affine (equivalently, $\HH$ is reductive) if and only if $\mathcal V$ and $\{\check v_D | \DD \text{ a color}\}$ are separated by a hyperplane, more precisely: there is an element $\mathfrak c$ in $\mathcal Q^*$ (the $\QQ$-dual of the vector space in which these sets live) with $\left<\check v,\mathfrak c\right> \le 0$ for every $\check v\in \mathcal V$ and $\left<\check v_D,\mathfrak c\right>>0$ for every color $\DD$. It is quasi-affine if and only if $\{\check v_D| \DD \text{ a color}\}$ spans a strictly convex cone and does not contain zero.
\end{quote}

\begin{proof}
 The first step is to show that $B_w({\tilde\chi})= \frac{\beta({\tilde\chi})}{\beta({^w}{\tilde\chi})}$. Since the $B_w$ satisfy cocycle relations, it is enough to show this for $w=w_\gamma$, where $\gamma\in \Delta_X$. We compute:
$$\frac{\beta({\tilde\chi})}{\beta({^{w_\gamma}}{\tilde\chi})} = -e^{\check\gamma}({\tilde\chi}) \cdot \prod_{{\check\theta}>0, w_\gamma{\check\theta}<0} \frac{1-\sigma_{\check\theta}q^{-r_{\check\theta}}e^{-{\check\theta}}}{1-\sigma_{\check\theta}q^{-r_{\check\theta}}e^{{\check\theta}}}.$$
By Statements \ref{posneg} and \ref{STATEMENT}, for $w=w_\gamma$ this is equal to $B_{w_\gamma}(\tilde\chi)$. Hence, $$\frac{\Omega_{\delta_{(X)}^\frac{1}{2}\tilde\chi} (x_{\check\lambda})}{\beta({\tilde\chi})} = \delta_{P(X)}^{-\frac{1}{2}}(x_{\check\lambda}) \sum_{W_X} \frac{1}{\beta({^w}{\tilde\chi})}\,{^w\chi(x_{\check\lambda})}.$$

This proves (\ref{Schur}), which in particular implies that $\frac{\Omega_{\delta_{(X)}^\frac{1}{2}\tilde\chi} (x_{\check\lambda})}{\beta({\tilde\chi})}$ is \emph{regular} in ${\tilde\chi}$, for every ${\check\lambda}$. In particular, for  ${\check\lambda}= 0$ we have $\frac{\Omega_{\delta_{(X)}^\frac{1}{2}\tilde\chi} (x_0)}{\beta({\tilde\chi})} = \omega({\tilde\chi})$ where $\omega$ is a $W_X$-invariant \emph{regular} function of ${\tilde\chi}$. Our proof will be complete if we show that in the affine case $\omega$ is a \emph{constant}; the statement about the precise value of $c$ follows from the fact that $\Omega_{\delta^\frac{1}{2}}=1$ (Lemma \ref{omegatrivial}).

We define two partial orders on the set of weights on $A_X^*$: We will write that ${\check\lambda} \succ_1 \check\mu$ if ${\check\lambda}-\check\mu$ is in the non-negative span $\mathcal R$ of $\check\Phi_X^+$; and that ${\check\lambda}\succ_2\check\mu$ if ${\check\lambda}-\check\mu$ is in the non-negative span $\mathcal T$ of the valuations induced by all colors (i.e.\ the cone dual to the cone of characters of $\BB$ on $k[\XX]^{(\BB)}$). Of course, all characters of $\BB$ on $k[\XX]^{(\BB)}$ are dominant, therefore $\mathcal T\supset\mathcal R$ and the second order is weaker than the first. Now we have:

\begin{lemma}
 If $\omega$ is a non-constant, ($W_X$-)symmetric polynomial on $A_X^*$ and ${\check\mu}$ is a minimal weight appearing with non-zero coefficient in $\omega$, then ${\check\mu}$ is also a minimal weight, appearing with the same coefficient, in $\omega \cdot \prod_{\check\gamma\in \check\Phi_X^+} (1-e^{\check\gamma})$.
\end{lemma}

The validity of this lemma is obvious, since ``minimal weight'' means minimal for the $\succ_1$ ordering and all weights of $\prod_{\check\gamma\in \check\Phi_X^+} (1-e^{\check\gamma})$ are $\succ_1 0$, with strict inequality except for the summand which is equal to the constant $1$.

Given this lemma and (\ref{Schur}), it suffices to prove that $\sum_{W_X} (-1)^w e^{\check\rho - w\check\rho + w{\check\mu}}$ does not contain any nonzero anti-dominant weights when ${\check\mu}=\sum_{{\check\theta}\in I} {\check\theta}$ and $I$ is any subset of $\Theta^+$. Here we use the fact that in the affine case $\mathcal V$ and $\{\check v_D | D \text{ a color}\}$ are separated by a hyperplane, in the sense recalled above.
By Statements \ref{posneg} and \ref{STATEMENT}, this means that:
\begin{quote}
 there is an element $\mathfrak c\in \mathcal Q^*$ with $\left<\mathcal V,\mathfrak c\right>\le 0$ and $\left<\Theta^+,\mathfrak c\right>>0$
\end{quote}
hence $\mathcal T$ is strictly convex and $\mathcal V \cap \mathcal T = \{0\}$.
Since we want to prove that for every $w\in W_X, I\subset\Theta^+$ the weight $\check\rho-w\check\rho + w\sum_{({\check\theta},r)\in I} {\check\theta}$ is not anti-dominant,
it suffices to show that:
\begin{equation}\label{claim}
\check\rho-w\check\rho + w\sum_{({\check\theta},r)\in I} {\check\theta} \succ_2 0.
\end{equation}
For simplicity of notation, in what follows, we denote by $|I|$ the sum $\sum_{({\check\theta},r)\in I} {\check\theta}$, for any $I\subset \Theta$.

Given that $|\Theta^+|-w|\Theta^+|=\check\rho-w\check\rho$, we can write (\ref{claim}) as:
$$ |\Theta^+|- w(|\Theta^+| - |I|) \succ_2 0.$$
Let $\Theta_1 = w\Theta^+\cap \Theta^+$ and $\Theta_2=w\Theta^+\smallsetminus\Theta_1\subset -\Theta^+$. Let $I_1=w(\Theta^+\smallsetminus I)\cap \Theta_1$ and $I_2=w(\Theta^+\smallsetminus I)\cap \Theta_2$. Then (\ref{claim}) can be written:
$$ |\Theta^+|-|I_1| + |-I_2| \succ_2 0 \iff |\Theta^+ \smallsetminus I_1| + |-I_2| \succ_2 0.$$
Since both $\Theta^+\smallsetminus I_1$ and $-I_2$ belong to $\Theta^+$, this holds!
\end{proof}

\begin{remark}
 With minor additions, this proof actually shows that: $$\prod_{\Theta^+} (1- \sigma_{\check\theta}q^{-r_{\check\theta}}T_{\check\theta}) s_{{\check\lambda}}({\tilde\chi}) - \sum_{W_X} e^{w{\check\lambda}} \succ_2 {\check\lambda},$$ where, for a symmetric polynomial $p$, the expression $p\succ_2{\check\lambda}$ means that all weights of $p$ are $\succ_2 {\check\lambda}$; equivalently, since $\succ_2$ is weaker than $\succ_1$, that all anti-dominant weights of $p$ are $\succ_2{\check\lambda}$.
\end{remark}

The theorem leads us to the following definition:

\begin{definition}
 We denote by $L^\frac{1}{2}_X$ the function 
$$\tilde\chi\mapsto c\beta(\tilde\chi)=c\frac{\prod_{\check\gamma\in\check\Phi_X^+}(1-e^{\check\gamma})}{\prod_{\check\theta\in\Theta^+} (1-\sigma_{\check\theta}q^{-r_{\check\theta}}e^{\check\theta})}(\tilde\chi)$$ on $A_X^*$.

We denote by $L_X$ and call ``the $L$-function of $X$'' the function $$L_X^\frac{1}{2}(\tilde\chi)L_X^\frac{1}{2}(\tilde\chi^{-1}) = c^2\frac{\prod_{\check\gamma\in\check\Phi_X}(1-e^{\check\gamma})}{\prod_{\check\theta\in\Theta} (1-\sigma_{\check\theta}q^{-r_{\check\theta}}e^{\check\theta})},$$ which is  $W_X$-invariant on $A_X^*$ and hence can be thought of as a conjugation-invariant function on $\check G_X$.
\end{definition}

\begin{remark}
 The importance of this definition lies in the (mostly conjectural) relationship between period integrals of automorphic forms and the value of $\Omega_{\delta_{(X)}^\frac{1}{2}\tilde\chi}$ in the affine, multiplicity-free case, see section \ref{secEisenstein}. 
\end{remark}

\begin{example}\label{tripleproduct2}
 Let $\XX=\PPGL_2\backslash(\PPGL_2)^3$, the variety of Example \ref{tripleproduct1}. Here the valuations induced by the colors are $\frac{\check\alpha_1+\check\alpha_2-\check\alpha_3}{2}, \frac{\check\alpha_1-\check\alpha_2+\check\alpha_3}{2}, \frac{-\check\alpha_1+\check\alpha_2+\check\alpha_3}{2}$ and the set $\Theta^+$ contains those and also the co-weight $\frac{\check\alpha_1+\check\alpha_2+\check\alpha_3}{2}$ (for simplicity, since $\sigma_{\check\theta}=+$ and $r_{\check\theta}=\frac{1}{2}$ for all $\check\theta$, we only write the value of $\check\theta$). Clearly, $\Theta=\Theta^+ \sqcup (-\Theta^+)$ is $W=W_X$-invariant, and therefore the final formula of Example \ref{tripleproduct1} holds. In this example, $L_X$ is up to zeta-factors equal to the the quotient of the tensor product $L$-function at $\frac{1}{2}$ by the adjoint $L$-function at $0$.
\end{example}

\begin{example}\label{GrossPrasad}
This example is related to the period integral proposed by Gross and Prasad, cf.\ \cite{GP}. Let $\XX=\SSO_n\backslash\SSO_n\times\SSO_{n+1}$ (the stabilizer embedded diagonally in the product of the two groups). Assume all groups are split. One can check that in this case $\check G_X=\check G$, and that the set $\Theta^+$ satisfies Statement \ref{posneg} and is equal to the set of all non-trivial weights $\check \theta$ of the tensor-product representation of $\check G_X$ such that $\left<\check \theta, \rho\right>>0$ (where $\rho$ denotes, as usual, the half-sum of positive roots of $\GG$, and all $\sigma_{\check\theta}=+$ and $r_{\check\theta}=\frac{1}{2}$). Hence, the corresponding $L$-value $L_X$ is, up to zeta factors, equal to: $$ \frac{L(\pi_1\otimes\pi_2,\frac{1}{2})}{L(\pi_1,\operatorname{Ad},0)L(\pi_2,\operatorname{Ad},0)}$$ where we have decomposed an unramified representation $\pi$ of $G$ as $\pi_1\otimes\pi_2$ according to the decomposition $\GG=\SSO_n\times\SSO_{n+1}$. In this case our 
calculation is similar to that performed (in greater generality) by Ichino and Ikeda \cite{II}.

We remark that the valuations associated to colors of type $T$ are usually very easy to compute as follows: One computes the valuation associated to a color in $\mathring\XX\PP_\alpha$ by computing the stabilizer in $\PP_\alpha$ of a point and using the fact that $\left<\check v_D,\alpha\right>=1$. Then one uses the fact that for such a 
divisor $\DD$ we have $\DD\in\mathring\XX\PP_\beta$ if and only if $\left<\check v_D,\beta\right>=1$ and that if $\DD,\DD'$ are the divisors in $\mathring\XX\PP_\beta$ then $\check v_D +\check v_{D'}=\check\beta$. By these rules the calculation of one $\check v_D$ implies many others, in many cases (such as the present one) all of them.
\end{example}

\begin{example}
 Let $\XX=\GGL_n\backslash\SSO_{2n+1}$, where $n$ is an even number. Here we have $\check G_X=\Sp_n\times\Sp_n\subset \Sp_{2n}=\check G$ and the valuations corresponding to colors are all simple short roots of $\check G$ (with $\sigma_{\check\theta}=+, r_{\check\theta}=1$), as well as half the long root of $\check G$ with multiplicity two (and $\sigma_{\check\theta}=+, r_{\check\theta}=\frac{1}{2}$). Therefore, $L_X$ (as a conjugation-invariant function on $\check G_X$) is up to zeta factors equal to:
$$\frac{L(\pi_1,\frac{1}{2})^2 L(\pi_1\otimes\pi_2,1)}{L(\pi_1,\operatorname{Ad},0)L(\pi_2,\operatorname{Ad},0)}.$$
\end{example}


\part{Applications}

\section{The Hecke module structure, multiplicity one and good test vectors.}\label{secHecke} \setcounter{subsubsection}{0}

From now on we let $\XX$ be a homogeneous, \emph{affine} spherical variety which satisfies the assumptions of Theorem \ref{improvedformula}. Assume also that $\mathring X$ has a \emph{unique} $B$-orbit. In particular, the strong form of the Cartan decomposition holds: $X/K \simeq \Lambda_X^+$ and $A_X^*\subset A^*$. The ``spherical Hecke algebra'' $\mathcal H(G,K)$ acts on the module  $C_c^\infty(X)^K$ via the quotient $\mathcal H_X$, defined as the image of the restriction map:
$$\mathcal H(G,K)\xrightarrow{\sim} \CC[A^*]^W \to \CC[\delta_{(X)}^\frac{1}{2} A_X^*],$$
where the first map is the Satake isomorphism. Of course, the image of this map lies in invariants under the normalizer of $\delta_{(X)}^\frac{1}{2} A_X^*$ in the Weyl group, but it may not coincide with it, as the following example shows:

\begin{example}\label{niceexample}
 Consider the spherical variety $\XX=\SSL_2\times\SSL_2\times\SSL_2\backslash \SSp_4\times\SSp_4$, where the first and second copies of $\SSL_2$ are embedded as $\SSp_2$ in the first and second copies of $\SSp_4$, while the third copy of $\SSL_2$ is embedded diagonally in both. If $\alpha_1,\alpha_1'$ denote the short roots in the two copies of $\SSp_4$ and $\alpha_2,\alpha_2'$ denote the long ones, the normalized spherical roots of $\XX$ are $\gamma_1=\alpha_1+\alpha_2,\gamma_2=\alpha_1+\alpha_1',\gamma_3=\alpha_1'+\alpha_2'$. The dual group of this variety is the image of the composite: 
$$\SL_2\times\SL_2^\diag\times \SL_2 \to \SL_2\times\SL_2\times \SL_2\times\SL_2 \to \SO_4\times\SO_4\to \check G=\SO_5\times\SO_5,$$
where $\diag$ denotes the diagonal embedding into the second and third copies, and the second arrow corresponds to the isomorphism: $\SO_4=(\SL_2 \times \SL_2)/\{\pm 1\}$.

It is easy to see that, while $W_X$ is equal to the normalizer of $A_X^*$ in $W$ (here $\delta_{(X)}^\frac{1}{2} $ is trivial), the image of $A_X^*$ in $A^*/W$ is not isomorphic to $A_X^*/W_X$: for example, $W_X$-nonconjugate points of $A_X^*$ belonging to the diagonal copy of $\SO_4$ can be conjugate under $W$. Therefore, the restriction map:
$$\mathcal H(G,K)\xrightarrow{\sim} \CC[A^*]^W \to \CC[\delta_{(X)}^\frac{1}{2} A_X^*]^{\mathcal N_W(\delta_{(X)}^\frac{1}{2} A_X^*)}$$
is not surjective in this case.
\end{example}

 It was proven in \cite{Sa2}[Theorem 6.2.1] that $C_c^\infty(X)^K$ is torsion-free over $\mathcal H_X$. Here we determine precisely the $\mathcal H(G,K)$-structure:

\begin{theorem}\label{modulestructure}
 There is a canonical isomorphism: $C_c^\infty(X)^K \simeq \CC[\delta_{(X)}^\frac{1}{2}A_X^*]^{W_X}$, compatible with the $\mathcal H(G,K)$-structure and the Satake isomorphism, under which the element $1_{\XX(\mathfrak o)}$ is mapped to the constant $1$.
\end{theorem}

\begin{proof}
 Let $P_{\check\lambda}$ be the polynomial on $\delta_{(X)}^\frac{1}{2}A_X^*$ defined as: 
\begin{equation}P_{\check\lambda}(\chi\delta_{(X)}^\frac{1}{2}) = \frac{\Omega_{\chi\delta_{(X)}^\frac{1}{2}}(x_{\check\lambda})}{L_X^\frac{1}{2}(\chi)}.\end{equation}
In particular, $P_0=1$, and $P_{\check\lambda}$ is, up to the constant $c$ of (\ref{constantc}), equal to the right hand side of \eqref{Schur}. Then the map $1_{x_{\check\lambda} K}\mapsto P_{\check\lambda}$ is $\mathcal H(G,K)$-equivariant with respect to the Satake isomorphism $\mathcal H(G,K)= \CC[A^*]^W$. As remarked after the proof of Theorem \ref{improvedformula}, $\check\lambda$ is the lowest weight of $P_{\check\lambda}$ under the order $\succ_2$ defined by the strictly convex cone $\mathcal T$. Therefore, by the usual inductive argument (which here uses the fact that $\mathcal T\cap \mathcal V = \{0\}$ to commence), the polynomials $P_{\check\lambda}$ generate $\CC[\delta_{(X)}^\frac{1}{2}A_X^*]^{W_X}$, and the result follows.
\end{proof}

\begin{corollary}\label{geomEnd}
 Let $\left(\End_{\mathcal H(G,K)} C_c^\infty(X)^K\right)^\geom$ (where ``$\geom$'' stands for ``geometric'') denote the subring of those $\mathcal H(G,K)$-module endomorphisms of $C_c^\infty(X)^K$ which act by scalars on the morphisms $S_\chi$, i.e.\ if $D$ is such an endomorphism:
$$S_{\chi^{-1}\nu^{-1}} \circ D = c_D(\chi) S_\chi$$ 
for some scalar $c_D(\chi)$. (Recall that $S_{\chi^{-1}\nu^{-1}}$ is adjoint to $\Delta_\chi$.) Under the assumptions of the present section, the map $D\mapsto c_D(\chi)$ gives rise to a canonical isomorphism:
$$\left(\End_{\mathcal H(G,K)} C_c^\infty(X)^K\right)^\geom \simeq \CC[\delta_{(X)}^\frac{1}{2}A_X^*]^{W_X}.$$
\end{corollary}

This was conjectured in \S 6.3 of \cite{Sa2}. The proof if straightforward: clearly, ``multiplication'' by any element of $\CC[\delta_{(X)}^\frac{1}{2}A_X^*]^{W_X}$ (under the isomorphism of Theorem \ref{modulestructure}) is a geometric endomorphism. The converse direction was already proven before the statement of the conjecture in \cite{Sa2}.

\begin{remark}
 Example \ref{niceexample} shows that a correction is due to \cite[Theorem 6.3.2]{Sa2}, namely, the theorem there is proven only in the case where the restriction map surjects onto $\CC[\delta_{(X)}^\frac{1}{2} A_X^*]^{\mathcal N_W(\CC[\delta_{(X)}^\frac{1}{2} A_X^*])}$ (and cases parabolically induced from those), with the rest of the cases remaining conjectural; however, all cases of that theorem falling under our current assumptions are covered by Corollary \ref{geomEnd}.
\end{remark}

We can draw more interesting corollaries from Theorem \ref{modulestructure} if the restriction map: \begin{equation} \label{Heckerestriction}
                                        \CC[A^*]^W\to \CC[\delta_{(X)}^\frac{1}{2} A_X^*]^{W_X}
                                       \end{equation}
 is surjective:

\begin{corollary}
Assume that the restriction map (\ref{Heckerestriction}) is surjective or, equivalently, that $\mathcal H_X\simeq \CC[\delta_{(X)}^\frac{1}{2} A_X^*]^{W_X}$ under restriction of the Satake isomorphism. Then:
\begin{enumerate}
 \item $C_c^\infty(X)^K$ is a principal module under $\mathcal H(G,K)$, generated by the characteristic function of $\XX(\mathfrak o)$.
 \item (Multiplicity one:) For every character of $\mathcal H(G,K)$ belonging to $\spec_M \mathcal H_X$ the corresponding eigenspace in the Hecke module $C^\infty(X)^K$ is one-dimensional. For every irreducible unramified representation $\pi$ we have $$\dim\Hom_G(\pi,C^\infty(X))\le 1.$$ (Of course, a necessary condition for the dimension to be non-zero is that the Hecke eigencharacter of $\pi^K$ belongs to $\spec_M \mathcal H_X$.)
 \item (Good test vectors:) For every irreducible unramified representation $\pi$ and any non-zero $H$-invariant functional $L$ on $\pi$ we have:
\begin{equation}
 L|_{\pi^K}\ne 0.
\end{equation}
\end{enumerate}
\end{corollary}

The last statement also has the following implication when $\GG,\HH$ are defined over a global field $F$ and locally at almost every place satisfy the assumptions of the Corollary: If $\Hom_{H_v}(\pi_v,\CC)\ne 0$ for every place $v$ of $F$, then $\Hom_{\HH(\Ad)} (\pi,\CC)\ne 0$, where $\Ad$ denotes the ring of adeles of $F$.

The proof of the statements in the Corollary is straightforward and will be omitted. We notice that a variety which satisfies the assumptions of the corollary is multiplicity-free not only for unramified representations in general position (in the support of $\mathcal H_X$), but for \emph{every} irreducible unramified representation $\pi$, i.e.:
$$\dim \Hom_H(\pi,\CC)\le 1$$ 
Example \ref{niceexample} shows that without the assumption that (\ref{Heckerestriction}) is surjective one can have generic multiplicity one without having multiplicity one for every unramified representation. Indeed, as we remarked there are distinct points of $\delta_{(X)}^\frac{1}{2} A_X^*/W_X$ which map to the same point of $A^*/W$; at those points the multiplicity of the fiber of $C_c^\infty(X)^K$ over $\mathcal H_X$ is greater than one, even though the generic multiplicity is one.

\section{Unramified Plancherel formula}\setcounter{subsubsection}{0} \label{secPlancherel} 

We continue to make the same assumptions on $\XX$ as in the previous section. Since $\HH$ is reductive, we may and will assume that the $\GG$-eigenform $\omega_X$ on $\XX$ is $\GG$-invariant, and hence defines a $G$-invariant measure $|\omega_X|$ on $X$. Remember that we have normalized that measure so that $|\omega_X|(x_0 J) =1 $. We will compute the Plancherel formula for unramified functions on $X$. More precisely, it is known that there exists a(n essentially unique) decomposition of $L^2(X)$ as a direct integral of irreducible, unitary representations. Specializing to $K$-invariant elements of $L^2(X)$, we get for every $\Phi\in C_c^\infty(X)^K$ a formula for $\Vert \Phi \Vert^2_{L^2(X)}$ as an integral of $|\left<\Phi,\Omega\right>|^2$ where $\Omega$ ranges over $\mathcal H(G,K)$-eigenfunctions belonging to unitary representations. We will compute this formula; as an application, we will compute the volume of $\XX(\mathfrak o)$ (essentially, its ``Tamagawa volume'').

Since we are keeping the assumptions of the previous section, every Hecke eigenfunction is a multiple of 
$$ \Omega_\chi' := \frac{\Omega_\chi}{L_X^\frac{1}{2}(\chi)}$$
for some $\chi\in \delta_{(X)}^\frac{1}{2}A_X^*$. Notice that $\Omega_\chi'(x_0)=1$ and $\Omega_{^w \chi}' = \Omega_{\chi}'$ for every $w\in W_X$. Therefore, the Plancherel formula will have the form: 
\begin{equation}\label{abstractPlancherel}
 \Vert \Phi\Vert^2 = \int_{\delta_{(X)}^\frac{1}{2}A_X^*/W_X} |\left<\Phi,\Omega_{\chi}'\right>|^2 d\mu(\chi)
\end{equation} for every $\Phi\in C_c^\infty(X)^K$ and a \emph{unique} positive measure on $\delta_{(X)}^\frac{1}{2}A_X^*/W_X$ (supported, of course, on the set of points belonging to unitary representations) which, for brevity, when normalized this way, will be called \emph{the} Plancherel measure for $X$.

\begin{theorem}\label{Planchereltheorem}
The Plancherel measure for the unramified spectrum of $X$ is supported on $\delta_{(X)}^\frac{1}{2}A_X^{*,1}/W_X$, where $A_X^{*,1}$ denotes the maximal compact subgroup of $A_X^*$. For every $\Phi\in C_c^\infty(X)^K$ we have:
\begin{equation} \label{Plancherel}
 \Vert \Phi\Vert^2 =  \frac{1}{Q_{P(X)}\cdot|W_X|} \int_{A_X^{*,1}} \left|\left<\Phi, \Omega'_{\delta_{(X)}^\frac{1}{2}\chi}\right>\right|^2 L_X(\chi) d\chi,
\end{equation}
where $d\chi$ is probability Haar measure on $A_X^{*,1}$ and 
$$Q_{P(X)}=\frac{\Vol(K)}{\Vol(\PP(\XX)^-(\mathfrak o)\PP(\XX)(\mathfrak o))}=\prod_{\check\alpha\in\check\Phi_{P(X)}^+} \frac{1-q^{-1-\left<\check\alpha,\rho\right>}}{1-q^{-\left<\check\alpha,\rho\right>}},$$
where $\PP(\XX)^-$ denotes a parabolic opposite to $\PP(\XX)$ defined over $\mathfrak o$ and $\check\Phi_{P(X)}^+$ is the set of coroots corresponding to roots in the unipotent radical of $\PP(\XX)$.
\end{theorem}

\begin{remark}
 With respect to our original eigenfunctions $\Omega_\chi$ we have:
\begin{equation}\label{Plancherel2} \Vert \Phi\Vert^2 = \frac{1}{Q_{P(X)}\cdot|W_X|} \int_{A_X^{*,1}} \left|\left<\Phi, \Omega_{\delta_{(X)}^\frac{1}{2}\chi}\right>\right|^2  d\chi.\end{equation}
\end{remark}

\begin{proof}
As in the previous section, let $P_{\check\lambda}$ be the polynomial on $\delta_{(X)}^\frac{1}{2}A_X^*$ defined as $P_{\check\lambda}(\chi) = \Omega_\chi'(x_{\check\lambda})$. The fact that $\left<1_{x_0K},1_{x_{\check\lambda} K}\right>_{L^2(X)}=0$ for ${\check\lambda}\ne 0$ implies, via the abstract Plancherel formula (\ref{abstractPlancherel}), that $$\int_{\delta_{(X)}^\frac{1}{2}A_X^*/W_X} P_{\check\lambda}(\chi) d\mu(\chi)=0$$
for every ${\check\lambda}\ne 0$.

Recall that the $P_{\check\lambda}$ span the space $\mathcal P$ of polynomials on $\delta_{(X)}^\frac{1}{2}A_X^*/W_X$. Given a linear functional on this space of polynomials, there exists at most one real-valued measure of bounded support on $\delta_{(X)}^\frac{1}{2}A_X^*/W_X$ which represents this functional. (Indeed, this follows from the density of the functions of the form $\Re P, \Im P$, $P\in\mathcal P$, in the space of continuous functions on any compact domain.) We will show in a combinatorial way that $[1_{x_{\check\lambda} K}, 1_{x_0 K}] = 0$ for ${\check\lambda}\ne 0$, where $[,]$ denotes the hermitian inner product defined by the right hand side of (\ref{Plancherel}) (or (\ref{Plancherel2})). It will then follow that Plancherel measure is a multiple of the measure of (\ref{Plancherel}). 

\begin{lemma}
 For every ${\check\lambda}\ne 0$ we have: $[1_{x_{\check\lambda} K}, 1_{x_0 K}] = 0$, where $[,]$ denotes the hermitian inner product defined by the right hand side of (\ref{Plancherel}).
\end{lemma}

It suffices to show that $\int_{A_X^{*,1}} P_{\check\lambda}(\delta_{(X)}^\frac{1}{2}\chi) L_X(\chi) d\chi = 0$ for ${\check\lambda}\ne 0$. The value of this integral is equal to the constant term of the Laurent series expansion of the rational function $P_{\check\lambda}(\delta_{(X)}^\frac{1}{2}\chi) L_X(\chi)$ (see, for instance, \cite{MD}). We have: 
$$P_{\check\lambda}(\delta_{(X)}^\frac{1}{2}\chi) L_X(\chi) = c\frac{\prod_{\check\gamma\in\check\Phi_X}(1-e^{\check\gamma})}{\prod_{{\check\theta}\in\Theta} (1-\sigma_{\check\theta}q^{-r_{\check\theta}}e^{\check\theta})} \sum_{W_X} \frac{\prod_{{\check\theta}\in\Theta^+} (1-\sigma_{\check\theta}q^{-r_{\check\theta}}e^{\check\theta})}{\prod_{\check\gamma\in\check\Phi_X^+}(1-e^{\check\gamma})}e^{\check\lambda}\left({^w\chi}\right)=$$
$$= c \sum_{W_X} \frac{\prod_{\check\gamma\in\check\Phi_X^-}(1-e^{\check\gamma})}{\prod_{{\check\theta}\in\Theta^-} (1-\sigma_{\check\theta}q^{-r_{\check\theta}}e^{\check\theta})}e^{\check\lambda} \left({^w\chi}\right).$$

In the notation of the proof of Theorem \ref{improvedformula}, all weights in the Laurent expansion of $\frac{\prod_{\check\gamma\in\check\Phi_X^-}(1-e^{\check\gamma})}{\prod_{{\check\theta}\in\Theta^-} (1-\sigma_{\check\theta}q^{-r_{\check\theta}}e^{\check\theta})}e^{\check\lambda} $ are $\prec_2 {\check\lambda}$. Therefore, $0$ is a weight of the above expression only if $0 \in w({\check\lambda}-\mathcal T)$ for some $w\in W_X$, equivalently only if ${\check\lambda}\in\mathcal T \iff {\check\lambda} = 0$. This proves the lemma.

Hence the Plancherel formula has to be a multiple of the right-hand-side of (\ref{Plancherel}). On the other hand, a criterion of Bernstein, whose full proof appears in \cite[Theorem 11.3.1]{SV}, implies that the right-hand-side of (\ref{Plancherel}) is precisely the \emph{most continuous part} of the Plancherel formula. Having proven that the Plancherel formula for $L^2(X)^K$ is a multiple of (\ref{Plancherel}), Bernstein's criterion proves that the proportionality constant is $1$, so the theorem follows.

We explicate this criterion with many simplicifactions that apply to our case: Notice that, because of the form of the $\mathcal H(G,K)$-eigenfunctions, the Plancherel formula for $L^2(X)^K$ has the form:
$$\Vert \Phi\Vert^2 = \int \sum_{w,w'} c_{w,w'}(\chi)\hat\Phi({^w\chi}) \overline{\hat \Phi({^{w'}\chi})} \mu(\chi).$$
Here the integral is over representatives in $\delta_{(X)}^\frac{1}{2} A_X^*$ of points in the image of $\delta_{(X)}^\frac{1}{2}A_X^*$ in $A^*/W$, and the sum is over the elements of $W$ which map a given representative $\chi$ into $\delta_{(X)}^\frac{1}{2} A_X^*$. We denote by $\hat \Phi$ the Mellin transform of $\Phi$ considered as a function on $A_X^*$. The factors $c_{w,w'}(\chi)$ are constants which can be read off from our formula, and the measure $\mu(\chi)$ is the (unknown for now) Plancherel measure.

Bernstein's criterion \cite[Theorem 11.3.1]{SV}, in a very explicit form, says that if we ignore the ``cross terms'' of the above formula, i.e.\ the terms with ${^w\chi}^{-1}\ne \overline{{^{w'}\chi}}$ (hence keeping only those with unitary $\chi$ and with $w'=w$), then we should get the Plancherel formula for ``the'' horospherical variety $\XX_\emptyset:= \HH_\emptyset\backslash \GG$. Here we identify $\Phi$ with a function on $X_\emptyset$ via ``the'' Iwasawa decomposition $X_\emptyset/K \simeq \Lambda_X$ and the embedding $\Lambda_X^+\hookrightarrow \Lambda_X$. The definite article ``the'' has been placed in quotation marks here, because the variety $X_\emptyset$ has nontrivial $G$-automorphisms which, in fact, act transitively on $X_\emptyset /K$. The point is, however, that given any choice of such an injection of $K$-orbits $X/K \simeq \Lambda_X^+\hookrightarrow X_\emptyset/K$ there is a (unique, obviously) $G$-eigenmeasure on $X_\emptyset$ such that for elements of $\Lambda_X^+$ far enough from the 
walls of the cone spanned by 
$\Lambda_X^+$ this injection is measure-preserving. Finally, we notice that this measure-preserving identification of orbits close enough to infinity is not particular to $K$-orbits but applies to any open compact subgroup and is compatible with inclusions of any two open compact subgroups as long as we restrict it close enough to infinity (far enough from the walls of $\Lambda_X^+$, in our case). In particular, by comparing the computation of volumes of Iwahori-orbits in Lemma \ref{Iwahorivolume} with the analogous computation for $X_\emptyset$ it is easy to see what measure to put on $X_\emptyset/K$ and to compare Plancherel formulas. 

We remark that there is a more conceptual way to understand the variety $\XX_\emptyset$ and these orbit identifications: the variety $\XX_\emptyset$ is the \emph{open $\GG$-orbit} on the normal bundle to the closed $\GG$-orbit -- call it $\ZZ$ -- in the ``wonderful'' compactification of $\XX$ (or some suitable toroidal compactification); any $p$-adic analytic map from a 
neighborhood of $Z$ in $X_\emptyset$ to $X$ which induces the identity on normal bundles gives rise to such an identification of $K$-orbits close enough to $Z$.

\end{proof}

The Plancherel formula has the following corollary. This corollary does not apply, of course, to the case of non-trivial line bundles $\mathcal L_\Psi$.

\begin{theorem}\label{measuretheorem}
 The measure of $\XX(\mathfrak o)$ is:
\begin{equation}\label{measure}
 |\omega_X|(\XX(\mathfrak o))= Q_{P(X)}\cdot c^{-1} = Q_{P(X)} \frac{\prod_{\check\gamma\in\check\Phi_X^+}(1-e^{\check\gamma})}{\prod_{\check\theta\in\Theta^+} (1-\sigma_{\check\theta}q^{-r_{\check\theta}}e^{\check\theta})}(\delta_{P(X)}^\frac{1}{2}).
\end{equation}
\end{theorem}

\begin{proof}
 We have: $$|\omega_X|(\XX(\mathfrak o))=\Vert 1_{\XX(\mathfrak o)}\Vert^2 = \frac{1}{Q_{P(X)}\cdot|W_X|} \int_{A_X^{*,1}} \left|\left<1_{\XX(\mathfrak o)}, \Omega'_{\delta_{(X)}^\frac{1}{2}\chi}\right>\right|^2 L_X(\chi) d\chi = $$
$$= \frac{1}{Q_{P(X)}\cdot|W_X|} \int_{A_X^{*,1}} |P_0(\delta_{(X)}^\frac{1}{2}\chi)|^2 \left(|\omega_X|(\XX(\mathfrak o))\right)^2 L_X(\chi) d\chi \Rightarrow$$
$$|\omega_X|(\XX(\mathfrak o)) = Q_{P(X)}\cdot|W_X| \left(\int_{A_X^{*,1}} L_X(\chi) d\chi\right)^{-1}.$$
As before, the integral of $L_X(\chi)=P_0(\delta_{(X)}^\frac{1}{2}\chi) L_X(\chi)$ is equal to the constant term in the Laurent expansion of:
$$c \sum_{W_X} \frac{\prod_{\check\gamma\in\check\Phi_X^-}(1-e^{\check\gamma})}{\prod_{{\check\theta}\in\Theta^-} (1-\sigma_{\check\theta}q^{-r_{\check\theta}}e^{\check\theta})}$$
which in this case is equal to $c\cdot |W_X|$. Therefore:

$|\omega_X|(\XX(\mathfrak o)) = Q_{P(X)}\cdot c^{-1} = Q_{P(X)}\cdot \frac{\prod_{\check\gamma\in\check\Phi_X^+}(1-e^{\check\gamma})}{\prod_{{\check\theta}\in\Theta^+} (1-\sigma_{\check\theta}q^{-r_{\check\theta}}e^{\check\theta})}(\delta^\frac{1}{2})$.	
\end{proof}

\begin{remark}
 It is natural to call the measure $(1-q^{-1})^{{\rm{rk}}(A_X)}\cdot |\omega_X|$ the \emph{Tamagawa measure} on $X$. Indeed, any globally defined invariant volume form on $\XX$ will induce, for almost every place, the same measure on $\mathring X$ as $\mathfrak q^*(a^{-1} da)\wedge du$ (where $\mathfrak q:\mathring \XX \to \AA_X$ is the natural projection based at the point $x_0$). With respect to the latter, we have $\Vol(x_0J=x_0B_0) = a^{-1} da(A_{X0}) = (1-q^{-1})^{{\rm{rk}}(A_X)}$, therefore the ``Tamagawa measure'' $d^{\rm{Tam}}x$ is equal to $(1-q^{-1})^{{\rm{rk}}(A_X)} |\omega_X|$.

Therefore we have shown: \begin{equation}\label{Tamagawa}
d^{\rm{Tam}}x(\XX(\mathfrak o))=Q_{P(X)}\cdot (1-q^{-1})^{{\rm{rk}}(A_X)}\cdot \frac{\prod_{\check\gamma\in\check\Phi_X^+}(1-e^{\check\gamma})}{\prod_{{\check\theta}\in\Theta^+} (1-\sigma_{\check\theta}q^{-r_{\check\theta}}e^{\check\theta})}(\delta^\frac{1}{2}).\end{equation}
\end{remark}

\section{Periods of Eisenstein series.}\label{secEisenstein}\setcounter{subsubsection}{0}

Let now $F$ be a number field and $\GG$ a split connected reductive group defined over the ring of integers of $F$, and let $\XX$ be a spherical $\GG$-variety over $F$. We use the same notation: $\BB$, $\AA$, etc.\ as above and assume that $\BB(F)$ has a single orbit on $\mathring \XX(F)$. We also assume that for almost all completions $F_v$ of $F$ the variety $\XX_{F_v}$ satisfies the assumptions of the previous two sections. We will denote by $\AA(\Ad)^1$ (and similarly for other groups than $\AA$) the intersection of the kernels of all homomorphisms: $\AA(\Ad)\to \Gm(\Ad) \to \RR^\times_+$ where the first arrow denotes an algebraic character and the second denotes the absolute value. Hence, $\AA(\Ad)/\AA(\Ad)^1 \simeq (\RR_+^\times)^{\operatorname{rk}(\AA)}$. For every place $v$ we denote the group $\GG(F_v)$ by $G_v$, the group $\GG(\mathfrak o_v)$ (if $v<\infty$) by $K_v$, etc.

Any idele class character of $\AA(\Ad)$ can be twisted by characters of the group $\AA(\Ad)/\AA(\Ad)^1$, and thus lives in a $\operatorname{rk}(\AA)$-dimensional complex manifold of characters. Let $\omega$ denote such a family. For $\chi\in\omega$ we denote by $I(\chi)=\Ind_{\BB(\Ad)}^{\GG(\Ad)}(\chi\delta^\frac{1}{2})$ the normalized principal series of $\GG(\Ad)$, considered as a holomorphic family of vector spaces; that is, we fix a notion of ``holomorphic sections'' by identifying the underlying vector spaces of the representations in the usual way: by restriction to a compact subgroup $K$ satisfying: $G(\Ad) = B(\Ad) K$. Our convention for the archimedean places will be that $\Ind$ denotes the $K_\infty$-finite vectors in that representation, where $K_\infty$ is a maximal compact subgroup of $G_\infty$. For a meromorphic family of sections $\chi\mapsto f_\chi\in I(\chi)$ we have the principal Eisenstein series defined by the convergent sum:
\begin{equation}\label{Es}
 E(f_\chi, g) = \sum_{\gamma \in {\BB}(F)\backslash \GG(F)} f_\chi(\gamma g)
\end{equation}
if $\left<\check\alpha,\Re(\chi)\right>\gg 0$ for all $\alpha\in\Delta$, and by meromorphic continuation to the whole $\omega$.

Let $\HH$ be a spherical subgroup of $\GG$ over $F$. 
We would like to compute the period integral $\int_{\HH(F)\backslash \HH(\Ad)} E(f_\chi,h) dh$. Of course, this integral may not be convergent, therefore we have to understand it as a distribution on the Eisenstein spectrum of $\GG$. 
Our goal is to compute the most continuous part of this distribution. We notice that regularized periods of Eisenstein series abound in the literature (a typical example is \cite{JLR}), and different methods are suitable for different purposes. Our present method is much softer than most, and if nothing else it helps motivate the general philosophy linking periods of automorphic forms with local Plancherel formulas.

For simplicity of notation, we will discuss only the case where $\omega$ consists of the characters of $\AA(\Ad)/\AA(\Ad)^1$; however, the argument and the result hold for any family of idele class characters, with the only modification that the summation in (\ref{periodformula}) is only over orbits of maximal rank for which $\omega$ is trivial on $\BB_\xi(\Ad)^1$ (in the notation that follows).

In what follows, we use Tamagawa measures for all groups. It is convenient to do formal computations with measures given by differential forms, without convergence factors, and interpret the non-convergent formal products in the end as special values of zeta functions or residues thereof (if $\zeta(1)$ appears, where $\zeta$ denotes the Dedekind zeta function of $F$). 
Let $\Phi\in \cind_{\AA(\Ad)^1 \UU(\Ad)}^{\GG(\Ad)}(1)$ (where $\cind$ denotes compact induction), and write $\Phi$ in terms of its Mellin transform with respect to the left-$\AA(\Ad)$-action: 
$$ \Phi(g)= \int_{\widehat{\AA(\Ad)/\AA(\Ad)^1}} f_{\chi\delta^{-\frac{1}{2}}} (g) d\chi$$
where $f_\chi\in I(\chi)$ and $d\chi$ is Haar measure on the unitary dual of $\AA(\Ad)/\AA(\Ad)^1$. Note that the unitary dual of $\AA(\Ad)/\AA(\Ad)^1$ can naturally be identified with the imaginary points $i\mathfrak a_\RR^*$  of the Lie algebra of the dual torus, via an isomorphism which we will denote by $\exp$, i.e.\ $\widehat{\AA(\Ad)/\AA(\Ad)^1}=\exp(i\mathfrak a_\RR^*)$. Notice also the shift $\delta^{-\frac{1}{2}}$ in the above formula because $f_\chi$ has been defined with respect to the \emph{normalized} action of $\AA(\Ad)$. However, by our assumption that $\Phi$ is compactly supported modulo $\AA(\Ad)^1 \UU(\Ad)$, its Mellin transform is entire in $\chi$ and rapidly decaying on vertical strips, and hence we can shift the contour of integration and write:
$$ \Phi(g)= \int_{\exp(\kappa+ i\mathfrak a_\RR^*)} f_{\chi} (g) d\chi$$
for any $\kappa\in \mathfrak a_\CC^*$. In particular, we can shift to the domain of convergence of the Eisenstein sum (\ref{Es}) and then we will have: 
$$ \sum_{\gamma \in {\BB}(F)\backslash \GG(F)} \Phi(\gamma g) = \int_{\exp(\kappa+i\mathfrak a_\RR^*)} E(f_\chi,g) d\chi,$$
a function of rapid decay on the automorphic quotient $\GG(F)\backslash\GG(\Ad)$ (cf.\ \cite[II.1.11]{MW}).

We integrate over $\HH(F)\backslash \HH(\Ad)$:
$$\int_{\HH(F)\backslash \HH(\Ad)} \sum_{\gamma \in {\BB}(F)\backslash \GG(F)} \Phi(\gamma h) dh = $$
\begin{eqnarray}\label{periodcomp}
=\sum_{\xi\in [{\BB}(F)\backslash \GG(F)/\HH(F)]} \int_{\HH_\xi(\Ad)\backslash \HH(\Ad)} \int_{\HH_\xi(F)\backslash \HH_\xi(\Ad)} \int_{\exp(\kappa+i\mathfrak a_\RR^*)} f_\chi(\xi a h) d\chi da dh. \nonumber \\
\end{eqnarray}
Here $[{\BB}(F)\backslash \GG(F)/\HH(F)]$ denotes a (finite) set of representatives in $\GG(F)$ for the $(\BB(F),\HH(F))$-double cosets, and $\HH_\xi:= \HH\cap \xi^{-1}\BB\xi$. Similarly we will denote: $\BB_\xi:=\BB\cap \xi\HH\xi^{-1}$, and we will let $\YY$ denote the $\BB$-orbit of $\xi\HH$ on $\GG/\HH$. Here the measure $da$ is a (Tamagawa) $\HH_\xi(\Ad)$-eigenmeasure, with eigencharacter the inverse of the modular character of $\HH_\xi$, so that the integral over $\HH$ admits a factorization as above. We will denote this character by $\eta_\xi$, and its $\xi$-conjugate -- which is a character of $\BB_\xi(\Ad)$ -- by $\eta_Y$. Hence the two inner integrals are valued in the line bundle over $\HH_\xi(\Ad)\backslash \HH(\Ad)$ defined by $\eta_\xi^{-1}$, and $dh$ is an $\HH(\Ad)$-invariant measure valued in the dual of that line bundle. 

For a fixed $h\in\HH(\Ad)$ the two inner integrals:
$$\int_{\HH_\xi(F)\backslash \HH_\xi(\Ad)}\int_{\exp(\kappa+i\mathfrak a_\RR^*)} f_\chi(\xi a h) d\chi da=\int_{\BB_\xi(F)\backslash \BB_\xi(\Ad)}\int_{\exp(\kappa+i\mathfrak a_\RR^*)} f_\chi(a \xi h) d\chi da$$
are equal to: 
$$\Vol(\BB_\xi(F)\backslash \BB_\xi(\Ad)^1) \cdot \int_{\BB_\xi(\Ad)^1\backslash \BB_\xi(\Ad)} \int_{\exp(\kappa+i\mathfrak a_\RR^*)} f_\chi(a\xi  h) d\chi da. $$
The Tamagawa volume that appears here\footnote{ERRATUM: This is not quite true in the case $\AA_X\ne \AA$ -- i.e.\ if $\BB_\xi$ contains a multiplicative part -- since the standard convention for Tamagawa measures involves the residue of the Dedekind zeta function in this case. In any case, the treatment of this case is missing a thorough discussion of measures: apart from that residue, one needs to normalize the measure on $\mathfrak a_{X,\RR}^*$ appropriately for an application of Fourier inversion.} is equal to $1$, since $\BB_\xi$ is a connected, split, solvable group. By abelian Fourier analysis the last expression is equal to:
$$\int_{\delta^{-\frac{1}{2}} \eta_Y^{-1} \exp(i\mathfrak a_{Y,\RR}^*)} f_\chi(\xi h) d\chi,$$
where we have taken into account that $\exp(\mathfrak a_Y^*)$, where $\mathfrak a_Y^*$ is the Lie algebra of $A_Y^*$, considered as a subgroup of the group of characters $\exp(\mathfrak a^*)$ of $\AA(\Ad)/\AA(\Ad)^1$, is the orthogonal complement of the image of $\BB_\xi(\Ad)$.

To determine the most continuous part of the $\HH$-period integral, hence, it is enough to consider those $\xi$ which correspond to orbits $\YY$ of maximal rank. Again, we can move the contour of integration, this time to $\exp(\kappa_Y+i\mathfrak a_{Y,\RR}^*)$, where $\kappa_Y$ is deep in the region where the morphisms $\Delta^Y_{\chi,v}$ introduced in the present paper (but now with an index $v$ to indicate the place of $F$) are convergent. Though we have not discussed archimedean places here, all definitions and properties of meromorphic -- not rational now -- continuation can easily be established in a similar manner. Notice also that the convergence, for $\Re(\chi)$ in a certain cone, of the product over all places $v$ of the operators $\Delta_{\chi,v}^Y$, considered as an $\HH(\Ad)$-invariant functional on $C_c^\infty(\UU\backslash\GG (\Ad))$, is established by the same argument as in the local case.

Returning to (\ref{periodcomp}), we can interchange the order of integration to express the 
contribution of the orbit $\YY$ as: 
$$\int_{\exp(\kappa_Y + i\mathfrak a_{Y,\RR}^*)} \int_{\HH_\xi(\Ad)\backslash \HH(\Ad)} f_\chi(\xi h) dh d\chi$$ 
and the new inner integral is equal to $\prod_v \Delta^{Y,{\Tam}}_{\chi,v}$ (where the exponent $^{\Tam}$ stands to show that we are using Tamagawa measures, rather than the measures used throughout the paper). Interchanging the order of integration is justified as follows: The function
$$\Phi'(a):= \exp(\kappa_Y)(a) \int_{\HH_\xi(\Ad)^1\backslash \HH(\Ad)} \Phi(a\xi h)dh,$$
on $\AA_Y(\Ad)$ is a Schwartz--Harish-Chandra function if $\kappa_Y$ is sufficiently deep in the domain of convergence of the morphisms $\Delta_{\chi,v}^Y$. Therefore, its value at $1$ (which is represented by the iterated integral of (\ref{periodcomp}) before changing the order of integration) is equal to the integral of its Mellin transforms. But those are given by the double integral of  $\Phi \in C_c^\infty(\UU(\Ad)\AA(\Ad)^1\backslash \GG(\Ad))$ over the corresponding orbit of the group $\AA(\Ad)/\AA(\Ad)^1\times\HH(\Ad)$ and against a character of that subgroup; for characters of the form $\exp(\kappa_Y + i\mathfrak a_{Y,\RR}^*)$ similar considerations as in the local case imply that this double integral is absolutely convergent, and therefore equal to:
$$ \int_{\HH_\xi(\Ad)\backslash\HH(\Ad)} f_\chi(\xi h) dh.$$

Fix a finite set of places $S$, including the infinite ones and those finite places where our assumptions on the spherical variety $\XX=\HH\backslash \GG$ do not hold, such that we have a factorization: $f_\chi = \prod_v f_{\chi,v}$ with 
$f_{\chi,v}$ being the ``standard'' $K_v$-invariant function: $f_{\chi,v}^0(bk)=\chi\delta^\frac{1}{2}(b)$ (denoted by $\phi_{K,\chi}$ in \S \ref{ssnormalization}) for every $v\notin S$. For each orbit $\YY$ of maximal rank, choose an element $w\in W$ with $\YY=\,{^w\mathring \XX}$ and the length of $w$ equal to the codimension of $\YY$ (in the notation of \S \ref{eigenforms}: $w\in W(\YY)^{-1}$). We can then write $f_{\chi,v}^0$ as $j_{w,v}^{-1}(^{w^{-1}}\chi) T_w f_{^{w^{-1}}\chi,v}^0$, where: $$j_{w,v}(\chi):=\prod_{\check\alpha>0,w\check\alpha<0}\frac{1-q^{-1}e^{\check\alpha}(\chi_v)}{1-e^{\check\alpha}(\chi_v)}$$ at the finite place with residual degree $q$ (we are using the standard intertwining operators $T_w$ here, cf.\ (\ref{Twphi})). We will also set: $$\tilde j_{w,v}(\chi):=\prod_{\check\alpha>0,w\check\alpha>0}\frac{1-q^{-1}e^{\check\alpha}(\chi_v)}{1-e^{\check\alpha}(\chi_v)}.$$ Using the fact that $\Delta_{\chi,v}^{Y,{\Tam}}\circ T_{w}= \Delta^{{\Tam}}_{^{w^{-1}}\chi,v}$ 
(Lemma \ref{composeStypeU}), we can express the contribution of 
the orbit $\YY$ to (\ref{periodcomp}) as:
$$ \int_{\exp(\kappa_Y+i\mathfrak a_{Y,\RR}^*)} j_{w}^S(^{w^{-1}}\chi)^{-1} \prod_{v\notin S} \Delta^{\Tam}_{^{w^{-1}}\chi,v} (f_{^{w^{-1}}\chi,v}^0) \prod_{v\in S} \Delta_{\chi,v}^{Y,{\Tam}} (f_{\chi,v}) d\chi.$$ 

Notice that the domain of convergence of $\Delta_{\chi,v}$ contains a translate of the cone of regular eigenfunctions, which is the cone $\mathcal T^\vee\subset\varchi(\XX)\otimes\mathbb R$ dual to the cone $\mathcal T$ spanned by the valuations $\check v_D$ of all colors. Since the cone of colors contains the images of the positive co-roots of $\GG$, we have simultaneous convergence for $\Delta_{\chi,v}$ and $T_w$ (acting on $I(\chi_v)$). It is easy to argue, based on the fact that the map $\BB w^{-1} \BB \times^{\BB} \YY \to \XX$ is birational \cite[Lemma 5]{BrO} that whenever $\Delta_{\chi,v}$ and $T_w$ (acting on $I(\chi)$) converge simultaneously, $\Delta_{^w\chi,v}^Y$ also converges. Therefore, if we substitute $\chi$ by ${^w\chi}$, the domain of integration now can be taken to be $\exp(\kappa+i\mathfrak a_{X,\RR}^*)$, where $\kappa\in \rho_{(X)}+\mathfrak a_{X,\CC}^*$ is deep in the domain of convergence of the integral for $\Delta_{\chi,v}$. The Eisenstein sum also converges in that region. 

We discuss how $\Delta_{\chi,v}^{\Tam}$ and $\Delta_{\chi,v}$ are related:
\begin{lemma} We have:
\begin{equation}
 \Delta^\Tam_{\chi,v} = Q_v \Delta_{\chi,v}.
\end{equation}
\end{lemma}

\begin{proof}
 As we did in the case of $\Delta_{\chi,v}$, we compose the morphism $\Delta_{\chi,v}^\Tam$ with the map $\mathcal S(U\backslash G)\to I(\chi)$ (integration against an $A$-eigenmeasure with $da(A_0)=1$) to get a morphism: $\mathcal S(U\backslash G)\to C^\infty(X)$. Notice that $da = (1-q^{-1})^{-\rk (A)} d^\Tam a$. Therefore, the integral expression for $\Delta^\Tam_{\chi,v}$ on $U\backslash G$ is an integral, over the open $A\times H$-orbit, against $(1-q^{-1})^{-\rk(A)}$ times a Tamagawa eigenmeasure. On the other hand, the morphisms $\Delta_{\chi,v}$ were defined using an $A\times H$-eigenfunction times the $G$-invariant measure $dx$ on $U\backslash G$ such that: $dx(U\backslash UK)=1$. Therefore, if $dx^\Tam $ denotes ``Tamagawa'' measure on $U\backslash G$ then we have:
$$\Delta^\Tam_{\chi,v} = (1-q^{-1})^{-\rk(A)} \frac{dx^\Tam }{dx} \Delta_{\chi,v}.$$

We compute: $dx^\Tam (U\backslash Uw_lJ)=dx^\Tam  (B_0) = (1-q^{-1})^{\rk(A)}$. Therefore:
$$\Delta^\Tam_{\chi,v} = \frac{1}{dx(U\backslash Uw_l J)} \Delta_{\chi,v} = Q_v \Delta_{\chi,v}.$$
\end{proof}

Therefore:
$$\Delta_{\chi,v}^{\Tam}(f^0_{\chi,v})= Q_v \Delta_{\chi,v}(f^0_{\chi,v}) =   \prod_{\check\alpha>0}\frac{1-q^{-1}e^{\check\alpha}}{1-e^{\check\alpha}}(\chi)\Omega_{\chi,v}(x_0) = $$ $$=  \prod_{\check\alpha>0}\frac{1-q^{-1}e^{\check\alpha}}{1-e^{\check\alpha}}(\chi) \cdot L_{\chi,v}^\frac{1}{2}$$

Recall that $L_{X,v}^\frac{1}{2}(\chi)=c_v\cdot \beta_v(\chi)$, where $c_v$ is a quotient of products of local values for the Dedekind zeta function of $F$ and $\beta_v$ is a quotient of products of Dirichlet $L$-values which depend on $\chi$. If we consider the product $\prod_{v\notin S} c_v$ it may not converge in general. However, we can make sense of it by considering the leading term of its Laurent expansion, when considered as a specialization of a product/quotient of translates of $\zeta^S$. We will denote this number by $\left({c^S}\right)^*$. The standard definition of Tamagawa measures implies that when computing the product over all $v\notin S$ in the expression above, we should use $\left({c^S}\right)^*$ wherever ${c^S}$ formally appears in the product. Similarly, we will denote: $\left(L_X^{\frac{1}{2},S}\right)^*=(c^S)^*\cdot \prod_{v\notin S}\beta_v(\chi)$.
Therefore we get:

\begin{theorem}\label{period}
 The period integral of:
\begin{equation}
 \sum_{\gamma \in {\BB}(F)\backslash \GG(F)} \Phi(\gamma g) = \int_{\exp(\kappa+i\mathfrak a_\RR^*)} E(f_\chi,g) d\chi
\end{equation}
over $\HH(F)\backslash\HH(\Ad)$ is equal to:
\begin{equation}\label{periodformula}\int_{\exp(\kappa+i\mathfrak a_{X,\RR}^*)} \left(L_X^{\frac{1}{2},S}(\chi)\right)^*   \sum_{\left[W/W_{(X)}\right]} \left(\tilde j_{w}^S(\chi) \prod_{v\in S} \Delta_{{^{w}\chi},v}^{Y,{\Tam}} (f_{{^{w}\chi},v})\right) d\chi
\end{equation}
plus terms which depend on the restriction of $f_\chi$, as a function of $\chi$, to a subvariety of smaller dimension. Here $[W/W_{(X)}]$ denotes a set of representatives of minimal length for $W/W_{(X)}$-cosets, $\kappa\in \rho_{(X)}+\mathfrak a_{X,\CC}^*$ is deep in the domain of convergence of $\Delta_\chi$, and $f_\chi$, the Mellin transform of $\Phi$ with respect to the normalized left $\AA(\Ad)/\AA(\Ad)^1$-action, is assumed to be factorizable with factors $f_{\chi,v}^0$ for $v\notin S$.
\end{theorem}

\begin{remarks}
 \begin{enumerate}
 \item It appears as if the above expression depends on the choice of representatives $w\in W(\YY)^{-1}$. However, the factors $c_{w}^S(\chi)$, for $\chi\in A_X^*$, do not depend on the choice of $w$. This is easily seen in the case of $\PPGL_2\backslash (\PPGL_2\times\PPGL_2)$, and the general case can be reduced to that by Brion's analysis of the weak order graph (cf.\ Proposition \ref{associates}).
 \item For the special case of the spherical variety $\HH\backslash (\HH\times\HH)$, where all $\BB$-orbits are of maximal rank and hence the above formula is precise, compare with the calculation of the scalar product of two pseudo-Eisenstein series in \cite[II.2.1]{MW}.
 \item One can continue along these lines and give a new argument relating the Tamagawa numbers of $\GG$ and $\HH$, which would specialize to the argument of \cite{La} in the group case.
\end{enumerate}
\end{remarks}

 What about period integrals of cusp forms? Based on the work of Waldspurger \cite{Wal}, Ichino and Ikeda \cite{II}, we formulate in \cite{SV} a conjecture which links the period integral of cusp forms to the local Plancherel formula. More precisely, the ``canonical'' global functional: ``$\HH$-period on $\pi$ times $\HH$-period on $\tilde\pi$'' (where $\pi$ denotes the space of a cuspidal automorphic representation and $\tilde\pi$ its dual) is conjectured to be related to the Plancherel measure on $\XX$ and hence, through the computation of \S \ref{secPlancherel}, to the quotient of $L$-values $L_X$.

\appendix

\section{Table of examples}

We give a table of spherical varieties $\XX$, the associated parabolics $\PP(\XX)$, the dual groups $\check G_X$ and the invariants $L_X$, expressed as quotients of unramified $L$-values for the group $\check G_X$. All of the examples below satisfy the assumptions of Theorem \ref{improvedformula}. The answer is expressed up to zeta-factors, which can be computed from Lemma \ref{omegatrivial}. For example, if the dual group of a variety is $\Sp_{2n}$ then the (local unramified) $L$-function associated to the second fundamental representation of $\Sp_{2n}$ is not distinguished in the table below from the exterior-square $L$-function. Notice that we multiply by the adjoint $L$-factor $L(\pi,\Adj,0)$ (for $\check G_X$); for the reader interested in global periods, the relevant quotient of $L$-values is the quotient:
$$ \frac{L(\pi,\Adj,0)}{L(\pi,\Adj,1)} \cdot L_X.$$
An $L$-function written as $L(\pi,s)$ implies the standard representation of $\check G_X$. If the dual group is a direct product, we write $\pi_1,\pi_2$ in order to indicate an $L$-function associated to its first, resp.\ second factor. Notice that in cases of multiplicity (e.g.\ in the examples of $P(\Sp_{2n}\times\Gm)\backslash P\Sp_{2n+2}$ and $\GL_n\backslash \SO_{2n+1}$ below) it is ambiguous how to consider a given distinguished unramified representation as a semisimple conjugacy class in $\check G_X$; in these cases, we do not know what the (conjectural) significance of $L_X$ is for period integrals of automorphic forms. The parabolic $\PP(\XX)$ is denoted by indicating the isomorphism class of its Levi quotient -- and in the case of $\GGL_n$ we indicate a partition of $n$ instead of the isomorphism class of the Levi. Finally, whenever the abelianization of the group $\HH/\HH\cap \mathcal Z(\GG)$ is non-trivial (in which case it will be, for all examples, isomorphic to $\Gm$), we write the $L$-value 
corresponding to the line bundle over $\XX$ defined by the character $|\eta|^s$ of $\HH$, where $\eta$ denotes a primitive character of $\HH/\HH\cap \mathcal Z(\GG)$ -- differently stated, for the $L$-value (only) we consider the $\Gm\times\GG$ variety $\XX_0=\HH_0\backslash \GG$, where $\HH_0$ is the kernel of all characters of $\HH$ which are trivial on $\HH\cap \mathcal Z(\GG)$.

Our source for these examples are tables 4 and 5 of \cite{KnVS}, where all ``primitive'' spherical pairs of Lie algebras $\mathfrak g, \mathfrak h$ with $\mathfrak g, \mathfrak h$ reductive appear (over the algebraic closure). Moving from the algebraic closure to $k$ and from Lie algebras to groups, we take all groups $\GG, \HH$ in the examples below to be split, and make convenient choices for the centers of the groups (in particular, so that the assumptions of Theorem \ref{improvedformula} are satisfied.) The examples included in the table represent all cases of table 4 of \emph{loc.cit.\ }which do not have spherical roots of type $N$ and where $\mathfrak g$ is classical, and also a few cases of table 5. The tables of \cite{BP} have also been useful in computing this table. We use regular, instead of boldface, font for the groups.

\begin{landscape}
\thispagestyle{empty}

\begin{equation*}
 \begin{array}{|c|c|c|c|c|}
\hline
& X & P(X) & \check G_X & L(\pi,\Adj,0) \cdot L_X \text{(up to }\zeta\text{-factors)}\\
 & & & & \\
\hline
1& \GL_m\times\GL_n\backslash\GL_{m+n}, m\ge n  &P_{1,1,\dots,m-n,1,\dots,1} &  \Sp_{2n} & L(\pi,\frac{1+m-n}{2}+s) L(\pi,\frac{1+m-n}{2}-s) L(\pi,\wedge^2,1) \text{ if }n>1\\
\hline
 & & & & L(\pi,\frac{1+m-n}{2}+s) L(\pi,\frac{1+m-n}{2}-s) \text{ if }n=1\\
\hline
2& \Sp_{2n}\backslash\GL_{2n} &P_{2,2,\dots,2} & \GL_n & L(\pi,\Adj,2)\\
\hline
3& \Sp_{2n}\backslash \GL_{2n+1} & B & \GL_n\times \GL_{n+1} & L(\pi_1\otimes\pi_2,1)L(\tilde\pi_1\otimes \tilde\pi_2,1) \\
\hline
4& \Sp_{2m}\times \Sp_{2n}\backslash \Sp_{2m+2n}, m\ge n & P_{\GL_2^n\times\Sp_{2m-2n}} & \Sp_{2n} & L(\pi,\frac{1}{2}) L(\pi,\frac{3}{2}) L(\pi,\wedge^2,1)  \text{ if }n>1 \\
\hline 
 & & & & L(\pi,\frac{1}{2}) L(\pi,\frac{3}{2}) \text{ if } n=1\\ 
\hline
5& \operatorname{GSp}_{2n}\backslash \operatorname{PSp}_{2n+2} &P_{\Gm^2\times \Sp_{2n-2}}& \Spin_4\simeq \SL_2\times\SL_2 & L(\pi_1,\frac{1}{2}+s)L(\pi_1,\frac{1}{2}-s) L(\pi_1\otimes\pi_2, n)\\
\hline
6& \GL_n\backslash \SO_{2n}, n\ge 4 & P_{(\GL_2)^{\lfloor \frac{n}{2}\rfloor}} & \Sp_{2\lfloor\frac{n}{2}\rfloor}& L(\pi, \wedge^2, 1) L(\pi,  \frac{1}{2}+s) L(\pi,  \frac{1}{2}-s) \text{ if }n\text{ is even}\\
\hline
 &&&& L(\pi, \wedge^2, 1) L(\pi,  \frac{3}{2}+s)L(\pi,  \frac{3}{2}-s) \text{ if }n \text{ is odd}\\
\hline
7& \GL_n\backslash \SO_{2n+1}, n\ge 2 & B & \Sp_{2\lceil\frac{n}{2}\rceil}\times \Sp_{2\lfloor\frac{n}{2}\rfloor}& L(\pi_1,\frac{1}{2}+s)L(\pi_1,\frac{1}{2}-s) L(\pi_1\otimes\pi_2,1) \\
\hline
8& G_2\backslash\SO_7 & P_{\GL_3} & \SL_2 & L(\pi,\Adj,3)\\
\hline
9& G_2\backslash \operatorname{PSO}_8 & P_{\Gm^2\times \GL_2} & \SL_2\times\SL_2\times\SL_2 & L(\pi_1\otimes\pi_2\otimes \pi_3, \frac{3}{2})\\
\hline
10 &\Spin_7\backslash \SO_9 & P_{\Gm\times\GL_3} & \SL_2\times\SL_2 & L(\pi_1\otimes\pi_2, 2) L(\pi_2,\Adj,3) \\
\hline
11 &\operatorname{GSpin}_7\backslash \operatorname{PSO}_{10} & P_{\Gm^3\times \GL_2} & \SL_4 \times \SL_2 & L(\pi_1, 4+s) L(\tilde\pi_1, -3-s) L(\wedge^2 \pi_1\otimes  \pi_2, \frac{3}{2} ) \\
\hline
12 & H\backslash H\times H & B & \check H^{\iota-\diag} & L(\pi,\Adj,1)\\
\hline
13 &\GL_n\backslash \GL_n\times\GL_{n+1} & B & \check G & L(\pi_1\otimes\pi_2,\frac{1}{2}+s) L(\tilde\pi_1\otimes\tilde\pi_2,\frac{1}{2}-s)\\
\hline
14 &\SO_n\backslash \SO_n\times \SO_{n+1} &B & \check G & L(\pi_1\otimes\pi_2,\frac{1}{2})\\
\hline
\end{array}
\end{equation*}

\begin{remark}
 The exponent $^{\iota-\diag}$ denotes the set of elements $(h,h^\iota)$ where $\iota$ is the Chevalley involution -- which takes Langlands parameters of a representation to Langlands parameters of its dual.
\end{remark}

\end{landscape}


\begin{thebibliography}{KKV89}

\bibitem[Ak83]{Ak} D.\ Akhiezer, \emph{Equivariant completions of homogeneous algebraic varieties by homogeneous divisors.}  Ann.\ Global Anal.\ Geom.\  1  (1983),  no.\ 1, 49--78.

\bibitem[BO07]{BO} Y.\ Benoist and  H.\ Oh \emph{Polar decomposition for $p$-adic symmetric spaces.} Int.\ Math.\ Res.\ Not.\  2007,  no.\ 24, Art.\ ID rnm121, 20 pp.

\bibitem[Be88]{BePl} J.\ Bernstein, \emph{On the support of Plancherel measure.} JGP - Vol.5, n.\ 4, 1988, 663--710.

\bibitem[BK98]{BK} A.\ Braverman and D.\ Kazhdan, \emph{On the Schwartz space of the basic affine space},  Selecta Math.\ (N.S.)  5  (1999),  no.\ 1, 1--28.

\bibitem[BK02]{BK2} A.\ Braverman and D.\ Kazhdan, \emph{Normalized intertwining operators and nilpotent elements in the Langlands dual group.} Mosc. Math. J.  2  (2002),  no. 3, 533--553.

\bibitem[BP05]{BP}  P.\ Bravi and G.\ Pezzini, \emph{Wonderful varieties of type $D$.} Represent.\ Theory 9 (2005), 578--637.

\bibitem[Br01]{BrO} M.\ Brion, \emph{On orbit closures of spherical subgroups in flag varieties}, Comment.\ Math.\ Helv.\ 76 (2001) 263–299.

\bibitem[BLV86]{BLV} M.\ Brion, D.\ Luna, and Th.\ Vust, \emph{Espaces homog\`enes sph\'eriques.} Invent.\ Math. 84 (1986), no.\ 3, 617--632.

\bibitem[Ca80]{C}
W.\ Casselman, \emph{The unramified principal series of $p$-adic groups. I. The spherical function.}  Compositio Math.\  40  (1980), no.\ 3, 387--406.

\bibitem[CS80]{CS}
W.\ Casselman and J.\ Shalika, \emph{The unramified principal series of $p$-adic groups. II. The Whittaker function.}  Compositio Math.\  41  (1980), no.\ 2, 207--231.

\bibitem[DS11]{DS} P.\ Delorme and V.\ S\'echerre, \emph{An analogue of the Cartan decomposition for $p$-adic symmetric spaces of split $p$-adic reductive groups.} Pacific J.\ Math., 251(1):1--21, 2011.

\bibitem[Ga99]{Ga} P.\ Garrett, \emph{Euler factorization of global integrals.}  Automorphic forms, automorphic representations, and arithmetic (Fort Worth, TX, 1996),  35--101, Proc.\ Sympos.\ Pure Math., 66, Part 2, Amer.\ Math.\ Soc., Providence, RI, 1999.

\bibitem[GN10]{GN} D.\ Gaitsgory and D.\ Nadler, \emph{Spherical varieties and Langlands duality.} Mosc.\ Math.\ J.\ 10 (2010), no.\ 1, 65--137, 271.

\bibitem[GP92]{GP} B.\ H.\ Gross and D.\ Prasad, On the decomposition of a representation of $\SO_n$ when restricted to $\SO_{n−1}$ , Canad.\ J.\ Math.\ 44 (1992), 974–1002.

\bibitem[Hi99]{Hi1} Y.\ Hironaka, \emph{Spherical functions and local densities on Hermitian forms.}  J.\ Math.\ Soc.\ Japan  51  (1999),  no.\ 3, 553--581.

\bibitem[Hi05]{Hi2} Y.\ Hironaka, \emph{Spherical functions on ${\rm Sp}\sb 2$ as a spherical homogeneous ${\rm Sp}\sb 2\times({\rm Sp}\sb 1)\sp 2$-space.}  J.\ Number Theory  112  (2005),  no.\ 2, 238--286. 

\bibitem[HS88]{HS1} Y.\ Hironaka and F.\ Sato, \emph{Spherical functions and local densities of alternating forms.}  Amer.\ J.\ Math.\  110  (1988),  no.\ 3, 473--512.

\bibitem[HS00]{HS2} Y.\ Hironaka and F.\ Sato, \emph{Local densities of representations of quadratic forms over $p$-adic integers (the non-dyadic case).}  J.\ Number Theory  83  (2000),  no.\ 1, 106--136.

\bibitem[II10]{II} A.\ Ichino and T.\ Ikeda, \emph{On the periods of automorphic forms on special orthogonal groups and the Gross-Prasad conjecture.} Geom.\ Funct.\ Anal.\ 19 (2010), no.\ 5, 1378--1425.

\bibitem[JLR99]{JLR} H.\ Jacquet, E.\ Lapid and J.\ Rogawski, \emph{Periods of automorphic forms.} J.\ Amer.\ Math.\ Soc.\ 12 (1999), no.\ 1, 173--240.


\bibitem[KMS03]{KMS} S.-i.\ Kato, A.\ Murase and T.\ Sugano, \emph{Whittaker-Shintani functions for orthogonal groups.}  Tohoku Math.\ J.\ (2)  55  (2003),  no.\ 1, 1--64.

\bibitem[Kn91]{KnLV}  F.\ Knop, \emph{The Luna-Vust theory of spherical embeddings.}  Proceedings of the Hyderabad Conference on Algebraic Groups (Hyderabad, 1989),  225--249, Manoj Prakashan, Madras, 1991.\ 

\bibitem[Kn94]{KnAs} F.\ Knop, \emph{The asymptotic behavior of invariant collective motion.}  Invent.\ Math.\  116  (1994),  no.\ 1-3, 309--328.

\bibitem[Kn95]{KnOrbits}  F.\ Knop, \emph{On the set of orbits for a Borel subgroup.}  Comment.\ Math.\ Helv.\  70  (1995),  no.\ 2, 285--309.

\bibitem[Kn96]{KnAut} F.\ Knop, \emph{Automorphisms, root systems, and compactifications of homogeneous varieties.} J.\ Amer.\ Math.\ Soc.\ 9 (1996), no.\ 1, 153--174. 

\bibitem[KV06]{KnVS} F.\ Knop and B.\ Van Steirteghem, \emph{Classification of smooth affine spherical varieties.}  Transform.\ Groups  11  (2006),  no.\ 3, 495--516.

\bibitem[Ko88]{Ko}  R.E.\ Kottwitz, \emph{Tamagawa numbers.} Ann.\ of Math., 127  (1988)  pp.\ 629--646.

\bibitem[Lan66]{La} R.P.\ Langlands, \emph{The Volume of the Fundamental Domain for Some Arithmetical Subgroups of Chevalley Groups.} Proc.\ Sym.\ Pure Math., 9, AMS, Providence, 1966, 143--148.

\bibitem[Lap06]{Lapid} E.\ Lapid, \emph{The relative trace formula and its applications.} Automorphic Forms and Automorphic $L$-Functions (Kyoto, 2005), RIMS Kokyuroku Vol.\ 1468 (2006) pp.\ 76--87 

\bibitem[Lu01]{Lu} D.\ Luna, \emph{Vari\'et\`es sph\'eriques de type $A$.} Publ.\ Math.\ Inst.\ Hautes Études Sci.\  no.\ 94  (2001), 161--226.

\bibitem[LV83]{LV} D.\ Luna, T.\ Vust, \emph{Plongements d'espaces homog\`enes.} Comment.\ Math.\ Helv.\ 58 (1983), no.\ 2, 186--245. 

\bibitem[Ma71]{MD71} I.G.\ Macdonald, \emph{Spherical Functions on a Group of $p$-adic Type.} Publications of the Ramanujan Institute, no.\ 2. Ramanujan Institute, Centre for Advanced Study in Mathematics,University of Madras, Madras, 1971.

\bibitem[Ma01]{MD} I.G.\ Macdonald, \emph{Orthogonal polynomials associated with root systems.}  S\'em.\ Lothar.\ Combin.\  45  (2000/01), Art.\ B45a, 40 pp.\ (electronic).


\bibitem[MW94]{MW} C.\ M\oe glin and J.-L.\ Waldspurger, \emph{D\'ecomposition spectrale et s\'eries d'Eisenstein. Une paraphrase de l'\'Ecriture.} Progress in Mathematics, 113. Birkh\"auser Verlag, Basel, 1994.

\bibitem[Of04]{Of} O.\ Offen, \emph{Relative spherical functions on $p$-adic symmetric spaces (three cases).}  Pacific J.\ Math.\  215  (2004),  no.\ 1, 97--149.

\bibitem[Sa06]{Sa1} Y.\ Sakellaridis, \emph{A Casselman-Shalika formula for the Shalika model of $\GL_n$.} Canad.\ J.\ Math.\ 58 (2006), no.\ 5, 1095-–1120.

\bibitem[Sa08]{Sa2} Y.\ Sakellaridis, \emph{On the unramified spectrum of spherical varieties over $p$-adic fields.} Compositio Math.\ 144 (2008), no.\ 4, 978--1016.

\bibitem[Sa12]{Sa3} Y.\ Sakellaridis, \emph{Spherical varieties and integral representations of $L$-functions.} Algebra \& Number Theory, 6(4):611--667, 2012.

\bibitem[SV]{SV} Y.\ Sakellaridis and A.\ Venkatesh, \emph{Periods and harmonic analysis on spherical varieties.} Preprint, arXiv:1203.0039.

\bibitem[Wal85]{Wal} J.-L.\ Waldspurger, \emph{Sur les valeurs de certaines fonctions $L$ automorphes en leur centre de sym\'etrie.} Compositio Math.\ 54 (1985), 173--242.

\bibitem[Was96]{Wa} B.\ Wasserman, \emph{Wonderful varieties of rank two.} Transform.\ Groups 1 (1996), no.\ 4, 375--403. 

\end{thebibliography}
\end{document}